\mag=\magstep1
\documentclass{amsart}
\usepackage{amsmath,amssymb,amsthm}
\usepackage{amscd,fancyhdr}
\makeatletter
\@addtoreset{equation}{section}

\makeatother
\makeatletter
\renewcommand\section{\@startsection {section}{1}{\z@}%
             {-3.5ex \@plus -1ex \@minus -.2ex}%
             {2.3ex \@plus.2ex}%
             {\normalfont\large\bfseries}}
\renewcommand\subsection{\@startsection{subsection}{2}{\z@}%
             {-3.25ex\@plus -1ex \@minus -.2ex}%
             {1.5ex \@plus .2ex}%
             {\normalfont\normalsize\bfseries}}
\renewcommand\subsubsection{\@startsection{subsubsection}{3}{\z@}%
             {-3.25ex\@plus -1ex \@minus -.2ex}%
             {1.5ex \@plus .2ex}%
             {\normalfont\normalsize\bfseries}}
\makeatother
\newcommand{\OP}[1]{\operatorname{#1}}
\newcommand{\BF}[2]{\langle #1,#2 \rangle} % bilinear form
\newcommand{\LF}[1]{\langle #1 \rangle}
\newcommand{\MC}[1]{\mathcal{#1}}
\newcommand{\MF}[1]{\mathfrak{#1}}
\newcommand{\Z}{\mathbb{Z}}
\newcommand{\Q}{\mathbb{Q}}
\newcommand{\R}{\mathbb{R}}
\newcommand{\C}{\mathbb{C}}
\renewcommand{\P}{\mathbb{P}}
\def\rank{\operatorname{rank}}
\newcommand{\disc}{\operatorname{disc}}
\newcommand{\ord}{\operatorname{ord}}
\newcommand{\sign}{\operatorname{sign}}
\newcommand{\CONG}{{\displaystyle\mathop{\rightarrow}^\sim}}

\newcommand{\Gsymp}{\mathfrak{G}_{K3}^\text{\rm symp}}

\newcommand{\clos}{\text{\rm clos}}
\title{Finite Symplectic Actions on the $K3$ Lattice}
\author{Kenji Hashimoto}
\thanks{This work was partially supported by %
 Grant-in-Aid for JSPS Fellows No.\ 20-56181.}
\address{Korea Institute for Advanced Study, %
 207-43 Cheongnyangni-2dong, Dongdaemun-gu,
 Seoul, 130-722, Korea}
\email{hashimoto@kias.re.kr}
\begin{document}
\theoremstyle{plain}
\newtheorem{thm}{Theorem}[section]
\newtheorem*{thmn}{Theorem}
\newtheorem{thmdash}[thm]{Theorem$'$}
\newtheorem*{mainthm}{Main Theorem}
\newtheorem{lem}[thm]{Lemma}
\newtheorem{cor}[thm]{Corollary}
\newtheorem{prop}[thm]{Proposition}
\newtheorem{clm}[thm]{Claim}
\newtheorem{ex}[thm]{Example}
\newtheorem{nota}[thm]{Notation}
\theoremstyle{definition}
\newtheorem{defn}[thm]{Definition}
\newtheorem{rem}[thm]{Remark}

\begin{abstract}
In this paper,
 we study finite symplectic actions on $K3$ surfaces $X$,
 i.e.\ actions of finite groups $G$ on $X$
 which act on $H^{2,0}(X)$ trivially.
We show that
 the action on the $K3$ lattice $H^2(X,\Z)$
 induced by a symplectic action of $G$ on $X$
 depends only on $G$ up to isomorphism,
 except for five groups.
\end{abstract}
\maketitle
\setcounter{tocdepth}{1}
%\tableofcontents
\setcounter{section}{-1}
\section{Introduction} \label{sect:intro}

A compact complex surface $X$ is called a $K3$ surface
 if it is simply connected and has a nowhere vanishing
 holomorphic $2$-form $\omega_X$.
For properties on $K3$ surfaces, see \cite{bhpv}.
An automorphism $g$ of $X$ is said to be symplectic
 if $g^* \omega_X=\omega_X$.
Nikulin \cite{nikulin76,nikulin79fin} studied
 symplectic actions of finite groups
 on $K3$ surfaces.
In particular, he showed the following result:

\begin{thm}[\cite{nikulin79fin}]
\label{thm:nikulinab}
There exist exactly $14$ finite abelian groups $G$
 ($G=C_2,C_3,\ldots$)
 which act on $K3$ surfaces
 faithfully and symplectically.
Moreover, for each $G$,
 the action of $G$ on the $K3$ lattice
 induced by a faithful and symplectic action of $G$
 on a $K3$ surface
 is unique up to isomorphism.
\end{thm}

In this paper, we prove that
 the above uniqueness
 holds for any finite groups
 except for five groups
 (see Theorem \ref{thm:main}).
We use the same notations for groups as in \cite{xiao96}
 (see Table \ref{subsect:qands}).

\begin{mainthm}
Let $G$ be a finite group
 such that
 $G\neq Q_8,T_{24},\MF{S}_5,L_2(7),\MF{A}_6$.
Then the action of $G$ on the $K3$ lattice
 induced by a faithful and symplectic action of $G$
 on a $K3$ surface
 is unique up to isomorphism.
More precisely,
 if $G_i\cong G$ acts on a $K3$ surface $X_i$
 faithfully and symplectically
 ($i=1,2$), then
 there exists an isomorphism
 $\alpha:H^2(X_1,\Z)\rightarrow H^2(X_2,\Z)$
 preserving the intersection forms such that
 $\alpha\circ G_1 \circ\alpha^{-1}=G_2$
 in $\OP{GL}(H^2(X_2,\Z))$.
\end{mainthm}

As a corollary, 
 we have the following
 by a similar argument in \cite{nikulin79fin}
 (see \cite{whitcher09} for a detailed argument).

\begin{cor} \label{cor:deform}
Let $G$ be a finite group
 which does not belong to the exceptional cases
 listed above.
If $G$ acts on $K3$ surfaces $X_i$
 faithfylly and symplectically
 for $i=1,2$,
 then there exists a connected family $\MC{X}$ of
 $K3$ surfaces with an action of $G$
 which satisfies the following conditions:
\begin{enumerate}
\renewcommand{\labelenumi}{(\arabic{enumi})}
\item
$X_1,X_2$ are fibers of $\MC{X}$;
\item
the restriction of the action of $G$ on $\MC{X}$
 to the fiber $X_i$ coincides with the given one
 ($i=1,2$);
\item
the action of $G$ on each fiber of $\MC{X}$ is symplectic.
\end{enumerate}
\end{cor}

If two $K3$ surfaces $X_1$ and $X_2$ 
 with actions of $G$ satisfy
 the conclusions of Corollary \ref{cor:deform},
 $X_1$ and $X_2$ are said to be
 $G$-deformable.

We recall known results on finite symplectic actions
 on $K3$ surfaces.
After a result of Nikulin \cite{nikulin79fin},
 Mukai \cite{mukai88} classified
 finite groups
 which act on $K3$ surfaces faithfully and symplectically
 by listing the eleven maximal groups
 (see Theorem \ref{thm:mukai}).
Xiao \cite{xiao96} gave another proof of Mukai's result
 by studying the singularities of
 the quotient $G\backslash X$
 for a $K3$ surface $X$ with a symplectic action
 of a finite group $G$.
Moreover, he showed the following:

\begin{thm}[\cite{xiao96}] \label{thm:xiao}
Let $G$ be a finite group.
Suppose that $G\neq Q_8,T_{24}$.
Then,
 for any $K3$ surface $X$
 with a faithful and symplectic action of $G$,
 the quotient $G\backslash X$ has
 the same $A$-$D$-$E$-configuration of the singularities.
\end{thm}

Considering his result, one may expect
 that the uniqueness as in
 Theorem \ref{thm:nikulinab} holds for
 most of non-abelian finite groups as well.
This paper is motivated by this expectation.
We follow Kond\=o's approach \cite{kondo98}
 with which he gave another proof of Mukai's result
 (see Remark \ref{rem:history} for Nikulin's contribution
 to this another proof).
He embeds the coinvariant lattice
 $H^2(X,\Z)_G=(H^2(X,\Z)^G)^\bot$
 into a Niemeier lattice $N$,
 and describes a symplectic action as
 an action on $N$.
Here a Niemeier lattice is a negative definite
 even unimodular lattice of rank $24$
 which is not isomorphic to the Leech lattice.
By looking this action more carefully, we prove
 Main Theorem.
For some finite groups, the uniqueness of
 their (symplectic) actions on $K3$ surfaces were studied
 by several authers
 \cite{kondo99,oguisozhang02,keumoguisozhang05,%
 oguiso05,zhang06,whitcher09}.
In the case where $G$ is abelian,
 Garbagnati and Sarti
 \cite{garbagnatisarti07,garbagnatisarti09},
 and Garbagnati \cite{garbagnati10}
 computed the structure of
 $H^2(X,\Z)^G$ and $H^2(X,\Z)_G$,
 and corrected errors of computations of discriminant groups
 of $H^2(X,\Z)_G$ in \cite{nikulin79fin}.
We use computer algebra systems
 GAP \cite{GAP} and Maxima \cite{maxima}
 for the computations of permutation groups and lattices.

The paper proceeds as follows.
In Section \ref{sect:lattice},
 we recall basic facts on lattices,
 which are used through the paper.
We recall results on finite symplectic actions
 on $K3$ surfaces in Section \ref{sect:symplectic}.
Using these results,
 we take a lattice theoretic approach
 to study finite symplectic actions on $K3$ surfaces.
We introduce the notion of ``finite symplectic actions
 on the $K3$ lattice $\Lambda$,''
 taking account of Nikuin's characterization of
 symplectic actions on $K3$ surfaces
 (see Definition \ref{defn:sympgrp} and
 Proposition \ref{prop:nikulinsymp}).
The set of finite symplectic actions
 $G\subset\OP{O}(\Lambda)$ on $\Lambda$
 is denoted by $\MC{L}$.
%For a symplectic action of a finite group $G$
% on the $K3$ lattice $\Lambda$,
For $G\in\MC{L}$,
 there exists a $K3$ surface $X$ with
 a symplectic action of $G$ such that we have
 a $G$-equivalent isomorphism $\Lambda\cong H^2(X,\Z)$.
Section \ref{sect:embedding}
 is the key of the paper.
%Suppose that a finite group $G$ acts
% on the $K3$ lattice $\Lambda$ symplectically.
By Kond\=o's lemma (see Lemma \ref{lem:kondoembed}),
 the coinvariant lattice $\Lambda_G$ for $G\in\MC{L}$
 is embedded
 into a Niemeier lattice $N$ primitively.
Since the action of $G$ on $\Lambda_G$ is extended to
 that on $N$ such that $N_G=\Lambda_G$,
 we can study $G$ as
 an automorphism group of $N$.
Applying the classification of Niemeier lattices,
 we classify the primitive embeddings of
 $\Lambda_G$ into Niemeier lattices.
To prove Main Theorem,
 we first prove
 the uniqueness of $\Lambda_G$ and $\Lambda^G$.
In Sections \ref{sect:coinvariantuniq}
 and \ref{sect:invariantuniq},
 we show the uniqueness of $\Lambda_G$ and $\Lambda^G$
 respectively,
 by using the result in Section \ref{sect:embedding}.
Next, we show the uniqueness of the glueing data of
 $\Lambda^G$ and $\Lambda_G$ to $\Lambda$.
In Sections \ref{sect:coinvariantsurj}
 and \ref{sect:invariantsurj},
 we show that either
 $\overline{\OP{O}(\Lambda_G)}=\OP{O}(q(\Lambda_G))$ or
 $\overline{\OP{O}(\Lambda^G)}=\OP{O}(q(\Lambda^G))$
 holds for any $G\in\MC{L}$.
This implies the uniqueness of the glueing data.
Finally, in Section \ref{sect:uniqueness},
 we prove Main Theorem by using the results
 in the previous sections.
Some applications of Main Theorem are given
 in Section \ref{sect:applications}.
We give the list of Niemeier lattices and
 the results of the computations
 in Section \ref{sect:tables}.

\subsubsection*{Acknowledgement}
The author would like to express his thanks to
 the former supervisor Professor Tomohide Terasoma
 for useful comments on a manuscript.
He also wishes to thank
 A. Garbagnati, S. Kond\=o, S. Mukai, V. Nikulin
 and U. Whitcher
 for comments and encouragements.
He would like to thank G. H\"ohn
 for kindly pointing out some errors
 in tables in Section \ref{sect:tables}.
Finally, he deeply thanks the referees for the helpful comments.

\section{Basic facts on lattices}
\label{sect:lattice}

\subsection{Definitions}

A lattice $L=(L,\BF{~}{~})$ is
 a free $\Z$-module $L$ of finite rank
 equipped with an integral symmetric bilinear form
 $\BF{~}{~}$.
We identify a lattice $L$ with its Gramian matrix
 $(\BF{v_i}{v_j})$ under an integral basis $(v_i)$ of $L$.
The discriminant $\disc(L)$ of $L$ is defined as
 the determinant of the Gramian matrix of $L$.
If $\disc(L)\neq 0$ (resp.\ $=\pm 1$),
 a lattice $L$ is said to be
 non-degenerate (resp.\ unimodular).
Let $t_{(+)}$ (resp.\ $t_{(-)}$) be the number of
 positive (resp.\ negative) eigenvalues of the Gramian
 matrix of $L$.
We call $(t_{(+)},t_{(-)})$ the signature of $L$
 and write
\begin{equation}
 \sign L=(t_{(+)},t_{(-)}).
\end{equation}
If $\BF{v}{v}\equiv 0 \bmod 2$ for all $v\in L$,
 a lattice $L$ is said to be even.
We denote by $L(\lambda)$ the $\Z$-module $L$ equipped with
 $\lambda$ times the bilinear form $\BF{~}{~}$,
 i.e.\ $(L,\lambda\BF{~}{~})$.
A sublattice $K$ of $L$ is said to be primitive
 if $L/K$ is torsion-free.
An automorphism of $L$ is defined
 as a $\Z$-automorphism of $L$
 preserving the bilinear form $\BF{~}{~}$.
We denote by $\OP{O}(L)$ the group of automorphisms of $L$.
For a subset $S\subset L$,
 we define a subgroup $\OP{O}(L,S)$ of $\OP{O}(L)$ by
\begin{equation} \label{osubgroup}
 \OP{O}(L,S)=\{ g\in\OP{O}(L) \bigm|
 g\cdot S=S \}.
\end{equation}

\begin{defn} \label{defn:latticewithaction}
A lattice $L$ with an action of $G$ is called a $G$-lattice
 if $G$ is a subgroup of $\OP{O}(L)$
 and is denoted by $(G,L)$.
For a $G$-lattice $(G,L)$,
 we define the invariant lattice $L^G$
 and the coinvariant lattice $L_G$ by
\begin{equation}
L^G=\{v\in L\bigm|g\cdot v=v ~ (\forall g\in G) \},\quad
L_G=(L^G)^\bot_L.
\end{equation}
A $G$-lattice $(G,L)$ and a $G'$-lattice $(G',L')$
 are said to be isomorphic
 if there exists an isomorphism
 $\alpha:L\rightarrow L'$ such that
\begin{equation}
 \alpha\circ G \circ \alpha^{-1}=G'.
\end{equation}
\end{defn}

We recall some basic properties on
 discriminant forms of lattices
for the reader's convenience.
See \cite{nikulin79int} for details.
Let $L$ be a non-degenerate even lattice.
The discriminant group $A(L)$ is a finite abelian group
 defined by
\begin{equation} \label{def_dual}
 A(L)=L^\vee/L,\quad
 L^\vee=\{ v\in L\otimes\Q \bigm|
 \BF{v}{L}\subset\Z \}.
\end{equation}
Here we extend the bilinear form $\BF{~}{~}$ on $L$
 to that on $L\otimes\Q$ linearly.
We have
\begin{equation}
 \left| A(L) \right|=\left| \disc(L) \right|.
\end{equation}
The discriminant form $q(L)$ of $L$
 is defined by
\begin{equation}
 q(L):A(L)\rightarrow\Q/2\Z;~
 x \bmod L \mapsto \BF{x}{x} \bmod 2\Z.
\end{equation}
We simply write $q(L)$ instead of $(A(L),q(L))$.
%For example, we write $q(L) \cong q(L')$.
%$q(L)$ is a non-degenerate quadratic form on
% a finite abelian group $A(L)$.
For a prime number $p$,
 let $A(L)_p$ and $q(L)_p$ denote
 the $p$-components of $A(L)$ and $q(L)$, respectively.
We have
\begin{equation}
 A(L)=\bigoplus_{p} A(L)_p, \quad
 q(L)=\bigoplus_{p} q(L)_p.
\end{equation}
We can consider $q(L)_p$ as
 the discriminant form of $L\otimes\Z_p$.
(The discriminant group and form
 for a non-degenerate even lattice over $\Z_p$
 are similarly defined.
Note that any lattice over $\Z_p$
 is even if $p\neq 2$.)
An automorphism of $q(L)$ is defined as
 an automorphism of a finite abelian group $A(L)$
 preserving $q(L)$.
We denote the group of automorphisms of $q(L)$
 by $\OP{O}(q(L))$.
An automorphism $\varphi\in\OP{O}(L)$
 induces an automorphism
 $\overline{\varphi}\in\OP{O}(q(L))$.
This correspondence gives the natural homomorphism
\begin{equation}
 \OP{O}(L)\rightarrow\OP{O}(q(L)).
\end{equation}
We define
\begin{equation}
 \OP{O}_0(L)=
 \OP{Ker}\bigl(\OP{O}(L)\rightarrow\OP{O}(q(L))\bigr)
\end{equation}
 and
\begin{equation} \label{image_in_oq}
 \overline{\OP{O}(L)}=
 \OP{Im}\bigl(\OP{O}(L)\rightarrow\OP{O}(q(L))\bigr).
\end{equation}

\subsection{Facts}

We use the following facts.
For details, see \cite{nikulin79int}.

\begin{lem}[\cite{nikulin79int}] \label{lem:unimodembed}
Let $L_1,L_2$ be non-degenerate even lattices.
We define
\begin{equation}
 \OP{Isom}(q(L_1),-q(L_2))=
 \{ \gamma:q(L_1)\CONG q(L_2) \}.
\end{equation}
% the set of isomorphisms between $q(L_1)$ and $q(L_2)$.
If $\gamma\in\OP{Isom}(q(L_1),-q(L_2))$,
 the lattice $\Gamma_\gamma$ defined by
\begin{equation}
\Gamma_\gamma=\{
 x\oplus y \in L_1^\vee \oplus L_2^\vee \bigm|
% x\in L_1^\vee, ~ 
 \gamma(x\bmod L_1)=y\bmod L_2
\}
\end{equation}
is an even unimodular lattice which contains
 $L_1$ and $L_2$ primitively.
This correspondence gives a one-to-one correspondence
 between $\OP{Isom}(q(L_1),-q(L_2))$ and
 the set of even unimodular lattices
 $\Gamma\subset L_1^\vee\oplus L_2^\vee$
 which contain $L_1$ and $L_2$ primitively.
Moreover, let $\gamma'\in\OP{Isom}(q(L_1),-q(L_2))$ and
 $\varphi_i\in\OP{O}(L_i)$.
Then,
 $\varphi_1\oplus\varphi_2
 \in\OP{O}(\OP{L_1}\oplus\OP{L_2})$
 is extended to an isomorphism
 $\Gamma_\gamma\rightarrow\Gamma_{\gamma'}$
 if and only if
 $\gamma'\circ\overline{\varphi}_1\circ\gamma^{-1}
 =\overline{\varphi}_2$
 in $\OP{O}(q(L_2))$.
\end{lem}

\begin{lem} \label{lem:unimodaction}
Let $\Gamma$ be a non-degenerate even lattice
 and $L$ a non-degenerate primitive sublattice
 of $\Gamma$.
\begin{enumerate}
\renewcommand{\labelenumi}{(\arabic{enumi})}
\item
If $g\in\OP{O}_0(L)$, the action of $g$ on $L$
 is extended to that on $\Gamma$
 whose restriction to $(L)^\bot_\Gamma$ is trivial.
\item
Suppose that $\Gamma$ is unimodular.
If $G$ is a subgroup of $\OP{O}(\Gamma,L)$ and
 the action of $G$ on $(L)^\bot_\Gamma$ is trivial,
 then the induced action of $G$ on $A(L)$
 is trivial.
\item
Suppose that $\Gamma$ is unimodular.
If a group $G$ acts on $\Gamma$ and $\Gamma_G$ is
 non-degenerate,
 then the induced action of $G$ on $A(\Gamma_G)$
 is trivial.
\end{enumerate}
\end{lem}

%To determine the discriminant form of a lattice,
% it is convenient to localize it, i.e.,
% consider it over $\Z_p$.
A lattice over $\Z_p$ is defined as
 a free $\Z_p$-module of finite rank
 with a $\Z_p$-valued symmetric bilinear form $\BF{~}{~}$.
First we consider the case $p\neq 2$.
In this case, any lattice can be diagonalized
 over $\Z_p$.

\begin{prop}[cf. \cite{cassels78,nikulin79int,SP}]
\label{prop:localodd}
Let $p$ be an odd prime and
 $\varepsilon_p\in\Z_p^\times$ a non-square
 $p$-adic unit.
If $L^{(p)}$ is a non-degenerate lattice over $\Z_p$,
\begin{equation}
 L^{(p)}\cong \bigoplus_{k\geq 0} \left(
 \LF{p^k}^{\oplus n_k}\oplus
 \LF{\varepsilon_p p^k}^{\oplus m_k}
 \right),
\end{equation}
 where $n_k\geq 0$ and $m_k\in\{0,1\}$
 are uniquely determined.
Hence
\begin{equation} \label{discformdecomp}
 q(L^{(p)})\cong \bigoplus_{k\geq 1} \left(
 q^{(p)}_+(p^k)^{\oplus n_k}
 \oplus q^{(p)}_-(p^k)^{\oplus m_k}
 \right),
\end{equation}
 where
\begin{align}
 q^{(p)}_+(p^k)&=\LF{1/p^k}\text{ on }
 \Z/p^k\Z, \\
 q^{(p)}_-(p^k)&=\LF{\varepsilon_p/p^k}\text{ on }
 \Z/p^k\Z.
\end{align}
In (\ref{discformdecomp}), the $n_k$ and $m_k$ are
 also uniquely determined.
\end{prop}

Let $L$ be a non-degenerate lattice.
We can determine $q(L)_p$ as follows.
Let $\Z_{(p)}$ be a localization of $\Z$
 by the prime ideal $(p)$,
 which is considered as a subring of $\Z_p$.
Then $L$ can be diagonalized over $\Z_{(p)}$.
%This is similar to the determination of
% the elementary divisors of integral matrices.
Then we can write
\begin{equation} \label{localdecomp}
 L\cong\bigoplus_{k\geq 0} L^{(p)}_k(p^k)
\end{equation}
 over $\Z_{(p)}$,
 where $L^{(p)}_k$ are lattices over $\Z_{(p)}$
 such that $L^{(p)}_k=0$
 or $\disc(L^{(p)}_k)\in
 \Z_{(p)}^\times/(\Z_{(p)}^\times)^2$.
(The discriminant of a lattice over a ring $R$
 is defined modulo $(R^\times)^2$.)
The $n_k$ and $m_k$ for $L\otimes\Z_p$
 in the above proposition
 are determined by
\begin{equation}
 (n_k,m_k)=
\begin{cases}
 (0,0) & \text{if } L^{(p)}_k=0, \\
 (\rank L^{(p)}_k,0) & \text{if }
 \disc(L^{(p)}_k) \in
 (\Z_p^{\times})^2/(\Z_{(p)}^\times)^2, \\
 (\rank L^{(p)}_k-1,1) & \text{otherwise.}
\end{cases}
\end{equation}

Next we consider the more complicated case $p=2$.

\begin{prop}[cf. \cite{cassels78,nikulin79int,SP}]
\label{prop:localeven}
Let $L^{(2)}$ be a non-degenerate lattice over $\Z_2$.
Then $L^{(2)}$ can be written as
 an orthogonal sum of
 the following lattices:
\begin{equation} \label{2adicbasis}
 \LF{\varepsilon 2^k},
 \begin{pmatrix} 0 & 2^k \\ 2^k & 0 \end{pmatrix},
 \begin{pmatrix} 2^{k+1} & 2^k \\ 2^k & 2^{k+1}
 \end{pmatrix},
\end{equation}
 where $k\geq 0$ and
 $\varepsilon\in\{1,3,5,7\}$.
Hence, if $L^{(2)}$ is even,
 $q(L^{(2)})$ can be written as an orthogonal
 sum of the following:
\begin{align}
 q^{(2)}_\varepsilon(2^k)&=\LF{\varepsilon/2^k}
 \text{ on } \Z/2^k\Z, \label{q2} \\
 u^{(2)}(2^k)&=
 \begin{pmatrix} 0 & 1/2^k \\ 1/2^k & 0
 \end{pmatrix} \text{ on }
 (\Z/2^k\Z)^{\oplus 2}, \label{u2} \\
 v^{(2)}(2^k)&=
 \begin{pmatrix} 1/2^{k-1} & 1/2^k \\ 1/2^k & 1/2^{k-1}
 \end{pmatrix} \text{ on } (\Z/2^k\Z)^{\oplus 2}.
 \label{v2}
\end{align}
\end{prop}

In the case $p=2$, the uniqueness as in
 Proposition \ref{prop:localodd} does not hold.
Although there is a complete system of invariants of
 a non-degenerate lattice over $\Z_2$
 (see \cite{SP}),
 we only recall the unimodular case.

\begin{prop}[cf.\ \cite{SP}]
\label{prop:2adicunimodular}
For a non-degenerate lattice $L^{(2)}$ over $\Z_2$
 with $\disc(L^{(2)})\in\Z_2^\times$,
 a quadruple $(r,d,t,e)$ defined as follows
 is a complete system of invariants
 of $L^{(2)}$.
If
\begin{equation} \label{2adicdecomp}
 L^{(2)}\cong
 \bigoplus_i \LF{\varepsilon_i} \oplus
 \begin{pmatrix} 0 & 1 \\ 1 & 0
 \end{pmatrix}^{\oplus n} \oplus
 \begin{pmatrix} 2 & 1 \\ 1 & 2
 \end{pmatrix}^{\oplus m},
\end{equation}
 the invariants $r,d,t,e$ are defined by
\begin{align}
 r&=\rank L^{(2)}, \\
 d&=\begin{cases}
 +1 & \text{if }
 \disc(L^{(2)}) \in \pm(\Z_2^\times)^2/(\Z_2^\times)^2, \\
 -1 & \text{otherwise,}
 \end{cases} \\
 t&=\sum_i \varepsilon_i \bmod 8\Z_2 \in \Z_2/8\Z_2, \\
 e&=\begin{cases}
 \text{\rm I} & \text{if $L^{(2)}$ is odd,} \\
 \text{\rm II} & \text{otherwise}.
 \end{cases}
\end{align}
\end{prop}

For example, we can directly check that
\begin{equation}
 \LF{1}^{\oplus 3} \cong
 \begin{pmatrix} 2 & 1 \\ 1 & 2 \end{pmatrix}
 \oplus \LF{3}
\end{equation}
 over $\Z_2$.
We actually have
 $(r,d,t,e)=(3,+1,\overline{3},\text{I})$
 for both lattices.
Using Proposition \ref{prop:2adicunimodular},
 we can determine $q(L)_2$
 for a non-degenerate even lattice $L$
 similarly to the case $p\neq 2$.
We can find an orthogonal decomposition
\begin{equation} \label{localdecompeven}
 L\cong\bigoplus_{k\geq 0} L^{(2)}_k(2^k)
\end{equation}
 over $\Z_2$,
 where $L^{(2)}_k$ is of the form (\ref{2adicdecomp}).
Then we can write $q(L)_2$
 as the corresponding orthogonal sum of
 (\ref{q2})--(\ref{v2}).
For relations between (\ref{q2})--(\ref{v2}),
 see \cite{nikulin79int}.

For a finite abelian group $A$,
 let $l(A)$ denote the minimum number
 of generators of $A$.
Let $L$ be a non-degenerate even lattice.
Since $\rank L^\vee=\rank L$ (see (\ref{def_dual})),
 we have
\begin{equation} \label{ineq_discgrp}
 l(A(L)) \leq \rank L.
\end{equation}
The follwoing theorem is a reformulation of
 Eichler's result in a view-point of discriminant forms.

\begin{thm}[\cite{nikulin79int}] \label{thm:nikulinuniq}
Let $L,L'$ be
 indefinite (non-degenerate) even lattices of rank $\geq 3$.
Suppose that the following conditins are satisfied:
\begin{enumerate}
\renewcommand{\labelenumi}{(\arabic{enumi})}
\item \label{thm:nikulinuniq:odd}
For each $p\neq 2$,
 either $\rank L\geq l(A(L)_p) +2$, or
 $n_k+m_k\geq 2$ for some $k$
 in the orthogonal decomposition
 (\ref{discformdecomp}), i.e.,
\begin{equation}
 q(L)_p\cong q_p \oplus q^{(p)}_{\pm}(p^k)\oplus
 q^{(p)}_{\pm}(p^k)
\end{equation}
 for some $q_p$ and $k>0$.
\item \label{thm:nikulinuniq:even}
Either $\rank L\geq l(A(L)_p) +2$, or
\begin{equation}
 q(L)_2\cong q_2 \oplus q'_2
\end{equation}
 for some $q_2$ and $q'_2$, where $q'_2$ is
 one of the following:
\begin{align}
 & u^{(2)}(2^k), ~ k>0, \\
 & v^{(2)}(2^k), ~ k>0, \\
 & q^{(2)}_{\varepsilon_1}(2^k)\oplus
   q^{(2)}_{\varepsilon_2}(2^k)\oplus
   q^{(2)}_{\varepsilon_3}(2^{k'}), ~
   \varepsilon_i\in\Z_2^\times,k,k'>0,|k-k'|\leq 1.
\end{align}
\item
 $\sign L=\sign L'$ and $q(L)\cong q(L')$.
\end{enumerate}
Then $L$ is isomorphic to $L'$.
\end{thm}

We use the following facts
 in Section \ref{sect:invariantsurj}.

\begin{thm}[\cite{nikulin79int}] \label{thm:nikulinsurj}
Let $L$ be an indefinite even lattice of rank $\geq 3$.
If the following conditins are satisfied,
 $\overline{\OP{O}(L)}=\OP{O}(q(L))$.
\begin{enumerate}
\renewcommand{\labelenumi}{(\arabic{enumi})}
\item \label{thm:nikulinsurj:1}
For each $p\neq 2$,
 $\rank L\geq l(A(L)_p) +2$.
\item \label{thm:nikulinsurj:2}
Either $\rank L\geq l(A(L)_p) +2$, or
\begin{equation}
 q(L)_2\cong
 q_2\oplus u^{(2)}(2) \text{~~or~~}
 q_2\oplus v^{(2)}(2)
\end{equation}
 for some $q_2$.
\end{enumerate}
\end{thm}

\begin{rem}
The conditions of Theorem \ref{thm:nikulinsurj}
 are stronger than those of Theorem \ref{thm:nikulinuniq}.
\end{rem}

\begin{thm}[\cite{nikulin79int}] \label{thm:localsurj}
If $L^{(p)}$ is a non-degenerate even lattice
 over $\Z_p$, we have
 $\overline{\OP{O}(L^{(p)})}=\OP{O}(q(L^{(p)}))$.
\end{thm}

\section{Finite symplectic actions on the $K3$ lattice $\Lambda$}
\label{sect:symplectic}

A compact complex surface $X$ is called a $K3$ surface
 if it is simply connected and has a nowhere vanishing
 holomorphic $2$-form $\omega_X$.

\begin{defn}
For a $K3$ surface $X$,
 an automorphism $g$ of $X$ is said to be symplectic
 if $g^*\omega_X=\omega_X$.
%We denote by $\OP{Aut}^0(X)$
% the group of symplectic automorphisms of $X$.
\end{defn}

We study faithful and symplectic actions of
 finite groups on $K3$ surfaces.

\begin{nota} \label{nota:grp}
%We identify
% abstract groups (notation: $\MF{G},\ldots$)
% which are isomorphic to each other.
%For a group $G$ acting on an object,
% the abstract group
% (forgetting its action) is denoted
% by $[G]$.
We use a Fraktur letter (e.g.\ $\MF{G}$) for
 an abstract group and use a roman letter (e.g.\ $G$)
 for a group with an action on an object
 (a lattice, a finite set,\ldots).
The abstract group structure of $G$ is denoted by $[G]$.
\end{nota}

\begin{defn}
We denote by $\MF{G}_{K3}^{\text{symp}}$
 the set of finite abstract groups
 $\MF{G}\neq 1$
 which can be realized as
 faithful and symplectic actions of groups
 on $K3$ surfaces.
\end{defn}

Mukai determined $\Gsymp$
 by listing the eleven maximal
 groups in $\Gsymp$.

\begin{thm}[\cite{mukai88}] \label{thm:mukai}
A finite abstract group $\MF{G}\neq 1$ is
 an element in $\Gsymp$
 if and only if $\MF{G}$ is
 a subgroup of the following eleven groups:
$$
 T_{48}, N_{72}, M_9, \MF{S}_5, L_2(7),
 H_{192}, T_{192}, \MF{A}_{4,4}, \MF{A}_6, F_{384},
 M_{20}.
$$
\end{thm}

There are exactly $79$ groups in $\Gsymp$.
See Table \ref{subsect:qands} for all elements in $\Gsymp$.
We use Xiao's notation \cite{xiao96}.

For a $K3$ surface $X$,
 the second integral cohomology group $H^2(X,\Z)$
 with its intersection form is isomorphic to
 the $K3$ lattice $\Lambda$ defined by
\begin{equation}
 \Lambda=
 \begin{pmatrix}0&1\\1&0\end{pmatrix}
 ^{\oplus 3}
 \oplus E_8(-1)^{\oplus 2},
\end{equation}
 which is the unique even unimodular lattice
 of signature $(3,19)$
 up to isomorphism
 (see Theorem \ref{thm:nikulinuniq}).
Here $E_8$ is the root lattice of
 type $E_8$.
The N\'eron--Severi group $\OP{NS}(X)$ of $X$
 is considered as a sublattice of $H^2(X,\Z)$.
If a group $G$ acts on $X$, the action of $G$
 induces a left action on $H^2(X,\Z)$ by
\begin{equation}
g \cdot v=(g^{-1})^*v,\quad
g\in G,v\in H^2(X,\Z).
\end{equation}
Note that if the action of $G$ is faithful,
 so is the induced action of $G$ on $H^2(X,\Z)$
 by the grobal Torelli theorem (see \cite{bhpv}).
Hence, if we take an isomorphism
 $\alpha:H^2(X,\Z)\rightarrow\Lambda$,
 the action of $G$ on $X$ induces a subgroup
 $\alpha\circ G \circ\alpha^{-1}
 \subset\OP{O}(\Lambda)$,
 which is isomorphic to $G$ as an abstract group.

We define the notion of ``finite symplectic actions
 on the $K3$ lattice.''

\begin{defn} \label{defn:sympgrp}
A finite subgroup $G\neq 1$ of
 $\OP{O}(\Lambda)$ is called
 a finite symplectic action
 on the $K3$ lattice $\Lambda$, if
 the following conditions are satisfied:
\begin{enumerate}
\renewcommand{\labelenumi}{(\arabic{enumi})}
\item
$\Lambda_{G}$ is negative definite;
\item
$\BF{v}{v}\neq -2$ for all $v\in \Lambda_{G}$.
\end{enumerate}
We denote the set of finite symplectic actions
 on the $K3$ lattices $\Lambda$ by $\MC{L}$.
Note that the finiteness of $G$ follows from
 the condition (1).
\end{defn}

Definition \ref{defn:sympgrp} is justified due to
 the following:

\begin{prop}[\cite{nikulin79fin}] \label{prop:nikulinsymp}
If a finite group $G$ acts
 on a $K3$ surface $X$ faithfully and symplectically,
 then $H^2(X,\Z)_G \subset \OP{NS}(X)$ and
 the induced subgroup of $\OP{O}(\Lambda)$
 is an element in $\MC{L}$.
Conversely, any element in $\MC{L}$
 is induced by a symplectic action
 of a finite group
 on a $K3$ surface.
\end{prop}

A $K3$ surface which admits a symplectic action
 of a finite group is characterized
 by coinvariant lattices $\Lambda_G$ of $G\in\MC{L}$.

\begin{prop}[\cite{nikulin79fin}] \label{prop:nikulinchara}
Let $\MF{G}\in\Gsymp$.
A $K3$ surface $X$ admits a symplectic action of $\MF{G}$
 if and only if
 there exists a primitive embedding
 $\Lambda_G \hookrightarrow \OP{NS}(X)$
 for some $G\in \MC{L}$ such that $[G]=\MF{G}$.
\end{prop}

Now we consider extensions of symplectic actions.

\begin{prop} \label{prop:extensionofaction}
Suppose that a finite group $G$ acts
 on a $K3$ surface $X$ faithfully and symplectically.
Then the action of $G$ on $X$ is
 extended to a faithful and symplectic action of
 $G':=\OP{O}_0(H^2(X,\Z)_G)$.
\end{prop}
\begin{proof}
[Proof (cf.\ \cite{nikulin79fin}).]
By Lemma \ref{lem:unimodaction}(1), the action of $G$ on
 $H^2(X,\Z)$ is extended to that of $G'$ such that
\begin{equation} \label{ggdash}
 H^2(X,\Z)^G=H^2(X,\Z)^{G'}.
\end{equation}
By the definition of a symplectic action,
 we have $\omega_X\in H^2(X,\C)^G$.
Since $G$ is a finite group, there exists
 a $G$-invariant K\"ahler $(1,1)$-form
 $\kappa\in H^2(X,\R)^G$.
By (\ref{ggdash}), the action of $G'$ also fixes
 $\omega_X$ and $\kappa$.
By the grobal Torelli theorem for $K3$ surfaces,
 the action of $G'$ on $H^2(X,\Z)$ is induced by
 that on $X$.
Since the action of $G'$ fixes $\omega_X$,
 the action of $G'$ on $X$ is symplectic.
\end{proof}

\begin{defn} \label{defn:closg}
For $G\in\MC{L}$, we define $\OP{Clos}(G)$ by
\begin{equation}
 \OP{Clos}(G)=\OP{O}_0(\Lambda_G).
\end{equation}
\end{defn}

By Lemma \ref{lem:unimodaction}(1),
 the action of $G$ on $\Lambda$ is extended to
 that of $\OP{Clos}(G)$ such that
 $\Lambda_G=\Lambda_{\OP{Clos}(G)}$,
 and $\OP{Clos}(G)$ is considered
 as an element in $\MC{L}$
 (see Definition \ref{defn:sympgrp}).
We define the subset $\MC{L}_\clos$ of $\MC{L}$ by
\begin{equation}
 \MC{L}_\clos=\{ G\in\MC{L} \bigm| \OP{Clos}(G)=G \}.
\end{equation}
By the following proposition,
 $\rank \Lambda_G$ depends only on
 the structure of $G$ as an abstract group.

\begin{prop}[\cite{nikulin79fin,mukai88}]
\label{prop:trace}
Let $g$ be an element in $O(\Lambda)$
 such that the group $\langle g \rangle$ generated by $g$
 is an element in $\MC{L}$.
Then $\ord(g)\leq 8$ and $\OP{Tr}(g;\Lambda)=\chi(g)-2$,
 where
\begin{equation}
 \chi(g)=24,8,6,4,4,2,3,2
 \text{~~if~~} \ord(g)=1,2,3,4,5,6,7,8.
\end{equation}
Hence, for $G\in\MC{L}$,
\begin{equation}
 \rank \Lambda_G=c(G):=24-
 \frac{1}{|G|} \sum_{g\in G} \chi(g).
\end{equation}
In particular, $c(G)=c(\OP{Clos}(G))$.
\end{prop}

\section{Embeddings of $\Lambda_G$ into Niemeier lattices}
\label{sect:embedding}

In this paper, a Niemeier lattice is a negative definite
 even unimodular lattice of rank $24$
 which is not isomorphic to the negative Leech lattice.
Here the negative Leech lattice is
 the unique negative definite even unimodular lattice of rank $24$
 which has no vector $v$ such that $\BF{v}{v}=-2$
 (cf.\ \cite{SP}).
In this section,
We study primitive embeddings
 of $\Lambda_G$ into Niemeier lattices.

\begin{defn} \label{defn:mcn}
Let $\MC{N}$ denote the set of isomorphism classes of
 $G$-lattices $(G,N)$
 which satisfy the following conditions:
\begin{enumerate}
\renewcommand{\labelenumi}{(\arabic{enumi})}
\item \label{defn:mcn:niemeier}
$G\neq 1$ and $N$ is a Niemeier lattice;
\item \label{defn:mcn:invariant}
there exists a vector $v\in N^{G}$ such that
 $\BF{v}{v}=-2$;
\item \label{defn:mcn:coinvariant}
%there exists a system of positive roots $\Delta^+$ of $N$
% which is stable
% under the action of $G$;
there exists no vector $v\in N_G$ such that
 $\BF{v}{v}=-2$;
\item \label{defn:mcn:embedding}
there exists a primitive embedding
 $N_G\hookrightarrow\Lambda$.
\end{enumerate}
\end{defn}

\begin{lem}[\cite{kondo98}] \label{lem:kondoembed}
For any $G\in\MC{L}$,
 $(G,\Lambda_G)\cong(G',N_{G'})$
% there exists a lattice with an action
 for some $(G',N)\in\MC{N}$.
% and isomorphisms $\iota:G\rightarrow G'$
% and $\alpha:\Lambda_G\rightarrow N_{G'}$
% such that
% $\alpha(g\cdot x)=\iota(g)\cdot \alpha(x)$
% for $g\in G$ and $x\in\Lambda_G$.
%
% $(G|_{\Lambda_G},\Lambda_G)\cong(G'|_{N_{G'}},N_{G'})$.
Conversely, if $(G',N)\in\MC{N}$, then there exists
 an element $G\in\MC{L}$ such that
% $(G|_{\Lambda_G},\Lambda_G)\cong(G'|_{N_{G'}},N_{G'})$.
 $(G,\Lambda_G)\cong(G',N_{G'})$.
\end{lem}

\begin{rem}
In the above lemma, we write $(G,\Lambda_G)$
 instead of $(G|_{\Lambda_G},\Lambda_G)$
 (cf.\ Definition \ref{defn:latticewithaction}).
We use the same notation in what follows.
\end{rem}

\begin{rem} \label{rem:history}
Lemma \ref{lem:kondoembed} is a direct consequence of
 Nikulin's work \cite{nikulin79fin,nikulin79int}.
Moreover, Nikulin pointed out that
 lattices such as $\Lambda_G$
 (``Leech-type'' lattices)
 can be classified
 by embedding them into even unimodular lattices
 (the latter part of Section 1.14 of \cite{nikulin79fin}).
\end{rem}

By Lemma \ref{lem:kondoembed},
 the study of $(G,\Lambda_G)$ for $G\in\MC{L}$
 is reduced to that of $\MC{N}$.
%\section{Study on $\MC{N}$}
In the following subsections, we present
 how to make a complete list of $\MC{N}$.
Some consequences from the list are given in
 Subsection \ref{subsect:conseq}.

\subsection{Some facts on Niemeier lattices}
\label{subsect:factniemeier}

The following theorem is standard.

\begin{thm}[cf.\ \cite{SP}]
There exist exactly 23 isomprphism classes
 of Niemeier lattices.
The isomorphism class of a Niemeier lattice $N$
 is determined by the root sublattice of $N$,
 whose type is given in Table \ref{subsect:niemeier}.
Here the root sublattice of $N$ is
 the sublattice generated by vectors $v\in N$
 such that $\BF{v}{v}=-2$.
\end{thm}

Let $N$ be a Niemeier lattice.
A vector $d\in N$ is called a root
 if $\BF{d}{d}=-2$.
Let $\Delta$ denote the set of roots of $N$.
A Weyl chamber $\MC{C}$ is a connected component of
$N\otimes\R - \cup_{d\in\Delta} d^\bot$.
The set of positive roots $\Delta^+$
 corresponding to $\MC{C}$ is defined by
\begin{equation}
 \Delta^+
 =\{ d\in\Delta \bigm| \BF{d}{\MC{C}}\subset\R_{>0} \}.
\end{equation}
We have $\Delta=\Delta^+\sqcup-\Delta^+$.
The set of simple roots $R(N,\Delta^+)$ corresponding to
 $\Delta^+$ is the set of positive roots
 $d\in\Delta^+$ such that there exists no decomposition
 $d=d_1+d_2$ with $d_i\in\Delta^+$.
It is known that $R(N,\Delta^+)$ becomes a Dynkin diagram
 of rank 24.
The automorphism group
 of the Dynkin diagram $R(N,\Delta^+)$
 is denoted by $\OP{Aut}(R(N,\Delta^+))$.
Let $W(N)$ denote the subgroup of $\OP{O}(N)$
 generated by
 reflections of $d\in\Delta$.
% and $\OP{O}(N,\Delta^+)$
% the subgroup of $\OP{O}(N)$
% which consists of $g\in\OP{O}(N)$ preserving $\Delta^+$.
The action of $W(N)$ on the set of Weyl chambers
 is free and transitive.
The group $\OP{O}(N,\Delta^+)$
 (see (\ref{osubgroup}))
 is considered as
 a subgroup of $\OP{Aut}(R(N,\Delta^+))$.
We have $\OP{O}(N)=W\rtimes\OP{O}(N,\Delta^+)$.

%\subsection{Subgroups of $\OP{O}(N,\Delta^+)$}
\subsection{Method for making the list of $\MC{N}$}
\label{subsect:method}

We use the above result to construct a complete list
 of $\MC{N}$.
For the proof of the following lemma,
 see \cite{kondo98}.

\begin{lem}[\cite{kondo98}] \label{lem:invariantroots}
Let $N$ be a Niemeier lattice and
 $G$ a subgroup of $\OP{O}(N)$.
Then the condition (\ref{defn:mcn:coinvariant}) in
 Definition \ref{defn:mcn} is satisfied if and only if
 there exists a $G$-invariant set of positive roots.
\end{lem}

Let $N_1.\cdots,N_{23}$ be all Niemeier lattices and
 $\Delta_i^+$ a set of positive roots of $N_i$.
Let $G\subset\OP{O}(N_i)$ be a subgroup satisfying
 the condition (\ref{defn:mcn:coinvariant}) in
 Definition \ref{defn:mcn}.
By the above lemma,
 we may assume that $G$ preserves $\Delta_i^+$
 by replacing $G$ by $\gamma G \gamma^{-1}$
 for some $\gamma\in W(N_i)$ if necessary.
Hence we may only consider subgroups of
 $\OP{O}(N_i,\Delta_i^+)$.
Using GAP \cite{GAP},
 we can make a complete list of subgroups
 $G_{i1},\cdots,G_{i j_i}$
 of $\OP{O}(N_i,\Delta^+_i)$
 such that $[G_{ij}]\in\Gsymp$
 up to conjugacy\footnote{Note that conjugacy in
 $\OP{O}(N_i,\Delta^+_i)$
 is equivalent to conjugacy in $\OP{O}(N_i)$, which is
 a property of semi-direct product groups.}.
Since $\OP{O}(N_i,\Delta^+_i)$ is realized as a subgroup of
 $\OP{Aut}(R(N_i,\Delta^+_i))$, so is $G_{ij}$.
%We set
%$$
%\MC{N}=\{
% (G_{ij},N_i) \bigm| (G_{ij},N_i)\in\MC{N}
%\}.
%$$
To decide whether $(G_{ij},N_i)\in\MC{N}$ or not,
 we should check conditions
 (\ref{defn:mcn:invariant})--(\ref{defn:mcn:embedding})
 in Definition \ref{defn:mcn} for $(G_{ij},N_i)$.

%Since $g\in G_{ij}$ preserves $R(N_i,\Delta^+_i)$,
% $\OP{Tr}(g;N)$ is equal to the number of orbits
% of $R(N_i,\Delta^+_i)$ under the action of $g$.
%We can check the condition (i) by using this equality.
The condition (\ref{defn:mcn:invariant})
 can be checked directly.
For example, if $N_i$ is of type $A_1^{\oplus 24}$,
 the condition (\ref{defn:mcn:invariant})
 is equivalent to
 the existence of a $G_{ij}$-fixed element in $R(N_i,\Delta^+_i)$.
By Lemma \ref{lem:invariantroots},
 the condition (\ref{defn:mcn:coinvariant})
 is already satisfied.

To confirm the condition (\ref{defn:mcn:embedding}),
 it is sufficient to show that
 there exists an even lattice $L$ such that
\begin{equation} \label{cond_L}
 \sign L=(3,19-c(G_{ij})),~
 q(L)\cong -q(N_{G_{ij}})
\end{equation}
 by Lemma \ref{lem:unimodembed} and
 Proposition \ref{prop:trace}.
We can compute the Gramian matrix of
 $N^{G_{ij}}$
 by using the orbit decomposition of $R(N_i,\Delta^+_i)$
 which is obtained from the list of $(G_{ij},N_i)$.
From the Gramian matrix of $N^{G_{ij}}$,
 we can determine
 $A(N^{G_{ij}})$ and $q(N^{G_{ij}})$
 (cf.\ Section \ref{sect:lattice}).
Since $q(N_{G_{ij}})\cong -q(N^{G_{ij}})$
 by Lemma \ref{lem:unimodembed}, we obtain the list of
 $q(N_{G_{ij}})$.
% the list of orbit decompositions of
% $R(N,\Delta^+)$ for $(G_{ij},N_i)$, where
% $G_{ij}$ is a non-abelian group.
%We can also make the list where $G_{ij}$ is abelian,
% which we omitted, since it is very large.
From the list, we have the following:

\begin{lem}
For $(G_{ij},N_i)$ satisfying
 the condition (\ref{defn:mcn:invariant})
 in Definition \ref{defn:mcn},
 the condition (\ref{defn:mcn:embedding}) is equivalent to
 the inequality
\begin{equation} \label{ineqgen}
 l(A(N^{G_{ij}}))\leq 22-c(G_{ij})
 =\rank N^{G_{ij}}-2.
\end{equation}
Here $l(A)$ denotes the minimum number
 of generators of a finite abelian group $A$.
\end{lem}
\begin{proof}
%The condition (\ref{defn:mcn:embedding}) implies that
%\begin{equation}
% l(A(N_{G_{ij}})) =
% l(A( (N_{G_{ij}})^\bot_\Lambda )) \leq
% \rank (N_{G_{ij}})^\bot_\Lambda =
% 22-c(G_{ij})
%\end{equation}
For each case satisfying
 the inequality (\ref{ineqgen}),
 we can find a lattice $L$ satisfying (\ref{cond_L}).
See Tables \ref{subsect:qands} and \ref{subsect:invariant}
 for $q(N_{G_{ij}})$ and $L$ in each case respectively.
Conversely, the existence of $L$ implies that
\begin{equation}
 l(A(N^{G_{ij}}))
 = l(A(N_{G_{ij}}))
 = l(A(L))
 \leq \rank L
 = 22-c(G_{ij})
\end{equation}
 by Lemma \ref{lem:unimodembed} and (\ref{ineq_discgrp}).
\end{proof}

By the above argument,
 the set which consists of $(G_{ij},N_i)$
 satisfying the condition (\ref{defn:mcn:invariant}) and
 the inequality (\ref{ineqgen})
 becomes a complete list of $\MC{N}$.
% $R(N_i,\Delta^+_i)$ for $(G_{ij},N_i)\in\MC{N}$.

\subsection{Example} \label{subsect:example}

We consider the case of
 the cyclic group $C_8$ of order $8$ as an example.
We make the list of $(G,N)\in\MC{N}$ with $[G]=C_8$.
Since $c(C_8)=18$, we have $\rank N_G=18$ and
 $\rank N^G=6$.
Using GAP \cite{GAP}, we can make a complete list of subgroups
 $G\subset \OP{O}(N,\Delta^+)$
 such that $[G]=C_8$ up to conjugacy
 for each Niemeier lattice $N$.
The result is as follows.
\begin{equation*}
\begin{array}{c|cccccc}
\text{case} & \text{(I)} & \text{(II)} & \text{(III)}
 & \text{(IV)} & \text{(V)} & \text{(VI)} \\
\hline
\text{root type of } N
 & E_6^{\oplus 4} & A_5^{\oplus 4}\oplus D_4
 & A_3^{\oplus 8} & A_2^{\oplus 12} & A_2^{\oplus 12}
 & A_1^{\oplus 24} \\
\hline
\begin{matrix} \text{number of stable} \\
               \text{components of $R(N,\Delta^+)$}
\end{matrix}
 & 0 & 1 & 0
 & 2 & 0 & 2 \\
\hline
(G,N)\in \MC{N}\text{?}
 & \text{no} & \text{yes} & \text{no}
 & \text{yes} & \text{no} & \text{yes}
\end{array}
\end{equation*}
If the condition (\ref{defn:mcn:invariant})
 in Definition \ref{defn:mcn}
 holds, then at least one component
 of the Dynkin diagram $R(N,\Delta^+)$
 is stable under the action of $G$.
In the case (I),
 the action of $G$
 as a permutation group of the components $E_6$
 of $R(N,\Delta^+)$ is transitive.
Therefore, we have $(G,N)\not\in\MC{N}$
 in the case (I).
Similarly, we have $(G,N)\not\in\MC{N}$
 in the cases (III) and (V).
In fact, we have $(G,N)\in\MC{N}$
 in the cases (II), (IV) and (VI),
 as we will see below.
Let $g$ be a generator of $G$.

The case (II).
There exists a numbering of
 $R(N,\Delta^+)=\{v_1,\ldots,v_{24}\}$
 as in Figure 1 such that
\begin{equation}
g\cdot v_i=
v_{\sigma(i)},
\end{equation}
 where
\begin{equation}
 \sigma=(1,6,11,16,5,10,15,20)(2,7,12,17,4,9,14,19)
 (3,8,13,18)(23,24).
\end{equation}
Hence $N^G\otimes\Q$ is generated by
\begin{equation}
\begin{array}{c}
 \displaystyle
 w_1=\sum_{i=0}^3 (v_{1+5i}+v_{5+5i}),~
 w_2=\sum_{i=0}^3 (v_{2+5i}+v_{4+5i}), \\
 \displaystyle
 w_3=\sum_{i=0}^3 v_{3+5i},~
 w_4=v_{21},~
 w_5=v_{22},~
 w_6=v_{23}+v_{24}
\end{array}
\end{equation}
over $\Q$.
From the explicit description of
 $G\subset\OP{O}(N,\Delta^+)$,
 we find that $N^G$ is generated by
 the above vectors and
%\begin{equation}
 $(w_1+w_3)/2$
%\end{equation}
 over $\Z$.
Therefore,
\begin{equation} \label{basisII}
 w_1,w_2,(w_1+w_3)/2,w_4,w_5,w_6
\end{equation}
 form a basis of $N^G$ over $\Z$.
The Gramian matrix of $N^G$
 under the basis (\ref{basisII}) is
\begin{equation} \label{gram1}
\begin{pmatrix}
-16 &   8 &  0 & 0 & 0 & 0 \\
  8 & -16 &  8 & 0 & 0 & 0 \\
  0 &   8 & -8 & 0 & 0 & 0 \\
0 & 0 & 0 & -2 &  1 & 0 \\
0 & 0 & 0 &  1 & -2 & 2 \\
0 & 0 & 0 &  0 &  2 & -4
\end{pmatrix}.
\end{equation}
We can determine $A(N^G)$ and $q(N^G)$ from (\ref{gram1})
 (cf.\ Section \ref{sect:lattice}):
\begin{align}
 A(N^G)&\cong \Z/2\Z \oplus \Z/4\Z \oplus
 (\Z/8\Z)^{\oplus 2}, \\
 q(N^G)&\cong \LF{1/2} \oplus \LF{1/4}
 \oplus \begin{pmatrix}0&1/8\\1/8&0\end{pmatrix}.
\end{align}
Since $q(N_G)\cong -q(N^G)$
 by Lemma \ref{lem:unimodembed}, we have
\begin{equation} \label{discformII}
 q(N_G)\cong \LF{-1/2} \oplus \LF{-1/4}
 \oplus \begin{pmatrix}0&1/8\\1/8&0\end{pmatrix}.
\end{equation}

\begin{figure}[bthp] \label{fig:4a5d4}
\setlength\unitlength{0.8truecm}
\begin{center}
\begin{picture}( 14 , 6 )(0,0) 
%\put(0,0){\framebox( 14 , 6 ){}} 
\put( 1 , 4.83 ){$ 1 $} 
\put( 1 , 3.83 ){$ 2 $} 
\put( 1 , 2.83 ){$ 3 $} 
\put( 1 , 1.83 ){$ 4 $} 
\put( 1 , 0.83 ){$ 5 $} 
\put( 3 , 4.83 ){$ 6 $} 
\put( 3 , 3.83 ){$ 7 $} 
\put( 3 , 2.83 ){$ 8 $} 
\put( 3 , 1.83 ){$ 9 $} 
\put( 3 , 0.83 ){$ 10 $} 
\put( 5 , 4.83 ){$ 11 $} 
\put( 5 , 3.83 ){$ 12 $} 
\put( 5 , 2.83 ){$ 13 $} 
\put( 5 , 1.83 ){$ 14 $} 
\put( 5 , 0.83 ){$ 15 $} 
\put( 7 , 4.83 ){$ 16 $} 
\put( 7 , 3.83 ){$ 17 $} 
\put( 7 , 2.83 ){$ 18 $} 
\put( 7 , 1.83 ){$ 19 $} 
\put( 7 , 0.83 ){$ 20 $} 
\put( 2 , 5 ){\circle{ 0.3 }} 
\put( 2 , 4 ){\circle{ 0.3 }} 
\put( 2 , 3 ){\circle{ 0.3 }} 
\put( 2 , 2 ){\circle{ 0.3 }} 
\put( 2 , 1 ){\circle{ 0.3 }} 
\put( 4 , 5 ){\circle{ 0.3 }} 
\put( 4 , 4 ){\circle{ 0.3 }} 
\put( 4 , 3 ){\circle{ 0.3 }} 
\put( 4 , 2 ){\circle{ 0.3 }} 
\put( 4 , 1 ){\circle{ 0.3 }} 
\put( 6 , 5 ){\circle{ 0.3 }} 
\put( 6 , 4 ){\circle{ 0.3 }} 
\put( 6 , 3 ){\circle{ 0.3 }} 
\put( 6 , 2 ){\circle{ 0.3 }} 
\put( 6 , 1 ){\circle{ 0.3 }} 
\put( 8 , 5 ){\circle{ 0.3 }} 
\put( 8 , 4 ){\circle{ 0.3 }} 
\put( 8 , 3 ){\circle{ 0.3 }} 
\put( 8 , 2 ){\circle{ 0.3 }} 
\put( 8 , 1 ){\circle{ 0.3 }} 
\put( 2 , 4.84 ){\line(0,-1){ 0.69 }} 
\put( 2 , 3.84 ){\line(0,-1){ 0.69 }} 
\put( 2 , 2.84 ){\line(0,-1){ 0.69 }} 
\put( 2 , 1.84 ){\line(0,-1){ 0.69 }} 
\put( 4 , 4.84 ){\line(0,-1){ 0.69 }} 
\put( 4 , 3.84 ){\line(0,-1){ 0.69 }} 
\put( 4 , 2.84 ){\line(0,-1){ 0.69 }} 
\put( 4 , 1.84 ){\line(0,-1){ 0.69 }} 
\put( 6 , 4.84 ){\line(0,-1){ 0.69 }} 
\put( 6 , 3.84 ){\line(0,-1){ 0.69 }} 
\put( 6 , 2.84 ){\line(0,-1){ 0.69 }} 
\put( 6 , 1.84 ){\line(0,-1){ 0.69 }} 
\put( 8 , 4.84 ){\line(0,-1){ 0.69 }} 
\put( 8 , 3.84 ){\line(0,-1){ 0.69 }} 
\put( 8 , 2.84 ){\line(0,-1){ 0.69 }} 
\put( 8 , 1.84 ){\line(0,-1){ 0.69 }} 

\put(10,3){\circle{0.3}}
\put(9.75,3.55){$21$}
\put(11.5,3){\circle{0.3}}
\put(11.25,3.55){$22$}
\put(12.5,4){\circle{0.3}}
\put(12.9,3.83){$23$}
\put(12.5,2){\circle{0.3}}
\put(12.9,1.83){$24$}
\put(10.15,3){\line(1,0){1.2}}
\put(11.61,3.11){\line(1,1){0.77}}
\put(11.62,2.90){\line(1,-1){0.77}}
\end{picture} 
\end{center}
\caption{$A_5^{\oplus 4}\oplus D_4$}
\end{figure}
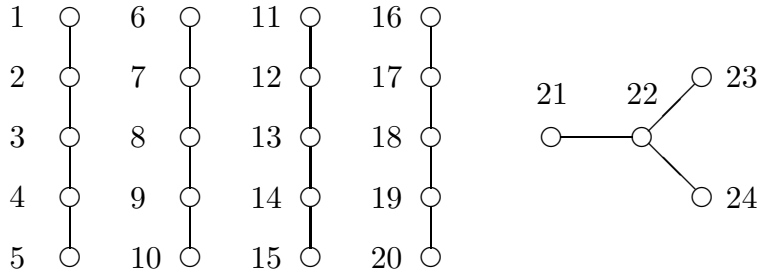

The case (IV).
Similarly, there exists a numbering of
 $R(N,\Delta^+)$
 as in Figure 2 such that
$
g\cdot v_i=
v_{\sigma(i)}
$,
 where
\begin{equation}
 \sigma=(3,4)(5,7,6,8)(9,11,13,15,17,19,21,23)
                     (10,12,14,16,18,20,22,24).
\end{equation}
Moreover, $N^G\otimes\Q$ is generated by
\begin{equation}
\begin{array}{c}
 \displaystyle
 w_1=v_1,~
 w_2=v_2,~
 w_3=v_3+v_4,~
 w_4=\sum_{i=5}^8 v_{i}, \\
 \displaystyle
 w_5=\sum_{i=0}^7 v_{9+2i},~
 w_6=\sum_{i=0}^7 v_{10+2i}
\end{array}
\end{equation}
over $\Q$, and $N^G$ is generated by
\begin{equation} \label{basisIV}
 w_1,w_2,w_3,w_4,w_5,
 \frac{1}{3} ( w_1-w_2+w_5-w_6 )
\end{equation}
 over $\Z$.
The Gramian matrix of $N^G$
 under the basis (\ref{basisIV}) is
\begin{equation} \label{gramIV}
\begin{pmatrix}
 -2 &   1 &  0 &  0 &  0 & -1 \\
  1 &  -2 &  0 &  0 &  0 &  1 \\
  0 &   0 & -2 &  0 &  0 &  0 \\
  0 &   0 &  0 & -4 &  0 &  0 \\
  0 &   0 &  0 &  0 &-16 & -8 \\
 -1 &   1 &  0 &  0 & -8 & -6
\end{pmatrix}.
\end{equation}
From (\ref{gramIV}),
 we can check that $q(N_G)$ is isomorphic to (\ref{discformII}).

\begin{figure}[bth]
\setlength\unitlength{0.8truecm}
\begin{center}
\begin{picture}( 11 , 3 )(0,0) 
%\put(0,0){\framebox( 11 , 3 ){}} 
\put( 1 , 1.83 ){$ 1 $} 
\put( 1 , 0.83 ){$ 2 $} 
\put( 3 , 1.83 ){$ 3 $} 
\put( 3 , 0.83 ){$ 4 $} 
\put( 7 , 1.83 ){$ 21 $} 
\put( 7 , 0.83 ){$ 22 $} 
\put( 9 , 1.83 ){$ 23 $} 
\put( 9 , 0.83 ){$ 24 $} 
\put( 2 , 2 ){\circle{ 0.3 }} 
\put( 2 , 1 ){\circle{ 0.3 }} 
\put( 4 , 2 ){\circle{ 0.3 }} 
\put( 4 , 1 ){\circle{ 0.3 }} 
\put(5.2,1.33){$\cdots$}
\put( 8 , 2 ){\circle{ 0.3 }} 
\put( 8 , 1 ){\circle{ 0.3 }} 
\put( 10 , 2 ){\circle{ 0.3 }} 
\put( 10 , 1 ){\circle{ 0.3 }} 
\put( 2 , 1.84 ){\line(0,-1){ 0.69 }} 
\put( 4 , 1.84 ){\line(0,-1){ 0.69 }} 
\put( 8 , 1.84 ){\line(0,-1){ 0.69 }} 
\put( 10 , 1.84 ){\line(0,-1){ 0.69 }} 
\end{picture}
\end{center}
\caption{$A_2^{\oplus 12}$}
\end{figure}
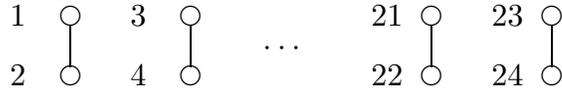

The case (VI).
There exists a numbering of
 $R(N,\Delta^+)$
 as in Figure 3 such that
$
g\cdot v_i=
v_{\sigma(i)}
$,
 where
\begin{equation}
 \sigma=(3,4)(5,6,7,8)(9,10,11,12,13,14,15,16)
                     (17,18,19,20,21,22,23,24).
\end{equation}
Moreover, $N^G\otimes\Q$ is generated by
\begin{equation}
\begin{array}{c}
 \displaystyle
 w_1=v_1,~
 w_2=v_2,~
 w_3=\sum_{i=3}^4 v_{i},~
 w_4=\sum_{i=5}^8 v_{i}, \\
 \displaystyle
 w_5=\sum_{i=9}^{16} v_{i},~
 w_6=\sum_{i=17}^{24} v_{i}
\end{array}
\end{equation}
over $\Q$, and $N^G$ is generated by
\begin{equation} \label{basisVI}
 w_1,w_2,w_3,
 \frac{1}{2} ( w_1+w_2+w_3+w_4 ),
 \frac{1}{2} ( w_4+w_5 ),
 \frac{1}{2} ( w_4+w_6 )
\end{equation}
 over $\Z$.
The Gramian matrix of $N^G$
 under the basis (\ref{basisVI}) is
\begin{equation} \label{gramVI}
\begin{pmatrix}
 -2 &   0 &  0 & -1 &  0 &  0 \\
  0 &  -2 &  0 & -1 &  0 &  0 \\
  0 &   0 & -4 & -2 &  0 &  0 \\
 -1 &  -1 & -2 & -4 & -2 & -2 \\
  0 &   0 &  0 & -2 & -6 & -2 \\
  0 &   0 &  0 & -2 & -2 & -6
\end{pmatrix}.
\end{equation}
From (\ref{gramVI}),
 we can check that $q(N_G)$ is isomorphic to (\ref{discformII}).

\begin{figure}[bth]
\setlength\unitlength{0.8truecm}
\begin{center}
\begin{picture}( 11 , 2 )(0,0) 
%\put(0,0){\framebox( 11 , 2 ){}} 
\put( 1 , 0.83 ){$ 1 $} 
\put( 3 , 0.83 ){$ 2 $} 
\put( 7 , 0.83 ){$ 23 $} 
\put( 9 , 0.83 ){$ 24 $} 
\put( 2 , 1 ){\circle{ 0.3 }} 
\put( 4 , 1 ){\circle{ 0.3 }} 
\put(5.2,0.83){$\cdots$}
\put( 8 , 1 ){\circle{ 0.3 }} 
\put( 10 , 1 ){\circle{ 0.3 }} 
\end{picture}
\end{center}
\caption{$A_1^{\oplus 24}$}
\end{figure}

The type of the root sublattice of $N^G$,
 i.e.\ the sublattice generated by vectors $v\in N^G$
 such that $\BF{v}{v}=-2$,
 in each case is as follows.
\begin{equation} \label{table_c8}
\begin{array}{c|ccc}
\text{case} & \text{(II)} & \text{(IV)} & \text{(VI)} \\
\hline
\text{root type} & A_3 & A_1 \oplus A_2 & A_1^{\oplus 2}
\end{array}
\end{equation}
Hence the condition (\ref{defn:mcn:invariant})
 in Definition \ref{defn:mcn} is satisfied.
The condition (\ref{defn:mcn:coinvariant}) is satisfied
 by Lemma \ref{lem:invariantroots}.
By the above argument, we have
\begin{equation}
 q(N_G)\cong \LF{-1/2} \oplus \LF{-1/4}
 \oplus \begin{pmatrix}0&1/8\\1/8&0\end{pmatrix}
\end{equation}
 in each case.
Let $L$ be a lattice defined by
\begin{equation}
 L=\LF{2} \oplus \LF{4} \oplus
 \begin{pmatrix}0&8\\8&0\end{pmatrix}.
\end{equation}
Then we have $\sign L=(3,1)$ and $q(L)\cong -q(N_G)$.
By Lemma \ref{lem:unimodembed},
 there exists a primitive embedding $N_G\hookrightarrow\Lambda$
 such that $(N_G)^\bot_\Lambda\cong L$.
Thus the condition (\ref{defn:mcn:embedding}) is satisfied.
Therefore, we have $(G,N)\in\MC{N}$
 in the cases (II), (IV) and (VI).

\subsection{Consequences from the list of $\MC{N}$}
\label{subsect:conseq}

Let $\MC{Q}$ denote a set defined by
\begin{equation} \label{mcq}
\MC{Q}=\{
 (\MF{G},q) \bigm| \exists G\in\MC{L}
 \text{ such that }
 \MF{G} = [G],
 q\cong q(\Lambda_{G})
\}.
\end{equation}
By Lemma \ref{lem:kondoembed}, we have
\begin{equation} \label{mcqequation}
\MC{Q}=\{
 (\MF{G},q) \bigm| \exists (G,N)\in\MC{N}
 \text{ such that }
 \MF{G} = [G],
 q\cong q(N_{G})
\}.
\end{equation}
We introduce an equivalence relation $\sim$ on $\MC{Q}$ by
\begin{equation}
 (\MF{G},q)\sim(\MF{G}',q')
 \Leftrightarrow
 \MF{G}=\MF{G}' \text{~and~} q\cong q'.
\end{equation}
By (\ref{mcqequation}) and the list of
 $q((N_i)^{G_{ij}})$ for $(G_{ij},N_i)\in\MC{N}$,
 we have the following:

\begin{prop} \label{prop:uniquemcq}
For $\MF{G}\in\Gsymp$, we have
\begin{equation}
 \sharp \left(
 \{ q \bigm| (\MF{G},q)\in\MC{Q} \}
 /\text{\rm isom}
 \right)=
 \begin{cases}
 1 & \text{if } \MF{G}\neq Q_8, T_{24}, \\
 2 & \text{if } \MF{G}  =  Q_8, T_{24}.
\end{cases}
\end{equation}
\end{prop}

\begin{rem}
From the Xiao's list \cite{xiao96},
 we have $\sharp \Gsymp=79$.
By the above proposition,
 $\sharp \left( \MC{Q}/\sim \right)=81$.
In Table \ref{subsect:qands},
 we list a complete representative
 $\{ (\MF{G}_n,q_n) \}$ of $\MC{Q}/\sim$.
%We fix a complete representative system
% $\{ (G_n,S_n) \}$ of $\MC{S}/\text{\rm isom}$
% such that $[G_n]=\MF{G}_n$ and $q(S_n)\cong q_n$.
Our numbering coincides with
 that in \cite{xiao96}.
\end{rem}

By (\ref{mcqequation}), we have the natural map
\begin{equation}
 \pi:\MC{N}\rightarrow \MC{Q};\quad
 (G,N) \mapsto ([G],q(N)).
\end{equation}
In Table \ref{subsect:orbits},
 the type of the root sublattice of
 $N^G$ for each $(G,N)\in\MC{N}$ is given.
From the table, we have the following:

\begin{prop} \label{prop:listuniq}
%There exists a subset $\MC{N}'$ of $\MC{N}$
% which satisfies the following conditions.
%\begin{enumerate}
%\renewcommand{\labelenumi}{(\arabic{enumi})}
%\item
%For any $G\in\MC{L}$ with $[G]\neq \Gamma_{22}a_1$,
% there exists an element $(G',N')\in\MC{N}'$ such that
% $[G']=[G]$ and
% $q(N'_{G'})\cong q(\Lambda_G)$.
%\item
%If $(G,N)\in\MC{N}$, $(G',N')\in\MC{N}'$,
% $[G] = [G']$ and
% $N^{G} \cong (N')^{G'}$,
% then $(G,N)\cong(G',N')$.
%\end{enumerate}
Let $\MC{Q}^\circ$ denote the subset of $\MC{Q}$
 defined by
\begin{equation}
\MC{Q}^\circ=\{
 (\MF{G},q) \in \MC{Q} \bigm| 
 \MF{G}\neq \MF{G}_{58}
\}.
\end{equation}
There exists a section
 $\sigma:\MC{Q}^\circ\rightarrow\pi^{-1}(\MC{Q}^\circ)$
 of $\pi$ with the following conditions.
We set $\MC{N}'=\sigma(\MC{Q}^\circ)$.
\begin{enumerate}
\renewcommand{\labelenumi}{(\arabic{enumi})}
\item
Let $(G,N)\in\MC{N}$ and $(G',N')\in\MC{N}'$.
If $\pi(G,N)=\pi(G',N')$
 and $N^G\cong (N')^{G'}$, then $(G,N)\cong(G',N')$.
\item
Let $(G,N)\in\MC{N}'$.
If $[G]\neq\MF{G}_3$,
 then $N$ is of type $A_1^{\oplus 24}$.
\end{enumerate}
\end{prop}
\begin{proof}
For each $(\MF{G},q)\in\MC{Q}^\circ$,
 we can chose $\sigma(\MF{G},q)\in\MC{N}$
 case by case.
For example,
 we consider the case of $C_8=\MF{G}_{14}$
 (see Subsection \ref{subsect:example}).
By the table (\ref{table_c8}),
 the root types of $N^G$ for
 $(G,N)\in\MC{N}$ with $[G]=C_8$
 are different from each other.
Therefore, $N^G$ are
 not isomorphic to each other.
Hence we can chose $(G,N)$ of the case (VI),
 in which $N$ is of type $A_1^{\oplus 24}$,
 as $\sigma(\MF{G}_{14},q_{14})$.
Similarly, for $(G,N)\in\MC{N}$
 with $\pi(G,N)=(\MF{G}_n,q_n)$,
 the isomorphism classes of $N^G$ can be distinguished
 by looking the root types
 except for the cases $n=32,41,56,63$.
For the cases $n=32,41,56,63$,
 we can distinguish them by looking
 the root types and the numbers of
 vectors $v\in N^G$ such that $\BF{v}{v}=-4$ or $-6$.
As a consequence,
 we can choose $(G,N)$
 enclosed by boxes in
 Table \ref{subsect:orbits}.
The choice of $\sigma$ is not unique.
\end{proof}

\section{Uniqueness of coinvariant lattices $\Lambda_G$}
\label{sect:coinvariantuniq}

Let $\MC{S}$ denote a set of $G$-lattices
 defined by
\begin{equation} \label{mcs}
\MC{S}=\{
(G,S) \bigm|
\exists G'\in\MC{L}
 \text{ such that }
 (G,S)\cong (G'%|_{\Lambda_{G'}}
,\Lambda_{G'})
% S=\Lambda_{G_1},G=G_1 |_S\subset\OP{O}(S)
\}.
\end{equation}
%For an isomorphism of elements in $\MC{S}$,
% see Definition \ref{defn:latticewithaction}.
For $(G,S)\in\MC{S}$, we have $G\subset\OP{O}_0(S)$
 by Lemma \ref{lem:unimodaction}(3).
In this section, we apply the results
 in the previous section to prove the following:
% $(G,S)\in\MC{S}$ is determined uniquely
% by $[G]$ and $q(S)$
% (see Theorem \ref{thm:coinvariantuniq}).

\begin{thm} \label{thm:coinvariantuniq}
%Consider a natural map
The natural map
 $\varphi:\MC{S}/\text{\rm isom}\rightarrow\MC{Q}/\sim$
 is bijective.
%The inverse image $f^{-1}([(\MF{G},q)])$ of
% $[(\MF{G},q)]\in\MC{Q}/\sim$
% with $\MF{G}\neq\Gamma_{22}a_1$
% consists of one element.
%
%Let $(\MF{G},q)\in\MC{Q}$ such that
% $\MF{G}\neq\Gamma_{22}a_1$.
%There exists unique $(G,S)\in\MC{S}$ up to isomorphism
% such that
% $[G]=\MF{G}$ and $q(S)\cong q$.
%The natural map $\MC{S}/\sim\rightarrow\MC{Q}/\sim$
% is bijective.
\end{thm}
\begin{proof}
%By the definitions of $\MC{S}$ and $\MC{Q}$,
% the existence of $(G,S)\in\MC{S}$ is trivial.
The surjectivity of $\varphi$ is trivial.
We shall show the injectivity.
Let $(\MF{G},q)\in\MC{Q}$.
Suppose that $(G,S)\in\MC{S}$,
 $[G]=\MF{G}$ and $q(S)\cong q$.
%We distinguish two cases:
% $\MF{G}\neq\MF{G}_{58}$ and $\MF{G}=\MF{G}_{58}$.
%Note that $\Gamma_{22}a_1$ which appears
% in Proposition \ref{prop:listuniq} is equal to
% $\MF{G}_{58}$.
We show that
 $(G,S)$ is uniquely determined up to isomorphism.

(1)
The case $\MF{G}\neq \MF{G}_{58}$.
By Proposition \ref{prop:listuniq},
 there exists an element $(\Gamma,N)\in\MC{N}'$
 such that $[\Gamma]=\MF{G}$ and $q(N_\Gamma)\cong q$.
We show that $(G,S)\cong (\Gamma,N_\Gamma)$.
By Lemma \ref{lem:unimodembed},
 $q(S) \cong q \cong q(N_\Gamma) \cong -q(N^\Gamma)$.
Again by Lemma \ref{lem:unimodembed},
 there exists a primitive embedding $S\hookrightarrow N'$
 of $S$ into a Niemeier lattice $N'$
 such that $(S)^\bot_{N'}\cong N^\Gamma$.
By Lemma \ref{lem:unimodaction},
 the action of $G$ on $S$ is extended to
 that on $N'$ such that
 $(N')_{G}=S$ and $(N')^{G}\cong N^\Gamma$.
Thus $(G,N')\in\MC{N}$
 (see Definition \ref{defn:mcn}).
By Proposition \ref{prop:listuniq},
 we have $(G,N')\cong(\Gamma,N)$.
Hence $(G,S)=(G,(N')_{G})\cong(\Gamma,N_\Gamma)$.
%Thus the map is injective
% except for the case $n=58$
% in Table \ref{subsect:qands},
% where $(\MF{G},q)\sim(\MF{G}_{58},q_{58})$.

(2)
The case $\MF{G}=\MF{G}_{58}$.
From Table \ref{subsect:trees}, we find that
 $\MF{G}_{43}\subsetneq \MF{G}_{58}$
 and
% $c(\Gamma_7 a_2) = c(\Gamma_{22} a_1)$.
 $c(\MF{G}_{43}) = c(\MF{G}_{58})$.
Hence there exists a subgroup $G'_{43}$ of $G$
 such that $[G'_{43}]=\MF{G}_{43}$.
Since $c(\MF{G}_{43}) = c(\MF{G}_{58})$,
 we have $(G'_{43},S)\in\MC{S}$.
Let $G_{43}\in\MC{L}$ be as in Lemma \ref{lem:sect4}.
By (1) and
 Proposition \ref{prop:uniquemcq},
 $(G',S')\in\MC{S}$
 such that $[G']=\MF{G}_{43}$
 is unique up to isomorphism.
Therefore, we have
 $(G'_{43},S)\cong (G_{43},\Lambda_{G_{43}})$.
By the condition (2) in Lemma \ref{lem:sect4},
 there exists a unique subgroup
 $G_{58}$ of $\OP{O}_0(\Lambda_{G_{48}})$
 such that $[G_{58}]=\MF{G}_{58}$
 up to conjugacy in $\OP{O}(\Lambda_{G_{48}})$.
Hence $(G,S)\cong(G_{58},\Lambda_{G_{43}})$.
\end{proof}

\begin{defn} \label{defn:sgq}
Let $(\MF{G},q)\in\MC{Q}$.
By Theorem \ref{thm:coinvariantuniq},
 there exists a unique element $(G,S)\in\MC{S}$
 such that $([G],q(S))\sim (\MF{G},q)$,
 i.e.,
 $[G]=\MF{G}$ and $q(S)\cong q$
 up to isomorphism.
The lattice $S$ determined by this conditions is
 denoted by $S(\MF{G},q)$.
Since $G\subset\OP{O}_0(S)$,
 $\MF{G}$ is a subgroup of
 $[\OP{O}_0(S(\MF{G},q))]$.
\end{defn}

By the definition of $S(\MF{G},q)$, we have
\begin{equation} \label{lambda2sgq}
 \Lambda_G \cong S([G],q(\Lambda_G))
\end{equation}
 for $G\in\MC{L}$.

\begin{cor} \label{cor:sgq}
%If $\MF{G}_n\subset\MF{G}_m$, $q_n\cong q_m$ and
% $c(\MF{G}_n)=c(\MF{G}_m)$, then
% $S(\MF{G}_n,q_n)\cong S(\MF{G}',q')$.
Let $(\MF{G},q),(\MF{G}',q')\in\MC{Q}$.
If $\MF{G}\subset\MF{G}'$, $q\cong q'$ and
 $c(\MF{G})=c(\MF{G}')$, then
 $S(\MF{G},q)\cong S(\MF{G}',q')$. 
\end{cor}
\begin{proof}
Let $G'\in\MC{L}$ such that
 $[G']=\MF{G}'$ and $q(\Lambda_{G'})\cong q'$.
Then $\Lambda_{G'}\cong S(\MF{G}',q')$.
%Let $G$ be a subgroup of $G'$ with $[G]=\MF{G}$.
Let $G$ be the subgroup of $G'$
 which corresponds to the subgroup $\MF{G}$ of $\MF{G}'$.
Since $c(G)=c(G')$, we have
 $S(\MF{G},q)\cong \Lambda_{G}=\Lambda_{G'}
 \cong S(\MF{G}',q')$.
\end{proof}

\begin{rem}
In Table \ref{subsect:trees},
 we give the trees of
\begin{equation}
 T_S=
 \{ \MF{G}_n \bigm| S(\MF{G}_n,q_n)\cong S \}
\end{equation}
 for $T_S$ with $\sharp T_S\geq 2$.
% which have isomorphic $S(\MF{G}_n,q_n)$.
From Tables \ref{subsect:qands} and \ref{subsect:trees},
 we find that
 there exist exactly 40 isomorphism classes of lattices
 $S(\MF{G}_n,q_n)$
 (or $\Lambda_G$ for $G\in\MC{L}$).
Also,
 we can check that the natural map
\begin{equation}
%\{ \Lambda_G \bigm| G\in\MC{L} \} /\text{\rm isom}
% \rightarrow
% \{ q \bigm| G\in\MC{L}, q\cong(\Lambda_G) \}
% /\text{\rm isom}
 \{ S(\MF{G},q) \bigm| (\MF{G},q)\in\MC{Q} \}
 /\text{\rm isom}
 \rightarrow
 \{ q \bigm| (\MF{G},q)\in\MC{Q}, q\cong q(S(\MF{G},q)) \}
 /\text{\rm isom}
\end{equation}
 is bijective.
\end{rem}

\begin{defn} \label{defn:closgq}
Let $(\MF{G},q)\in\MC{Q}$.
We define $\OP{Clos}(\MF{G},q)$ by
\begin{equation}
 \OP{Clos}(\MF{G},q)=
 ([\OP{O}_0(S(\MF{G},q))],q).
\end{equation}
Note that $\MF{G}$ is a subgroup of
 $[\OP{O}_0(S(\MF{G},q))]$
 (see Definition \ref{defn:sgq}).
\end{defn}

For $(\MF{G},q)\in\MC{Q}$,
 there exists an element $G\in\MC{L}$
 such that $([G],q(\Lambda_G))\sim (\MF{G},q)$.
Since $S([G],q(\Lambda_G))\cong \Lambda_G$,
 we have
\begin{equation}
 \OP{Clos}(\MF{G},q)
 = ([\OP{O}_0(\Lambda_G)],q)
 = ([\OP{Clos}(G)],q)
\end{equation}
 (see Definition \ref{defn:closg}).
In particular,
 we have $\OP{Clos}(\MF{G},q)\in\MC{Q}$.
Let $\MC{Q}_\clos$ denote a subset of $\MC{Q}$
 defined by
\begin{equation} \label{qclos}
 \MC{Q}_\clos
 =\{ (\MF{G},q)\in\MC{Q} \bigm|
 \OP{Clos}(\MF{G},q)=(\MF{G},q) \}.
\end{equation}
For $G\in\MC{L}$, $G\in\MC{L}_\clos$
 if and only if $([G],q(\Lambda_G))\in\MC{Q}_\clos$.

\begin{cor} \label{cor:qcloslambda}
The map
\begin{equation} \label{qclos2sgq}
 \MC{Q}_\clos/\sim
 \rightarrow
% \{ S(\MF{G},q) \bigm| (\MF{G},q)\in\MC{Q} \}
% /\text{\rm isom}
 \{ \Lambda_G \bigm| G\in\MC{L} \}/\text{\rm isom}
\end{equation}
 which is induced by the correspondence
 $(\MF{G},q)\mapsto S(\MF{G},q)$
 is bijective.
\end{cor}
\begin{proof}
The inverse map of (\ref{qclos2sgq}) is the map
 induced by the correspondence
 $S\mapsto ([\OP{O}_0(S)],q(S))$.
\end{proof}

\begin{cor} \label{cor:closure}
%For $n\in\{ 1,\ldots,81 \}$, we have
% $S(\MF{G}_n,q_n)\cong S(\MF{G}_m,q_m)$
% and $[\OP{O}_0(S(\MF{G}_n,q_n))]=\MF{G}_m$,
% where $m$ is determined as follows.
Let $(\MF{G},q)\in\MC{Q}$.
Then we have $\OP{Clos}(\MF{G},q)=(\MF{G}',q)$,
 where $\MF{G}'$ is the unique maximal element in
\begin{equation} \label{closset}
 \{ \MF{G}''\in\Gsymp \bigm|
 (\MF{G}'',q'')\in\MC{Q},
 \MF{G}\subset\MF{G}'', q\cong q'',
 c(\MF{G})=c(\MF{G}'') \}.
\end{equation}
Moreover, we have the following.
\begin{enumerate}
\renewcommand{\labelenumi}{(\arabic{enumi})}
\item \label{cor:closure:qt}
If $\MF{G}\in \{ Q_8,T_{24} \}$, i.e.,
 $(\MF{G},q)\sim (\MF{G}_n,q_n)$
 for $n\in \{ 12,13,37,38 \}$,
 then we have the follwoing table.
\begin{equation*}
\begin{array}{c|c|c|c}
n & \MF{G}=\MF{G}_n
 & m & \MF{G}'=\MF{G}_m \\
\hline
12 & Q_8 & 12 & Q_8 \\
13 & Q_8 & 40 & Q_8*Q_8 \\
37 & T_{24} & 77 & T_{192} \\
38 & T_{24} & 54 & T_{48}
\end{array}
\end{equation*}
Here $m$ is determined by
 $(\MF{G}_m,q_m)\sim \OP{Clos}(\MF{G},q)$.
\item \label{cor:closure:nonqt}
If %$n\not\in \{ 12,13,37,38 \}$, i.e.,
 $\MF{G}\not\in \{ Q_8,T_{24} \}$,
 then $\MF{G}'$
 is the unique maximal element in
% $\Gsymp$
% such that $\MF{G}_n\subset\MF{G}_m$
% and $c(\MF{G}_n)=c(\MF{G}_m)$.
\begin{equation}
 \{ \MF{G}''\in\Gsymp \bigm|
 \MF{G}\subset\MF{G}'', c(\MF{G})=c(\MF{G}'') \}.
\end{equation}
\end{enumerate}
\end{cor}
\begin{proof}
For any element $\MF{G}''$ in (\ref{closset}),
 we have $S(\MF{G},q) \cong S(\MF{G}'',q'')$
 by Corollary \ref{cor:sgq}.
Hence
 $\MF{G}''\subset \MF{G}'=[\OP{O}_0(S(\MF{G},q))]$.
Therefore, the former part of the corollary follows.
We can directly check the latter part by
 Proposition \ref{prop:uniquemcq} and
 Table \ref{subsect:trees}.
\end{proof}

\section{Property %
 $\overline{\OP{O}(\Lambda_{G})}=\OP{O}(q(\Lambda_{G}))$}
\label{sect:coinvariantsurj}

This section is devoted
 to prove the following theorem,
 which gives a sufficient condition
 for $G\in\MC{L}$
 such that
 $\overline{\OP{O}(\Lambda_{G})}=\OP{O}(q(\Lambda_{G}))$
 (see (\ref{image_in_oq})).

\begin{thm} \label{thm:surjcoinv}
Let $G\in\MC{L}$
 with $c(G)=\OP{rank}\Lambda_G\geq 17$
 (see Proposition \ref{prop:trace}).
The group $\overline{\OP{O}(\Lambda_{G})}$
 is qual to $\OP{O}(q(\Lambda_{G}))$
 if and only if
 $[\OP{Clos}(G)]\in \{ \MF{G}_{48},\MF{G}_{51} \}$.
In particular, if $c(G)=\OP{rank}\Lambda_G=19$, then
 $\overline{\OP{O}(\Lambda_{G})}=\OP{O}(q(\Lambda_{G}))$.
\end{thm}

Since $c(\MF{G}_{48})=c(\MF{G}_{51})=18$
 by Table \ref{subsect:qands},
 the latter part of the theorem follows
 from the former part.

\subsection{Criterion of the property %
 $\overline{\OP{O}(L)}=\OP{O}(q(L))$}

We prepare for a criterion of
 the property $\overline{\OP{O}(L)}=\OP{O}(q(L))$.

\begin{lem} \label{lem:doublecoset}
Let $H$ be a group and $K_1,K_2$ subgroups of $H$.
If $K_1\subset K_2$ and $\sharp K_1\backslash H/K_2=1$,
 then $K_2=H$.
\end{lem}
\begin{proof}
By the second assumption,
 any element in $H$ is
 of the form $k_1 k_2$ with $k_i\in K_i$.
Hence $K_2=H$ by the first assumption.
\end{proof}

\begin{prop} \label{prop:criterion}
Let $L_1$ be a non-degenerate even lattice.
The group $\overline{\OP{O}(L_1)}$
 is qual to $\OP{O}(q(L_1))$
 if and only if
 there exists a non-degenerate even lattice $L_2$
 satisfying the following conditions.
\begin{enumerate}
\renewcommand{\labelenumi}{(\arabic{enumi})}
\item \label{prop:criterion:uniq}
There exists an essentially unique
 even unimodular lattice
 $\Gamma\subset L_1^\vee\oplus L_2^\vee$
 which contains $L_i$ primitively.
Here the uniqueness of $\Gamma$ means that
 for another $\Gamma'$, there exist isomorphisms
 $\varphi_i\in\OP{O}(L_i)$ for $i=1,2$ such that
 $\varphi_1\oplus\varphi_2$ induces
 an isomorphism $\Gamma\rightarrow\Gamma'$.
\item \label{prop:criterion:surj}
%Let $\OP{O}(\Gamma,L_2)$ denote the group
% consisting of $g\in\OP{O}(\Gamma)$ preserving $L_2$.
The restriction map
 $\OP{O}(\Gamma,L_2)\rightarrow\OP{O}(L_2)$
 is surjective (see (\ref{osubgroup})).
\end{enumerate}
\end{prop}
\begin{proof}
Assume that there exists $L_2$
 satisfying the conditions
 (\ref{prop:criterion:uniq}) and
 (\ref{prop:criterion:surj}).
Let $\gamma\in\OP{Isom}(q(L_1),-q(L_2))$ be
 the isomorphism corresponding to $\Gamma$
 (see Lamma \ref{lem:unimodembed}).
The condition (\ref{prop:criterion:uniq}) implies that
\begin{equation}
\overline{\OP{O}(L_2)}\backslash\OP{Isom}(q(L_1),-q(L_2))
 /\overline{\OP{O}(L_1)}
 \cong
 \gamma^{-1}\circ\overline{\OP{O}(L_2)}\circ\gamma
 \backslash\OP{O}(q(L_1))
 /\overline{\OP{O}(L_1)}
\end{equation}
 is a one point set
 by Lemma \ref{lem:unimodembed}.
On the other hand,
 the condition (\ref{prop:criterion:surj}) implies that
 for any $\varphi_2\in\OP{O}(L_2)$,
 there exists an automorphism
 $\varphi_1\in\OP{O}(L_1)$ such that
 $\gamma\circ\overline{\varphi}_1\circ\gamma^{-1}
 =\overline{\varphi}_2$
 by Lemma \ref{lem:unimodembed}.
Hence
 $\gamma^{-1}\circ\overline{\OP{O}(L_2)}\circ\gamma
 \subset\overline{\OP{O}(L_1)}$.
By Lemma \ref{lem:doublecoset},
 we have $\overline{\OP{O}(L_1)}=\OP{O}(q(L_1))$.

Conversely, assume that
 $\overline{\OP{O}(L_1)}=\OP{O}(q(L_1))$.
Then any non-degenerate even lattice $L_2$
 with $q(L_2)\cong -q(L_1)$ satisfies
 the conditions (\ref{prop:criterion:uniq}) and
 (\ref{prop:criterion:surj})
 by Lemma \ref{lem:unimodembed}.
For example, we can take $L_1(-1)$ as $L_2$.
\end{proof}

\subsection{Proof of Theorem \ref{thm:surjcoinv}}

Now we apply Proposition \ref{prop:criterion}
 to prove Theorem \ref{thm:surjcoinv}.
Let $G_0\in\MC{L}$ with $c(G_0)\geq 17$.
%We may assume that $G_0\in\MC{L}_\clos$.
% $\OP{O}_0(\Lambda_{G_0})=G_0$.
By Corollary \ref{cor:qcloslambda},
 $\Lambda_{G_0}\cong S(\MF{G}_n,q_n)$
 for some $(\MF{G}_n,q_n)\in\MC{Q}_\clos$.
Since $n\neq 58$ (see Table \ref{subsect:trees}), we have
\begin{equation} \label{condition_gn}
 \Lambda_{G_0}\cong S(\MF{G}_n,q_n)\cong N_G,~
 ([G],q(N_G))\sim (\MF{G}_n,q_n)\in\MC{Q}_\clos
\end{equation}
 for some $(G,N)\in\MC{N}'$
% with $([G],q(N_G))\sim (\MF{G}_n,q_n)$
 by Proposition \ref{prop:listuniq}.
%By Lemma \ref{lem:kondoembed},
% $\Lambda_{G_0}\cong N_G$ for some $(G,N)\in\MC{N}'$
% such that $[G_0]=[G]$.
%Moreover,
% we may assume that $N$ is of type $A_1^{\oplus 24}$
% by Proposition \ref{prop:listuniq}.
Since $c(\MF{G}_3)=12<17$, $N$ is of type $A_1^{\oplus 24}$
 by Proposition \ref{prop:listuniq}.
To prove Theorem \ref{thm:surjcoinv},
 it is sufficient to show that
 the conditions (\ref{prop:criterion:uniq})
 and (\ref{prop:criterion:surj}) in
 Proposition \ref{prop:criterion} are satisfied
 for $L_1=N_G$ and $L_2=N^G$
 if and only if $n\neq 48,51$.

We check that
 for $(G,N)\in\MC{N}'$ satisfying the conditions
 (\ref{condition_gn}),
 the condition (\ref{prop:criterion:uniq})
 is satisfied
 as follows:
Let $N'\subset (N_G)^\vee\oplus(N^G)^\vee$ be
 a Niemeier lattce which contains
 $N_G$ and $N^G$ primitively.
By Lemma \ref{lem:unimodaction},
 the action of $G$ on $N_G$ is extended
 to that on $N'$ such that $(N')^G=N^G$.
We have $(G,N')\in\MC{N}$ by Definition \ref{defn:mcn}.
By Proposition \ref{prop:listuniq},
 $(G,N)\cong(G,N')$.
The uniqueness of $N$ is shown.

Before showing the condition (\ref{prop:criterion:surj}),
 we prepare for a couple of lemmas.

\begin{lem} \label{lem:restriction}
For $(G,N)\in\MC{N}'$ satisfying the conditions
 (\ref{condition_gn}),
 let $\pi$ denote the restriction map
\begin{equation}
 \pi:\OP{O}(N,N^G)\rightarrow\OP{O}(N^G).
\end{equation}
Then we have $\OP{Ker}(\pi)=G$.
In particular, $G\triangleleft \OP{O}(N,N^G)$.
%Let $\Gamma$ be an even unimodular lattice
% and $L_1$ a primitive non-degenerate sublattice
% of $\Gamma$.
%Then the kernel of the restriction map
% $\pi:\OP{O}(\Gamma,(L_1)^\bot_\Gamma)
% \rightarrow\OP{O}((L_1)^\bot_\Gamma)$
% coincides with
% $O_0(L_1)$, which is considerd as a subgroup of
% $\OP{O}(\Gamma,(L_1)^\bot_\Gamma)$.
\end{lem}
\begin{proof}
Clearly, we have $G\subset \OP{Ker}(\pi)$.
Let $g\in\OP{Ker}(\pi)$.
Then $g|_{N_G}\in\OP{O}_0(N_G)$
 by Lemma \ref{lem:unimodaction}(3).
Since $(\MF{G}_n,q_n)\in\MC{Q}_\clos$, i.e.,
 $\OP{Clos}(\MF{G}_n,q_n)=(\MF{G}_n,q_n)$,
 we have $g\in G$
 (see Definition \ref{defn:closgq}).
Hence $\OP{Ker}(\pi)\subset G$.
%By Lemma \ref{lem:unimodaction}(1),
% we have $\OP{O}_0(L_1)\subset\OP{Ker}(\pi)$.
%The converse follows from
% Lemma \ref{lem:unimodaction}(2).
\end{proof}

Let $\Delta^+$ be a set of positive roots of $N$
 which is stable under the action of $G$
 (see Subsection \ref{subsect:factniemeier}).
Since $N$ is of type $A_1^{\oplus 24}$,
 $\OP{O}(N,\Delta^+)$ is isomorphic to
 the Mathieu group $M_{24}$ of degree $24$ and
 the Weyl group $W(N)$ of $N$ is isomorphic to
 $C_2^{24}$.
We have $\OP{O}(N) = W(N)\rtimes M_{24}$.
 
\begin{lem} \label{lem:semidirect}
%Let $H\subset\OP{O}(N)$ be the stabilizer subgroup of
% $N^G$.
For $(G,N)\in\MC{N}'$ satisfying the conditions
 (\ref{condition_gn}),
 we have
\begin{equation}
 \OP{O}(N,N^G)
 =
 C_2^m\rtimes N_{M_{24}}(G),
\end{equation}
% where $n$ is the number of orbits of
% $R(N,\Delta^+)$ under the action of $G$
 where $m=\rank N^G=24-c(G)$
 and $N_{M_{24}}(G)$ is the normalizer subgroup
 of $G$ in $M_{24}$.
In particular, we have
 $\left|\OP{O}(N,N^G)\right|
 =2^m \left|N_{M_{24}}(G)\right|$
\end{lem}
\begin{proof}
Set $\{ v_1,\ldots,v_{24} \}=R(N,\Delta^+)$
 and $W'=\OP{O}(N,N^G)\cap W$.
The action of $G$ decomposes $R(N,\Delta^+)$ into
 $n$ orbits $O_1,\ldots,O_m$.
The invariant lattice $N^G$ is generated by
 $\sum_{v\in O_j} v$ ($j=1,\ldots,m$) over $\Q$.
Let $w \in W$.
Then $w$ is of the form
\begin{equation}
 w=\prod_{i=1}^{24} T(v_i)^{e_i},\quad e_i\in\{0,1\},
\end{equation}
 where $T(v)$ is the reflection of $v$.
Since
\begin{equation}
 w \cdot \sum_{i=1}^{24} a_i v_i =
 \sum_{i=1}^{24} (-1)^{e_i} a_i v_i, \quad
 a_i\in\Q,
\end{equation}
 $W'$ is generated by $\prod_{v\in O_j} T(v)$ ($j=1,\ldots,m$),
 thus $W'\cong C_2^m$.
By Lemma \ref{lem:restriction},
 we have an injection
 $\iota:\OP{O}(N,N^G)/W'\rightarrow N_{M_{24}}(G)$.
For $g\in N_{M_{24}}(G)$, we have
 $gG\cdot v_i=Gg\cdot v_i$.
Therefore, for any $j$, we have
 $g\cdot O_j=O_{j'}$ for some $j'$.
Hence we have $N_{M_{24}}(G)\subset\OP{O}(N,N^G)$,
 and $\iota$ is an isomorphism.
The assertion follows from this.
\end{proof}

Now we show that
 for $(G,N)\in\MC{N}'$ satisfying the conditions
 (\ref{condition_gn}),
 the condition (\ref{prop:criterion:surj})
 is satisfied.
By Lemma \ref{lem:semidirect},
 we can determine the order of $\OP{O}(N,N^G)$
 from the order of $N_{M_{24}}(G)$.
We can compute the order of $N_{M_{24}}(G)$
 by using GAP \cite{GAP}.
On the other hand,
 we can also determine
 the order of $\OP{O}(N^G)$ as follows.

Let $B=(b_{ij})\in M_m(\Z)$
 be the Gramian matrix of $N^G$.
Then $\OP{O}(N^G)$ is identified with
 the matrix group $H$
 consisting of $P\in M_m(\Z)$ such that ${}^tPBP=B$.
Let $S$ denote the set consisting of
 column vectors $v\in\Z^m$
 such that ${}^t vBv=b_{ii}$
 for some $i$.
Then any element $P\in H$ is of the form
 $(v_1 \cdots v_m)$ with $v_i\in S$.
Since $N^G$ is negative definite,
% there exists a positive number $\lambda$ such that
% $-M-\lambda\cdot 1_n$ is positive definite.
%For any $v=(v_j)\in S$, we can see that
% $|v_j|<\left(\max\{|b_{ii}|\}/\lambda\right)^{1/2}$
% for all $j$.
 we can enumerate all elements in $S$ and $H$
 in finite steps.
Practically, we take $B$ with
 smaller $|b_{ii}|$
 (cf.\ the reduction theory of quadratic forms).
%Also, we should take larger $\lambda$.
%For example, we start with $\lambda=1$.
%If $-M- 1_n$ is not positive definite,
% then we try $\lambda=99/100,98/100,\cdots$.
%Finally, we get $\lambda$ such that
% $-M-\lambda\cdot 1_n$ is positive definite.
Since the rank of $N^G$ is less than or equal to $24-17=7$
 by the assumption of the theorem,
 we can determine the order of
 $\OP{O}(N^G)$ in practical time by this method.
The author used Maxima \cite{maxima} for this computation.
%The result (for $c(G)\geq 17$) is as follows.
%\begin{equation*}
%\begin{array}{c|c|c|c}
%n & 24-c(G) & |N_{M_{24}}(G)/G| & |\OP{O}(N^G)| \\
%\hline
%12 & 7 & 48 & 6144 \\
%26 & \cdots & \\
%32 & &
%\end{array}
%\end{equation*}
%From the above table, we have the following:
The result is the following:

\begin{lem} \label{lem:order}
%For $(G,N)\in\MC{N}$ such that
% $G=\OP{O}_0(N_G)$, $c(G)\geq 17$,
% $[G]\neq\MF{G}_{48},\MF{G}_{51}$ and
% $N$ is of type $A_1^{\oplus 24}$,
For $(G,N)\in\MC{N}'$ satisfying the conditions
 (\ref{condition_gn}),
 we have
 $[\OP{O}(N,N^G):G]=\left|\OP{O}(N^G)\right|$
 if and only if
 $[G]\neq \MF{G}_{48},\MF{G}_{51}$.
\end{lem}

For example, we consider the case $n= 80$
 ($[G]=\MF{G}_{80}=F_{384}$).
There exists exactly one element $(G,N)\in\MC{N}$
 such that $[G]=F_{384}$.
The Niemeier lattice $N$ is of type $A_1^{\oplus 24}$.
We have $[N_{M_{24}}(G):G]=2$
 and $\left|\OP{O}(N^G)\right|=64$.
Since $c(G)=19$, we have
 $\left|\OP{O}(N^G)\right|
 =[\OP{O}(N,N^G):G]=2^{24-19}\cdot 2=64$
 by Lemma \ref{lem:semidirect}.

Similarly,
 for other cases except $n\neq 48,51$,
 we have $[\OP{O}(N,N^G):G]=\left|\OP{O}(N^G)\right|$.
The following is the table of $k(G)=[N_{M_{24}}(G):G]$.
\begin{equation*}
\begin{array}{c|ccccccccccccc}
\hline
n    & 12 & 26 & 32 & 33 & 34 & 39
     & 40 & 46 & 49 & 54 & 55 & 56 & 61 \\
\hline
k(G) & 48 &  4 &  2 &  6 &  8 &  2
     & 24 &  4 &120 &  2 &  6 & 12 &  2 \\
\hline
\hline
n    & 62 & 63 & 65 & 70 & 74 & 75
     & 76 & 77 & 78 & 79 & 80 & 81 \\
\hline
k(G) &  2 &  6 & 24 &  1 &  2 & 24
     &  2 &  4 &  4 &  2 &  2 & 24 \\
\hline
\end{array}
\end{equation*}

On the other hand,
 we have $[\OP{O}(N,N^G):G]<\left|\OP{O}(N^G)\right|$
 for the cases $n=48,51$, as follows.
\begin{equation*}
\begin{array}{c|cc}
\hline
n               & 48 & 51 \\
\hline
k(G)            &  2 &  2 \\
\hline
\OP{O}(N^G)/2^m &  6 &  6 \\
\hline
\end{array}
\end{equation*}

We shall finish the proof
 of Theorem \ref{thm:surjcoinv}.
We already checked that the condition (1) is satisfied.
By Lemma \ref{lem:restriction},
 the restriction map
 $\pi:\OP{O}(N,N^G)\rightarrow\OP{O}(N^G)$
 induces an injection
 $\OP{O}(N,N^G)/G\hookrightarrow\OP{O}(N^G)$.
By Lemma \ref{lem:order}, this map
 is an isomorphism if and only if $n\neq 48,51$.
Therefore, the condition (2), i.e.\ the surjectivity of $\pi$
 is satisfied if and only if $n\neq 48,51$.
By Proposition \ref{prop:criterion},
%Now we have checked the conditions
% (\ref{prop:criterion:uniq}) and
% (\ref{prop:criterion:surj}),
 $\overline{\OP{O}(N_G)}=\OP{O}(q_{N_G})$
 if and only if $n\neq 48,51$.

\section{Uniqueness of invariant lattices $\Lambda^G$}
\label{sect:invariantuniq}

This section is devoted to prove the following:

\begin{prop} \label{prop:uniqueinvariant}
Set $E= \{ \MF{S}_5, L_2(7), \MF{A}_6 \}$.
For $(\MF{G},q)\in\MC{Q}$ (see (\ref{mcq})),
 we have
\begin{align}
\sharp \left( \{
 \Lambda^{G} \bigm|
 G\in\MC{L},[G]=\MF{G},q(\Lambda_G) \cong q \}
 /\text{\rm isom}
 \right) %\\
=
\begin{cases}
2 & \text{if } \MF{G}
 \in E, \\
% \MF{S}_5, L_2(7), \MF{A}_6, \\
1 & \text{otherwise.} % \MF{G}  =
% \MF{S}_5, L_2(7), \MF{A}_6 \\
\end{cases}
\end{align}
The Gramian matrices of $\Lambda^G$
 are given in Table \ref{subsect:invariant}.
\end{prop}
\begin{proof}
Let $G\in\MC{L}$ such that
 $[G]=\MF{G}$ and $q(\Lambda_G)\cong q$.
By Lemma \ref{lem:unimodembed},
 $q(\Lambda^G)\cong -q(\Lambda_G)\cong -q$.

First we consider the case $\rank \Lambda^G > 3$.
Since $\OP{sign}\Lambda=(3,19)$ and $\Lambda_G$
 is negative definite,
 $\Lambda^G$ is indefinite in this case.
From Table \ref{subsect:invariant},
 we can check that the conditions
 (\ref{thm:nikulinuniq:odd}) and
 (\ref{thm:nikulinuniq:even})
 in Theorem \ref{thm:nikulinuniq}
 for $\Lambda^G$ are satisfied.
Hence the assertion follows
 from Theorem \ref{thm:nikulinuniq}.
We can directly find the Gramian matrices of $\Lambda^G$
 with the given signature and discriminant form
 for each case.

Next we consider the case $\rank \Lambda^G = 3$.
In this case, $\Lambda^G$ is positive definite.
From the table of definite ternary
 forms \cite{schiemann}, we can check that there exists
 a unique positive definite even lattice $K$
 of rank $3$
 such that $q(K)\cong -q$
 up to isomorphism, except for the cases
 $\MF{G} = \MF{S}_5,L_2(7),\MF{A}_6$.
If $\MF{G}=\MF{S}_5,L_2(7),\MF{A}_6$,
 there exist exactly two positive definite
 even lattices $K_1,K_2$ of rank $3$
 such that $q(K_i)\cong -q$ up to isomorphism.
For each $i=1,2$,
 there exists a primitive embedding
 $\Lambda_G\rightarrow\Lambda$
 such that $(\Lambda_G)^\bot_\Lambda\cong K_i$
 by Lemma \ref{lem:unimodembed}.
By Lemma \ref{lem:unimodaction},
 the action of $G$ on $\Lambda_G$ is extended to
 that on $\Lambda$ such that
 $\Lambda^{G}\cong K_i$.
This action is an element in $\MC{L}$
 by Definition \ref{defn:sympgrp}.
Therefore, the assertion follows.
\end{proof}

\section{Property %
 $\overline{\OP{O}(\Lambda^G)}=\OP{O}(q(\Lambda^G))$}
\label{sect:invariantsurj}

This section is devoted to prove the following:

\begin{thm} \label{thm:surj}
Let $G\in\MC{L}$.
If $\rank \Lambda^G\geq 4$,
 or equivalently, $c(G)\leq 18$
 (see Proposition \ref{prop:trace}), then
 $\overline{\OP{O}(\Lambda^G)}=\OP{O}(q(\Lambda^G))$.
\end{thm}

We may assume that $G\in\MC{L}_\clos$
 by replacing $G$ by $\OP{Clos}(G)$ if necessary.
Then $\Lambda_G\cong S(\MF{G}_n,q_n)$
 for some $(\MF{G}_n,q_n)\in\MC{Q}_\clos$
 (see Section \ref{sect:coinvariantuniq}).
We can check that $\Lambda^G$
 satisfies the conditions
 (\ref{thm:nikulinsurj:1}) and (\ref{thm:nikulinsurj:2})
 in Theorem \ref{thm:nikulinsurj}
 from Table \ref{subsect:invariant},
 except for the following nine cases:
\begin{equation} \label{nineexceptions}
 n=
  26,  30,  32,  33,
  40,  46,  48,  56,
  61.
\end{equation}
Hence we have
 $\overline{\OP{O}(\Lambda^G)}=\OP{O}(q(\Lambda^G))$
 except for these nine cases.

For example, in the case $n= 65$, we find that
\begin{gather}
 \Lambda^G \cong
 \begin{pmatrix} 4 & 2 \\ 2 & 4 \end{pmatrix} \oplus
 \LF{4} \oplus \LF{-8}, \\
% A(\Lambda^G)\cong (\Z/2\Z)^{\oplus 2} \oplus
% \Z/4\Z \oplus \Z/8\Z \oplus \Z/3\Z, \\
 q(\Lambda^G) \cong -q_{65} \cong
 v^{(2)}(2) \oplus q^{(2)}_1(4) \oplus q^{(2)}_7(8)
 \oplus q^{(3)}_+(3) \label{decomp65}
\end{gather}
 from Table \ref{subsect:invariant}.
Since
\begin{equation}
 \rank \Lambda^G=4>l(A(\Lambda^G)_3)+2=3,
\end{equation}
 the condition (\ref{thm:nikulinsurj:1}) is satisfied.
On the other hand, since $v^{(2)}(2)$ appears in
 the orthogonal decomposition (\ref{decomp65})
 of $q(\Lambda^G)$,
 the condition (\ref{thm:nikulinsurj:2}) is satisfied.

\subsection{Preparation
 for the cases (\ref{nineexceptions})}

Before studying the cases (\ref{nineexceptions}),
 we recall some properties of
 the spinor norm (see e.g.\ \cite{cassels78}).
Let $L$ be a non-degenerate lattice.
For any $\varphi\in\OP{O}(L\otimes\Q)$,
 $\varphi$ is written as a composition of reflections:
\begin{equation} \label{varphit}
 \varphi=\prod_{i=1}^{r} T(v_i), \quad
 v_i\in L\otimes\Q,~ \BF{v_i}{v_i}\neq 0.
\end{equation}
Here $T(v)\in\OP{O}(L\otimes\Q)$ is the reflection of $v$,
 which is defined by
\begin{equation}
 T(v)\cdot w=w-\frac{2\BF{v}{w}}{\BF{v}{v}} v.
\end{equation}
The spinor norm $\theta(\varphi)$ of $\varphi$
 is defined by
\begin{equation}
 \theta(\varphi)=\prod_{i=1}^{r} \BF{v_i}{v_i}
 \bmod (\Q^\times)^2
 \in \Q^\times/(\Q^\times)^2,
\end{equation}
 which is independent of the choice of
 the expression (\ref{varphit}).
We define a map $f$ and
 a subgroup $\OP{O}'(L)\subset \OP{O}(L)$ by
\begin{equation}
 f=\det\times\theta: \OP{O}(L)\rightarrow
 \{ \pm 1 \} \times
 \Q^\times/(\Q^\times)^2
\end{equation}
 and $\OP{O}'(L)=\OP{Ker}(f)$.
Note that if $L=L_1\oplus L_2$, then
 $f(\OP{O}(L_i))\subset f(\OP{O}(L))$.
We can define the spinor norm
 $\theta_p(\varphi_p)\in\Q_p^\times/(\Q_p^\times)^2$
 of $\varphi_p\in\OP{O}(L\otimes\Q_p)$
 in a similar way.
Moreover, we define
\begin{equation}
 f_p=\det\times\theta_p: \OP{O}(L_p)\rightarrow
 \{ \pm 1 \} \times
 \Q_p^\times/(\Q_p^\times)^2
\end{equation}
 and $\OP{O}'(L_p)=\OP{Ker}(f_p)$,
 where $L_p=L\otimes\Z_p$.

To deal with the cases (\ref{nineexceptions}),
 we use the following proposition,
 which is a consequence of Strong Approximation Theorem
 of quadratic forms (cf.\ \cite{cassels78}).

\begin{prop}
Let $L$ be an indefinite even lattice of rank $\geq 3$.
We set $\OP{O}_0(L_p)=\OP{Ker} \bigl( \OP{O}(L_p)\rightarrow
 \OP{O}(q(L_p)) \bigr)$ and $d=\disc(L)$.
If the natural map
\begin{equation} \label{mapquot}
 \OP{O}(L)\rightarrow\prod_{p|d}
 \frac{f_p(\OP{O}(L_p))}{f_p(\OP{O}_0(L_p))}
\end{equation}
 is surjective, then
 $\overline{\OP{O}(L)}=\OP{O}(q(L))$.
\end{prop}
\begin{proof}
We have a natural commutative diagram
\begin{equation} \label{snake}
\begin{array}{ccccccccc}
 1 & \rightarrow & \OP{O}'(L) & \rightarrow & \OP{O}(L)
 & \rightarrow
% & \displaystyle \prod_{p|d} f_p(\OP{O}(L_p))
 & f(\OP{O}(L))
 & \rightarrow & 1 \\
&&~\downarrow \alpha&&~\downarrow \beta&
 &~\downarrow \gamma&& \\
 1 & \rightarrow
 & \displaystyle \prod_{p|d}
 \frac{\OP{O}'(L_p)}{\OP{O}'_0(L_p)}
 & \rightarrow
 & \displaystyle \prod_{p|d}
 \frac{\OP{O}(L_p)}{\OP{O}_0(L_p)}
 & \rightarrow
 & \displaystyle \prod_{p|d}
 \frac{f_p(\OP{O}(L_p))}{f_p(\OP{O}_0(L_p))}
 & \rightarrow & 1
\end{array}
\end{equation}
 where
 $\OP{O}'_0(L_p)=\OP{O}'(L_p) \cap \OP{O}_0(L_p)$.
The rows in (\ref{snake}) are exact.
% $\OP{O}'(q(L)):=\OP{Ker}(\beta)$.
Since
\begin{equation}
 \OP{O}(q(L))=\prod_{p|d} \OP{O}(q(L)_p)
 \cong \prod_{p|d}
 \frac{\OP{O}(L_p)}{\OP{O}_0(L_p)}
\end{equation}
 by Theorem \ref{thm:localsurj},
 it is sufficient to show that $\beta$ is surjective.
Since $[\OP{O}'(L_p):\OP{O}'_0(L_p)]<\infty$,
 each coset of $\OP{O}'(L_p)/\OP{O}'_0(L_p)$ is
 open dense subset of $\OP{O}'(L_p)$
 in $p$-adic topology.
By Strong Approximation Theorem of quadratic forms
 (cf.\ \cite{cassels78}),
 the image of $\OP{O}'(L)$ in $\prod_{p|d} \OP{O}'(L_p)$
 is dense.
Therefore, $\alpha$ is surjective.
On the other hand, $\gamma$ is surjective
 by the assumption.
By chasing the diagram, $\beta$ is surjective.
\end{proof}

For $f(\OP{O}(L))$ and $f_p(\OP{O}_0(L_p))$,
 we have the following:

\begin{lem} \label{lem:ozero}
Let $L^{(p)}$ be a non-degenerate even lattice
 over $\Z_p$.
\begin{enumerate}
\renewcommand{\labelenumi}{(\arabic{enumi})}
\item \label{lem:ozero:line}
 If $v\in L^{(p)}$ satisfies
 $a=\BF{v}{v}\in \Z_p^\times \cup 2 \Z_p^\times$,
 then $T(v)\in\OP{O}_0(L^{(p)})$ and
 $f_p(T(v))=(-1,\overline{a})\in f_p(\OP{O}_0(L_p))$.
\item \label{lem:ozero:hyp}
 If $L^{(p)}$ contains
 $U=\left(
 \begin{smallmatrix}0&1\\1&0\end{smallmatrix}
 \right)$
 as a sublattice, then
\begin{equation} \label{subgroupu}
 f_p(\OP{O}_0(L^{(p)})) \supset
 \begin{cases}
 J_2:=
 \langle (1,\Z_2^\times/(\Z_2^\times)^2),
 (-1,\overline{2}) \rangle
 & \text{if } p=2, \\
 J_p:=
 \{ \pm 1 \} \times \Z_p^\times/(\Z_p^\times)^2
 & \text{otherwise}.
 \end{cases}
\end{equation}
\item \label{lem:ozero:nonhyp}
 If $p=2$ and $L^{(2)}$ contains
 $V=\left(
 \begin{smallmatrix}2&1\\1&2\end{smallmatrix}
 \right)$
 as a sublattice, then
\begin{equation} \label{subgroupv}
 f_2(\OP{O}_0(L^{(2)})) \supset
 J_2.
% \langle (1,\Z_2^\times/(\Z_2^\times)^2),
% (-1,\overline{2}) \rangle.
\end{equation}
\end{enumerate}
\end{lem}
\begin{proof}
Let $v,a$ be as in (\ref{lem:ozero:line}).
Since $T(v)\cdot w=w-(2\BF{v}{w}/a)v$ and
 $2/a\in\Z_p^\times$,
 we have $T(v)\cdot w\in L^{(p)}$ for $w\in L^{(p)}$.
Hence $T(v)\in\OP{O}(L^{(p)})$.
If $w\in (L^{(p)})^\vee$, then
 $\BF{v}{w}\in\Z_p$, thus
 $T(v)\cdot w\equiv w \bmod L^{(p)}$.
Hence $T(v)\in\OP{O}_0(L^{(p)})$.
Since the determinant of any reflection is euqal to $-1$,
 we have $f_p(T(v))=(-1,\overline{a})$.
This proves (\ref{lem:ozero:line}).

Let $(e_1,e_2)$ be a basis of $U$ such that
 $\BF{e_i}{e_i}=0$ and $\BF{e_1}{e_2}=1$.
For $x\in\Z_p^\times$, set $v_x=e_1+x e_2$.
We have $\BF{v_x}{v_x}=2x\in 2\Z_p^\times$.
By (\ref{lem:ozero:line}),
 $T(v_x)\in\OP{O}_0(L^{(p)})$ and
 $f_p(T(v_x))=(-1,\overline{2x})$.
We can check that the group generated by
 elements of the form $(-1,\overline{2x})$ is
 $J_2$ (resp.\ $J_p$) if $p=2$ (resp.\ $p\neq 2$).
% the right hand side of (\ref{subgroupu}).

The proof of (\ref{lem:ozero:nonhyp}) is
 similar to (\ref{lem:ozero:hyp}),
 and we omit it.
\end{proof}

\begin{lem} \label{lem:ogrobal}
Let $L$ be a non-degenerate even lattice.
\begin{enumerate}
\renewcommand{\labelenumi}{(\arabic{enumi})}
\item \label{lem:ogrobal:negate}
$f(-1_L)=((-1)^{\rank L},\overline{\disc(L)})$.
\item \label{lem:ogrobal:hyp}
If $L\cong U(t)\oplus L'$ for some $L'$, then
$f(\OP{O}(L))
 \supset \langle (-1,\pm \overline{2t}) \rangle$,
 where
 $U(t)=\left(
 \begin{smallmatrix}0&t\\t&0\end{smallmatrix}
 \right)$.
\end{enumerate}
\end{lem}
\begin{proof}
Let $(e_1,\ldots,e_r)$ be an orthogonal basis
 of $L\otimes\Q$, where $r=\rank L$.
Then, $-1_L=\prod_{i=1}^{r} T(e_i)$ and
 $\prod_{i=1}^{r} \BF{e_i}{e_i} \equiv \disc(L)
 \bmod (\Q^\times)^2$.
Therefore, $f(-1_L)=((-1)^r,\overline{\disc(L)})$.
This proves (\ref{lem:ogrobal:negate}).

Let $(e_1,e_2)$ be a basis of $U(t)$ such that
 $\BF{e_i}{e_i}=0$ and $\BF{e_1}{e_2}=t$.
Then, $\OP{O}(U(t))\cong (\Z/2\Z)^2$ is generated by
 $T(e_1 \pm e_2)$.
Therefore, $f(\OP{O}(U(t)))
 =\langle (-1,\pm \overline{2t}) \rangle$.
This proves (\ref{lem:ogrobal:hyp}).
\end{proof}

\subsection{Proof of Theorem \ref{thm:surj} %
 for the cases (\ref{nineexceptions})}

We set $L=\Lambda^G$, $r=\rank L$
 and $d=\disc(L)$.
We shall show that the map (\ref{mapquot}) is
 surjective in each case
 in (\ref{nineexceptions}).
In other words,
 we show that $\prod_{p|d} f_p(\OP{O}(L_p))$
 is generated by the images of $\OP{O}(L)$ and
 $\prod_{p|d} f_p(\OP{O}_0(L_p))$.
As is shown below,
 we have $f_p(\OP{O}(L_p))=N_p$ except fot
 the cases $n= 46,61$, where
\begin{equation}
 N_p =
 \{ \pm 1 \} \times \Q_p^\times/(\Q_p^\times)^2.
\end{equation}
Recall that the map
 $(a,b,c)\mapsto (-1)^a 3^b 2^c$ induces
 an isomorphism $(\Z/2\Z)^{3}\rightarrow
 \Q_2^\times/(\Q_2^\times)^2$.
Moreover, the map
 $(a,b)\mapsto \varepsilon_p^a p^b$
 induces an isomorphism
 $(\Z/2\Z)^{2}\rightarrow
 \Q_p^\times/(\Q_p^\times)^2$ if $p\neq 2$,
 where $\varepsilon_p$ is a non-square $p$-adic unit.
Let $(e_1,\ldots,e_r)$ be a basis of $L$
 whose Gramian matrix is given
 by Table \ref{subsect:invariant}.
We say $a$ is represented by $L$ if
 there exists a vector $v\in L$ such that
 $\BF{v}{v}=a$.
We denote $f(\OP{O}(L))$ and $f_p(\OP{O}_0(L_p))$ by
 $I$ and $I_p$, respectively.

%We shall show that the map (\ref{mapquot})
%%\begin{equation} \label{mapfs}
%% \OP{O}(L)\rightarrow\prod_{p|d}
%% \frac{ \{ \pm 1 \} \times \Q_p^\times/(\Q_p^\times)^2}
%% {f_p(\OP{O}_0(L_p))}
%%\end{equation}
% is surjective for $L=\Lambda^G$ in each case
% in (\ref{nineexceptions}).

(1)
The case $n= 26$.
We have
\begin{equation}
 L\cong
 \begin{pmatrix}0&8\\8&0\end{pmatrix}
 \oplus
 \LF{2}\oplus\LF{4},
 \quad d=-2^9.
\end{equation}
Since $2$ and $6$ are represented by $L$, we have
 $(-1,\overline{2}),(-1,\overline{6})\in I_2$
 by Lemma \ref{lem:ozero}(\ref{lem:ozero:line}).
By Lemma \ref{lem:ogrobal}(\ref{lem:ogrobal:hyp}),
 $(-1,\pm\overline{16}) = (-1,\pm\overline{1})
 \in I$.
We can check that the images of these four
 elements generate $N_2$.
(In what follows, we omit ``the image(s) of''
 for simplicity.)

(2)
The case $n= 30$.
We have
\begin{equation}
 L\cong
 \begin{pmatrix}0&3\\3&0\end{pmatrix}^{\oplus 2}
 \oplus
 \begin{pmatrix}2&3\\3&0\end{pmatrix},
 \quad d=-3^6.
\end{equation}
By Lemma \ref{lem:ogrobal}(\ref{lem:ogrobal:hyp}),
 $(-1,\pm \overline{6})\in I$.
Since $T(e_5)\in\OP{O}(L)$, we have
 $f(T(e_5))=(-1,\overline{2})\in I$.
We can check that
 these three elements generate $N_3$.

(3)
The case $n= 32$.
We have
\begin{equation}
 L\cong
 \begin{pmatrix}0&5\\5&0\end{pmatrix}
 \oplus
 \begin{pmatrix}4&2\\2&6\end{pmatrix},
 \quad d=-2^2 \cdot 5^3.
\end{equation}
Since $L_2$ contains $U$,
 we have $J_2\subset I_2$
% $\langle (1,\Z_2^\times/(\Z_2^\times)^2),
% (-1,\overline{2})\rangle \subset I_2$
 by Lemma \ref{lem:ozero}(\ref{lem:ozero:hyp}).
Since $4$ is represented by $L$, we have
 $(-1,\overline{4}) = (-1,\overline{1}) \in I_5$
  by Lemma \ref{lem:ozero}(\ref{lem:ozero:line}).
By Lemma \ref{lem:ogrobal}(\ref{lem:ogrobal:hyp}),
 $(-1,\pm \overline{10})\in I$.
Since $T(e_3)\in\OP{O}(L)$, we have
 $f(T(e_1))=(-1,\overline{4}) = (-1,\overline{1})
 \in I$.
Let $L'=\left(
 \begin{smallmatrix}4&2\\2&6\end{smallmatrix}
 \right)$.
By Lemma \ref{lem:ogrobal}(\ref{lem:ogrobal:negate}),
 $f(-1_{L'})=(1,\overline{20})=(1,\overline{5})
 \in I$.
Therefore, the images of $I,I_2,I_5$ contain
 the follwoing elements.
\begin{equation*}
\begin{array}{c|c}
 & \text{image in $N_2 \times N_5$} \\
\hline
I_2 & (1,\Z_2^\times/(\Z_2^\times)^2)
      \times(1,\overline{1}),
      (-1,\overline{2})\times(1,\overline{1}) \\
I_5 & (1,\overline{1})\times(-1,\overline{1}) \\
I   & (-1,\pm \overline{10})
              \times(-1,\pm \overline{10}),
      (-1,\overline{1}) \times (-1,\overline{1}),
      (1,\overline{5})\times(1,\overline{5})
\end{array}
\end{equation*}
From this, we can check that
 $I,I_2,I_5$ generate $N_2\times N_5$.

(4)
The case $n= 33$.
We have
\begin{equation}
 L\cong
 \begin{pmatrix}0&7\\7&0\end{pmatrix}
 \oplus
 \begin{pmatrix}2&1\\1&4\end{pmatrix},
 \quad d=-7^3.
\end{equation}
By Lemma \ref{lem:ogrobal}(\ref{lem:ogrobal:hyp}),
 $(-1,\pm \overline{14})\in I$.
Since $T(e_3)\in\OP{O}(L)$, we have
 $(-1,\overline{2})\in I$.
We can check that these three elements
 generate $N_7$.

(5)
The case $n= 40$.
We have
\begin{equation}
 L\cong \LF{4}^{\oplus 3} \oplus
 \LF{-4}^{\oplus 2},
 \quad d=2^{10}.
\end{equation}
Let $\varphi=T(e_1)T(e_1+2e_2)\in\OP{O}(L_2)$.
Then, modulo $L_2$, we have
\begin{align}
 \varphi\cdot \frac{e_1}{4}
 &=T(e_1)\cdot \left(
 \frac{e_1}{4} - \frac{2}{20}(e_1+2e_2)
 \right)
 \equiv T(e_1)\cdot \frac{3}{4} e_1
 \equiv \frac{e_1}{4}, \\
 \varphi\cdot \frac{e_2}{4}
 &=T(e_1)\cdot \left(
 \frac{e_2}{4} - \frac{4}{20}(e_1+2e_2)
 \right)
 \equiv T(e_1)\cdot \frac{e_2}{4}
 = \frac{e_2}{4}.
\end{align}
Hence $\varphi\in\OP{O}_0(L_2)$ and
 $f_2(\varphi)=(-1,\overline{4})\cdot
 (-1,\overline{20})=(1,\overline{5})\in I_2$.
Since $T(e_1),T(e_4),T(e_1+e_2) \in \OP{O}(L)$, we have
 $(-1,\pm \overline{4}),(-1,\overline{8}) \in I$.
We can check that these four elements
 generate $N_2$.

(6)
The case $n= 46$.
We have
\begin{equation}
 L\cong\begin{pmatrix}2&1\\1&2\end{pmatrix}
 \oplus\LF{6}\oplus\LF{-18},
 \quad d=-2^2 \cdot 3^4.
\end{equation}
Since $L_2$ contains $V$, we have
 $J_2 \subset I_2$
 by Lemma \ref{lem:ozero}(\ref{lem:ozero:nonhyp}).
By Theorem 3.14(i) of \cite{earnesthsia}, we have
 $f_2(\OP{O}(L_2))= J_2$,
 thus $I_2=f_2(\OP{O}(L_2))=J_2$.
Since $T(e_1),T(e_3),T(e_4)\in\OP{O}(L)$, we have
 $(-1,\overline{2}),
 (-1,\overline{6}),(-1,-\overline{18})\in I$.
From this, we can check that
 $I,I_2$ generate $f_2(\OP{O}(L_2))\times N_3$.

(7)
The case $n= 48$.
We have
\begin{equation}
 L\cong
 \begin{pmatrix}0&3\\3&0\end{pmatrix}
 \oplus
 \begin{pmatrix}12&6\\6&12\end{pmatrix},
 \quad d=-2^2 \cdot 3^5.
\end{equation}
Since $L_2$ contains $U$, we have
 $J_2 \subset I_2$
 by Lemma \ref{lem:ozero}(\ref{lem:ozero:hyp}).
By Lemma \ref{lem:ogrobal}(\ref{lem:ogrobal:hyp}),
 $(-1,\pm \overline{6})\in I$.
Since $T(e_3),T(e_3+e_4)\in\OP{O}(L)$, we have
 $(-1,\overline{12}),(-1,\overline{36})\in I$.
Therefore, the images of $I,I_2$ contains
 the follwoing elements.
\begin{equation*}
\begin{array}{c|c}
 & \text{image in $N_2 \times N_3$} \\
\hline
I_2 & (1,\Z_2^\times/(\Z_2^\times)^2)
      \times(1,\overline{1}),
      (-1,\overline{2})\times(1,\overline{1}) \\
I   & (-1,\pm \overline{6})
              \times(-1,\pm \overline{6}),
      (-1,\overline{3}) \times (-1,\overline{3}),
      (-1,\overline{1})\times(-1,\overline{1})
\end{array}
\end{equation*}
From this, we can check that
 $I,I_2$ generate $N_2\times N_3$.

(8)
The case $n= 56$.
We have
\begin{equation}
 L\cong\LF{4}^{\oplus 3}\oplus\LF{-8},
 \quad d=-2^9.
\end{equation}
By the argument in the case $n= 40$,
 $\varphi=T(e_1)T(e_1+2e_2)\in\OP{O}_0(L_2)$
 and $f_2(\varphi)=(1,\overline{5})\in I_2$.
Since $T(e_1),T(e_4),T(e_1+e_2)\in\OP{O}(L)$, we have
 $(-1,\overline{4}),(-1,-\overline{8}),
 (-1,\overline{8}) \in I$.
We can check that these four elements
 generate $N_2$.

(9)
The case $n= 61$.
We have
\begin{equation}
 L\cong
 \begin{pmatrix}0&3\\3&0\end{pmatrix}
 \oplus
 \begin{pmatrix}8&4\\4&8\end{pmatrix},
 \quad d=-2^4 \cdot 3^3.
\end{equation}
Since $L_2$ contains $U$, we have
 $J_2 \subset I_2$
 by Lemma \ref{lem:ozero}(\ref{lem:ozero:hyp}).
By Theorem 3.14(i) of \cite{earnesthsia},
 $f_2(\OP{O}(L_2))= J_2$,
 thus $I_2=f_2(\OP{O}(L_2))=J_2$.
Since $T(e_3)\in\OP{O}(L)$,
 $(-1,\overline{8})=(-1,\overline{2})\in I$.
By Lemma \ref{lem:ogrobal}(\ref{lem:ogrobal:hyp}),
 $(-1,\pm \overline{6})\in I$.
From this, we can check that
 $I,I_2$ generate $f_2(\OP{O}(L_2))\times N_3$.

Now we have proved Theorem \ref{thm:surj}.

\section{Uniqueness of symplectic actions %
 on the $K3$ lattice}
\label{sect:uniqueness}

In this section,
 we use the results in the previous sections
% to study symplectic actions on the $K3$ lattice.
 to prove Main Theorem.

\subsection{Case $c(G)\leq 18$}

\begin{prop} \label{prop:case18}
The natural map
\begin{equation}
 \{ G\in\MC{L} \bigm|
 c(G)\leq 18 \} / \text{\rm conj} \rightarrow
 \{ (G,S)\in\MC{S} \bigm|
 c(G)\leq 18 \} / \text{\rm isom}
\end{equation}
 is bijective.
\end{prop}
\begin{proof}
The surjectivity follows from the definition
 of $\MC{S}$ (see (\ref{mcs})).
Let $(G,S)\in\MC{S}$ such that
 $c(G)\leq 18$.
Suppose that $G_i\in\MC{L}$ and
 $(G_i,\Lambda_{G_i}) \cong (G,S)$ for $i=1,2$.
To prove the injectivity,
 it is sufficient to show that
 $G_1$ and $G_2$ are conjugate in $\OP{O}(\Lambda)$.
By Proposition \ref{prop:uniqueinvariant},
 $\Lambda^{G_1} \cong \Lambda^{G_2}$.
%We distinguish two cases: $c(G)\leq 18$ and $c(G)=19$.
By Theorem \ref{thm:surj},
 $\overline{\OP{O}(\Lambda^{G_1})}
 =\OP{O}(q(\Lambda^{G_1}))$.
Therefore,
 a primitive embedding
 $\Lambda_{G_1}\rightarrow\Lambda$ such that
 $(\Lambda_{G_1})^\bot_\Lambda\cong\Lambda^{G_1}$
 is unique
 up to isomorphism and the restriction map
\begin{equation} \label{restrictionmap}
 \pi:\OP{O}(\Lambda,\Lambda_{G_1})
 \rightarrow\OP{O}(\Lambda_{G_1})
\end{equation}
 is surjective
 by Lemma \ref{lem:unimodembed}.
Hence we may assume that $\Lambda_{G_1}=\Lambda_{G_2}$
 by replacing $G_2$ by $\varphi G_2 \varphi^{-1}$
 for some $\varphi\in\OP{O}(\Lambda)$ if necessary.
Since $(G_1,\Lambda_{G_1})\cong (G_2,\Lambda_{G_2}) \cong (G,S)$,
 $G_1$ and $G_2$ are conjugate
 as subgroups of $\OP{O}(\Lambda_{G_1})$.
Since $\pi$
 is surjective, $G_1$ and $G_2$ are conjugate
 in $\OP{O}(\Lambda)$.
\end{proof}

\subsection{Case $c(G)=19$}

%In this subsection, we study $G\in\MC{L}$ such that
% $c(G)=19$.

\begin{lem} \label{lem:maxconj}
Let $G_1,G_2\in\MC{L}$ such that $[G_1]=[G_2]$,
  $\OP{Clos}(G_1)=\OP{Clos}(G_2)$ and $c(G_i)=19$.
If $[\OP{Clos}(G_i)]\neq \MF{A}_{4,4},F_{384}$, then
 $G_1$ and $G_2$ are conjugate in $\OP{Clos}(G_i)$.
\end{lem}
\begin{proof}
It is sufficient to consider the case
 $G_i\subsetneq \OP{Clos}(G_i)$.
By Tables \ref{subsect:qands} and \ref{subsect:trees},
 we find that
 $\MF{H}:=[\OP{Clos}(G_i)]=T_{48},H_{192},T_{192},M_{20}$.
Using GAP \cite{GAP}, we can check that
 there exists a unique subgroup $\MF{G}$ of $\MF{H}$
 up to conjugacy in $\MF{H}$ such that $\MF{G}=[G_i]$
 (see Appendix).
The assertion follows from this.
\end{proof}

Now we consider subgroups $\MF{G}$ of
 $\MF{A}_{4,4}$ or $F_{384}$ such that $c(\MF{G})=19$.
In \cite{mukai88}, Mukai constructed $K3$ surfaces
 with maximal finite symplectic actions.
We use two $K3$ surfaces with symplectic actions of
 $\MF{A}_{4,4}$ or $F_{384}$ from \cite{mukai88}.

Let $X$ be a surface in $\P^5$ defined
 by the following equations:
\begin{align}
x^2+y^2+z^2&=\sqrt{3}u^2, \\
x^2+\zeta y^2+\zeta^2 z^2&=\sqrt{3}v^2, \\
x^2+\zeta^2 y^2+\zeta z^2&=\sqrt{3}w^2,
\end{align}
 where $\zeta=\exp(2\pi\sqrt{-1}/3)$ and
 $x,y,z,u,v,w$ are homogeneous coordinates of $\P^5$.
Since $X$ is a smooth complete intersection
 of type $(2,2,2)$ in $\P^5$,
 $X$ is a $K3$ surface.
Let $G$ denote a subgroup of $\OP{PGL}(6,\C)$
 generated by
\begin{align}
(x:y:z:u:v:w)&\mapsto(-x:-y:z:u:v:w), \label{eq:minusxy} \\
(x:y:z:u:v:w)&\mapsto(x:y:z:-u:-v:w), \\
(x:y:z:u:v:w)&\mapsto(y:z:x:u:\zeta v:\zeta^2 w), \\
(x:y:z:u:v:w)&\mapsto(x:\zeta^2 y:\zeta z:v:w:u), \\
(x:y:z:u:v:w)&\mapsto(-x:-z:-y:u:w:v).
\end{align}
Then $G$ acts on $X$ symplectically and
 $[G]=\MF{A}_{4,4}$.
Moreover, let $\widetilde{G}$ denote the group
 generated by $G$ and
\begin{equation}
g:(x:y:z:u:v:w)\mapsto(u:v:w:x:z:y).
\end{equation}
Then $\widetilde{G}$ acts on $X$ and
 $g^* \omega_X=\sqrt{-1} \omega_X$.
Using GAP, we can show the following (see Appendix):

\begin{lem} \label{lem:a44}
Suppose that $\MF{G}\in\Gsymp$ is a subgroup of
 $\MF{A}_{4,4}$ and $c(\MF{G})=19$.
Then there exists a unique subgroup $K$ of $G$
 such that $[K]=\MF{G}$
 up to conjugacy in $\widetilde{G}$.
\end{lem}

Let $Y$ be a surface in $\P^3$ defined
 by the following equation:
\begin{equation}
x^4+y^4+z^4+t^4=0,
\end{equation}
 where
 $x,y,z,t$ are homogeneous coordinates of $\P^3$.
Since $Y$ is a smooth quartic surface in $\P^3$,
 $Y$ is a $K3$ surface.
Let $H$ denote a subgroup of $\OP{PGL}(4,\C)$
 generated by
\begin{align}
(x:y:z:t)&\mapsto(i x:-i y:z:t), \\
(x:y:z:t)&\mapsto(y:-x:z:t), \\
(x:y:z:t)&\mapsto(y:z:t:-x),
\end{align}
 where $i=\sqrt{-1}$.
Then $H$ acts on $Y$ symplectically and
 $[H]=F_{384}$.
Moreover, let $\widetilde{H}$ denote the group
 generated by $H$ and
\begin{equation}
h:(x:y:z:t)\mapsto(i x:y:z:t).
\end{equation}
Then $\widetilde{H}$ acts on $Y$ and
 $h^* \omega_Y=i \omega_Y$.
Again using GAP, we can show the following (see Appendix):

\begin{lem} \label{lem:f384}
Suppose that $\MF{G}\in\Gsymp$ is a subgroup of
 $F_{384}$ and $c(\MF{G})=19$.
Then there exists a unique subgroup $K$ of $H$
 such that $[K]=\MF{G}$
 up to conjugacy in $\widetilde{H}$.
\end{lem}

%\begin{rem}
%In GAP system, we can realize
% $\widetilde{G}$ and $\widetilde{H}$
% as quotients of permutation groups.
%For example, a subgroup of $\OP{PGL}(2,\C)$
% generated by $(x:y)\mapsto(\zeta x:y)$
% and $(x:y)\mapsto(y:x)$ is realized as
%\begin{equation}
%\langle (1~2~3),(1~4)(2~5)(3~6) \rangle /
%\langle (1~2~3)(4~5~6) \rangle.
%%\subset \MF{S}_6.
%\end{equation}
%\end{rem}

\begin{rem}
We can show that the projective automorphism groups
 of $X$ and $Y$ are $\widetilde{G}$ and $\widetilde{H}$,
 respectively (cf.\ \cite{hashi}).
However, since $X$ and $Y$ have Picard number $20$,
 the automorphism groups of $X$ and $Y$ are
 infinite groups by \cite{shiodainose}.
\end{rem}

By considering induced actions on
 $H^2(X,\Z)$ and $H^2(Y,\Z)$,
 which are isomorphic to $\Lambda$,
 we have the following:

\begin{lem} \label{lem:gh}
Consider $G$ (resp.\ $H$)
 as a subgroup of $\OP{O}(\Lambda)$.
Suppose that $\MF{G}$ is a subgroup of
 $\MF{A}_{4,4}$ (resp.\ $F_{384}$)
 such that $c(\MF{G})=19$.
Then there exists a unique subgroup $K$
 of $G$ (resp. $H$) up to conjugacy
 in $\OP{O}(\Lambda)$
 such that $[K]=\MF{G}$.
\end{lem}

We use the following lemma in the proof of
 Theorem \ref{thm:coinvariantuniq}.

\begin{lem} \label{lem:sect4}
There exists an element $G_{43}\in\MC{L}$
 which satisfies the following:
\begin{enumerate}
\renewcommand{\labelenumi}{(\arabic{enumi})}
\item
$[G_{43}]=\MF{G}_{43}$;
\item
There exists a unique subgroup $G_{58}$ of
 $\OP{O}_0(\Lambda_{G_{43}})$
 such that $[G_{58}]=\MF{G}_{58}$
 up to conjugacy
 in $\OP{O}(\Lambda_{G_{43}})$.
\end{enumerate}
\end{lem}
\begin{proof}
We fix an identification $H^2(Y,\Z)=\Lambda$.
By Table \ref{subsect:trees},
 there exists a subgroup $G_{43}$ of $H$
 such that $[G_{43}]=\MF{G}_{43}$.
Since $c(\MF{G}_{43})=c(H)=19$,
 we have $\Lambda_{G_{43}}=\Lambda_H$.
Since $[H]=F_{384}$ is
 a maximal element in $\Gsymp$,
 we have $[\OP{O}_0(\Lambda_H)]=[H]$.
Since $H \triangleleft \widetilde{H}$,
 we have $\widetilde{H}\subset
 \OP{O}(\Lambda,\Lambda^H)$.
By Lemma \ref{lem:f384} and Table \ref{subsect:trees},
 the condition (2) is satisfied.
\end{proof}

We have the following by the above lemmas.

\begin{prop} \label{prop:case19}
Set $E=\{\MF{S}_5,L_2(7),\MF{A}_6\}\subset\Gsymp$.
%Let $\MF{G}\in\Gsymp$ with $c(\MF{G})=19$.
The natural map
\begin{equation}
 \{ G\in\MC{L} \bigm| c(G)=19,[G]\not\in E \}
 /\text{\rm conj}
 \rightarrow
 \{ (G,S)\in\MC{S} \bigm|
 c(G)=19,[G]\not\in E \}/\text{\rm isom}
\end{equation}
 is bijective.
\end{prop}
\begin{proof}
The surjectivity follows from the definition
 of $\MC{S}$ (see (\ref{mcs})).
Let $(G,S)\in\MC{S}$ such that
 $c(G)=19$ and $[G]\not\in E$.
Suppose that $G_i\in\MC{L}$ and
 $(G_i,\Lambda_{G_i}) \cong (G,S)$ for $i=1,2$.
To prove the injectivity,
 it is sufficient to show that
 $G_1$ and $G_2$ are conjugate in $\OP{O}(\Lambda)$.
By Proposition \ref{prop:uniqueinvariant},
 $\Lambda^{G_1} \cong \Lambda^{G_2}$.
%We distinguish two cases: $c(G)\leq 18$ and $c(G)=19$.
By Theorem \ref{thm:surjcoinv},
 $\overline{\OP{O}(\Lambda_{G_1})}
 =\OP{O}(q(\Lambda_{G_1}))$.
Therefore,
 a primitive embedding
 $\Lambda_{G_1}\rightarrow\Lambda$ such that
 $(\Lambda_{G_1})^\bot_\Lambda\cong\Lambda^{G_1}$
 is unique
 up to isomorphism
 by Lemma \ref{lem:unimodembed}.
Hence we may assume that $\Lambda_{G_1}=\Lambda_{G_2}$
 by replacing $G_2$ by $\varphi G_2 \varphi^{-1}$
 for some $\varphi\in\OP{O}(\Lambda)$ if necessary.
Thus $[\OP{Clos}(G_1)]=[\OP{Clos}(G_2)]$.

(1)
The case $[\OP{Clos}(G_i)]\neq \MF{A}_{4,4},F_{384}$.
By Lemma \ref{lem:maxconj},
 $G_1$ and $G_2$ are conjugate in
 $\OP{Clos}(G_i)$ ($\subset \OP{O}(\Lambda)$).

(2)
The case $[\OP{Clos}(G_i)]= \MF{A}_{4,4}$
 (resp.\ $F_{384}$).
By the above argument, we have
 $\Lambda_{G_i}=\Lambda_G$ (resp.\ $\Lambda_H$)
 for some identification
 $\Lambda=H^2(X,\Z)$ (resp.\ $H^2(Y,Z)$).
Hence $\OP{Clos}(G_i)=G$ (resp.\ $H$).
By Lemma \ref{lem:gh},
 $G_1$ and $G_2$ are conjugate in
 $\OP{O}(\Lambda)$.
\end{proof}

\begin{prop} \label{prop:threeexceptions}
For $\MF{G}=\MF{S}_5,L_2(7),\MF{A}_6$,
 there exist exactly two elements $G_1,G_2$ in $\MC{L}$
 up to conjugacy in $\OP{O}(\Lambda)$
 such that $[G_i]=\MF{G}$.
We have $\Lambda_{G_1}\cong\Lambda_{G_2}$,
 $q(\Lambda^{G_1})\cong q(\Lambda^{G_2})$
 and $\Lambda^{G_1}\not\cong\Lambda^{G_2}$.
\end{prop}
\begin{proof}
By Proposition \ref{prop:uniquemcq} and
 Theorem \ref{thm:coinvariantuniq},
 there exists a unique element $(G_0,S)\in\MC{S}$
 up to isomorphism such that $[G_0]=\MF{G}$.
Since $\MF{G}$ is a maximal element in $\Gsymp$,
 we have $\OP{O}_0(S)=G_0$.
By Theorem \ref{thm:surjcoinv},
 $\overline{\OP{O}(S)}=\OP{O}(q(S))$.
By Lemma \ref{lem:unimodembed} and
 Proposition \ref{prop:uniqueinvariant},
 there exist exactly two primitive
 sublattices $S_1,S_2$ of $\Lambda$
 such that $S_i\cong S$
 up to $\OP{O}(\Lambda)$.
The action of $G_i:=\OP{O}_0(S_i)$ on $S_i$ is extended to
 that on $\Lambda$ such that $\Lambda_{G_i}=S_i$ ($i=1,2$).
Let $G\in\MC{L}$ such that $[G]=\MF{G}$.
Then $\Lambda_G\cong S$.
Hence we may assume that $\Lambda_G=S_i$
 ($i=1,2$)
 by replacing $G$ by $\varphi G \varphi^{-1}$
 for some $\varphi\in\OP{O}(\Lambda)$ if necessary.
Then we have $G=G_i$.
This implies the assertion.
\end{proof}

\subsection{Proof of Main Theorem}

\begin{thm} \label{thm:main}
Let $\MF{G}\in\Gsymp$.
\begin{enumerate}
\renewcommand{\labelenumi}{(\arabic{enumi})}
\item
If $\MF{G}=Q_8,T_{24}$,
 there exist exactly two elements $G_1,G_2\in\MC{L}$
 such that $[G_i]=\MF{G}$ up to conjugacy in $\OP{O}(\Lambda)$.
We have
% $\Lambda_{G_1}\not\cong\Lambda_{G_2}$
% and $\Lambda^{G_1}\not\cong\Lambda^{G_2}$.
%In fact, we have
% $\disc(\Lambda_{G_1}) \neq \disc(\Lambda_{G_2})$
 the following table,
 by changing numbering of $G_1,G_2$
 if necessary
 (see Corollary \ref{cor:closure}).
\begin{equation*}
\begin{array}{c||c|c|c||c|c|c}
 \MF{G} & n & [\OP{Clos}(G_1)] & \disc(\Lambda_{G_1})
        & n & [\OP{Clos}(G_2)] & \disc(\Lambda_{G_2}) \\
\hline
Q_8    & 12 & Q_8 & -512 & 40 & Q_8*Q_8 & -1024 \\
\hline
T_{24} & 77 & T_{192} & -192 & 54 & T_{48} & -384
\end{array}
\end{equation*}
Here $n$ is determined by
 $([G_i],q(\Lambda_{G_i}))\sim (\MF{G}_n,q_n)$.
\item
If $\MF{G}=\MF{S}_5,L_2(7),\MF{A}_6$,
 there exist exactly two elements $G_1,G_2\in\MC{L}$
 such that $[G_i]=\MF{G}$ up to conjugacy in $\OP{O}(\Lambda)$.
We have $\Lambda_{G_1}\cong\Lambda_{G_2}$,
 $q(\Lambda^{G_1})\cong q(\Lambda^{G_2})$
 and $\Lambda^{G_1}\not\cong\Lambda^{G_2}$.
\item
Otherwise,
 there exists a unique $G\in\MC{L}$
 such that $[G]=\MF{G}$
 up to conjugacy in $\OP{O}(\Lambda)$.
%We have $\Lambda_G\cong S(\MF{G})$.
\end{enumerate}
\end{thm}
\begin{proof}
By Theorem \ref{thm:coinvariantuniq},
 $(G,S)\in\MC{S}$ is determined uniquely
 by $[G]$ and $q(S)$ up to isomorphism.
The assertions (1) and (3) follow from
 Propositions \ref{prop:case18}, \ref{prop:case19}
 and Table \ref{subsect:qands}.
The asserion (2) is the same as
 Proposition \ref{prop:threeexceptions}.
\end{proof}

\section{Applications} \label{sect:applications}

Combining Xioa's result (Theorem \ref{thm:xiao}),
 the following theorem is a consequence of
 Theorem \ref{thm:main} and grobal Torelli theorem
 for $K3$ surfaces (cf.\ \cite{nikulin79fin}).

\begin{thm}
Let $G$ be a group such that $[G]\in\Gsymp$
 (see Notation \ref{nota:grp}).
Set $E_1=\{Q_8,T_{24}\}$,
 $E_2=\{ \MF{S}_5,L_2(7),\MF{A}_6 \}$.
\begin{enumerate}
\renewcommand{\labelenumi}{(\arabic{enumi})}
\item
If $[G]\not\in E_1\cup E_2$, then
 the moduli of $K3$ surfaces
 with faithful and symplectic $G$-actions
 is connected.
\item
If $[G]\in E_1\cup E_2$, then
 the moduli of $K3$ surfaces
 with faithful and symplectic $G$-actions
 has exactly two connected components.
\item
%Suppose that $\MF{G}\not\in E_2$.
If $X_i$ is a $K3$ surface
 with a faithful and symplectic $G_i$-action
 for $i=1,2$ such that $[G_i]\not\in E_2$ and
 $G_1\backslash X_1,G_2\backslash X_2$
 have the same
 $A$-$D$-$E$-configuration of the singularities,
 then $[G_1]=[G_2]=:G$ and $X_1,X_2$ are
 $G$-deformable (see Section \ref{sect:intro}).
\item \label{assertion:ext}
If $X$ is a $K3$ surface
 with a faithful and symplectic action of $G$
 of type $(\MF{G},q)\in\MC{Q}$, i.e.,
 $([G],q(H^2(X,\Z)_G))\sim (\MF{G},q)$,
 then the action is extended to that of type
 $\OP{Clos}(\MF{G},q)$
 (see Section \ref{sect:coinvariantuniq}
 and Table \ref{subsect:trees}).
\end{enumerate}
\end{thm}

The assertion (\ref{assertion:ext}) for some cases
 was pointed out and studied in detail by
 Garbagnati \cite{garbagnati08,garbagnati09}.

\section{Tables} \label{sect:tables}

\subsection{Niemeier lattices} \label{subsect:niemeier}

We give the list of Niemeier lattices $N$
 (see Subsection \ref{subsect:factniemeier}).
Let $\Delta^+$ be a set of positive roots of $N$.
We denote by $\OP{O}(N,\Delta^+)_1$ the group
 which consists of $g\in\OP{O}(N,\Delta^+)$
 preserving each connected component of
 the Dynkin diagram $R(N,\Delta^+)$.
We set
 $\OP{O}(N,\Delta^+)_2
 =\OP{O}(N,\Delta^+)/\OP{O}(N,\Delta^+)_1$.
The group $\OP{O}(N,\Delta^+)_2$ acts on
 the set of connected components of $R(N,\Delta^+)$.

\begin{equation*}
\begin{array}{c|c|c|c|c}
i & \text{root type}
 & |\OP{O}(N_i,\Delta^+_i)_1|
 & \OP{O}(N_i,\Delta^+_i)_2 & |\OP{O}(N_i,\Delta^+_i)| \\
\hline
1  & D_{24}            & 1 & 1 & 1 \\ 
2  & D_{16} \oplus E_8 & 1 & 1 & 1 \\ 
3  & E_8^{\oplus 3}    & 1 & \MF{S}_3 & 6 \\ 
4  & A_{24}            & 2 & 1 & 2 \\ 
5  & D_{12}^{\oplus 2} & 1 & \MF{S}_2 & 2 \\ 
6  & A_{17} \oplus E_7 & 2 & 1 & 2 \\ 
7  & D_{10} \oplus E_7^{\oplus 2}
                       & 1 & \MF{S}_2 & 2 \\ 
8  & A_{15} \oplus D_9 & 2 & 1 & 2 \\ 
9  & D_8^{\oplus 3}    & 1 & \MF{S}_3 & 6 \\ 
10 & A_{12}^{\oplus 2} & 2 & \MF{S}_2 & 4 \\ 
11 & A_{11} \oplus D_7 \oplus E_6
                       & 2 & 1 & 2 \\ 
12 & E_6^{\oplus 4}    & 2 & \MF{S}_4 & 48 \\ 
13 & A_9^{\oplus 2} \oplus D_6
                       & 2 & \MF{S}_2 & 4 \\ 
14 & D_6^{\oplus 4}    & 1 & \MF{S}_4 & 24 \\ 
15 & A_8^{\oplus 3}    & 2 & \MF{S}_3 & 12 \\ 
16 & A_7^{\oplus 2} \oplus D_5^{\oplus 2}
                       & 2
            & \MF{S}_2 \times \MF{S}_2 & 8 \\ 
17 & A_6^{\oplus 4}    & 2 & \MF{A}_4 & 24 \\ 
18 & A_5^{\oplus 4} \oplus D_4
                       & 2 & \MF{S}_4 & 48 \\ 
19 & D_4^{\oplus 6}    & 3 & \MF{S}_6 & 2160 \\ 
20 & A_4^{\oplus 6}    & 2 & \MF{S}_5 & 240 \\ 
21 & A_3^{\oplus 8}    & 2
 & \mathbb{F}_2^3 \rtimes \OP{GL}(3,\mathbb{F}_2)
                                      & 2688 \\ 
22 & A_2^{\oplus 12}   & 2 & M_{12} & 190080 \\ 
23 & A_1^{\oplus 24}   & 1 & M_{24} & 244823040
\end{array}
\end{equation*}

\subsection{Abstract groups and discriminant forms}
\label{subsect:qands}

\newcommand{\BQ}[3]
 {\left(\begin{smallmatrix}
   #1 & #3 \\
   #3 & #2
 \end{smallmatrix}\right)}
\newcommand{\TQ}[6]
 {\left(\begin{smallmatrix}
   #1 & #6 & #5 \\
   #6 & #2 & #4 \\
   #5 & #4 & #3
  \end{smallmatrix}\right)}
\newcommand{\BQP}[4]{\BQ{#1}{#2}{#3}^{\oplus #4}}
\newcommand{\II}{\text{II}}
\newcommand{\LFP}[2]{\LF{#1}^{\oplus #2}}

We give the list of a complete representative
 $\{ (\MF{G}_n,q_n) \}$ of $\MC{Q}/\sim$.
Recall that
\begin{align*}
\MC{Q}&=\{
 (\MF{G},q) \bigm| \exists G\in\MC{L}
 \text{ such that }
 \MF{G} = [G],
 q\cong q(\Lambda_{G})
\} \\
&=\{
 (\MF{G},q) \bigm| \exists (G,N)\in\MC{N}
 \text{ such that }
 \MF{G} = [G],
 q\cong q(N_{G})
\}
\end{align*}
 and $(\MF{G},q)\sim (\MF{G}',q')$ if and only if
 $\MF{G}=\MF{G}'$, $q\cong q'$
 (see Subsection \ref{subsect:conseq}).
For $q:A(q)\rightarrow \Q/2\Z$,
 we denote the order of $A(q)$ by $|q|$.
We use the following notation
 (cf.\ \cite{SP}):
\begin{gather*}
 a^{+n} =
 q^{(p)}_+(a)^{\oplus n}, ~
 a^{-n} =
 q^{(p)}_+(a)^{\oplus n-1}
 \oplus q^{(p)}_-(a), \\
 b^{+n}_\II = {u^{(2)}(b)}^{\oplus n}, ~
 b^{-n}_\II = {u^{(2)}(b)}^{\oplus n-1}
 \oplus v^{(2)}(b), ~
 b^{dr}_t=q(L^{(2)}_{r,d,t,\text{I}}(b)),
\end{gather*}
where $p$ is an odd prime, $a=p^k$,
 $b=2^k$ and $L^{(2)}_{r,d,t,e}$ is
 a (unique) unimodular lattice over $\Z_2$
 which has the invariants $r,d,t,e$ defined in
 Proposition \ref{prop:2adicunimodular}
 (see Section \ref{sect:lattice}).
For example,
\begin{align*}
 A(q_{63}) &\cong (\Z/2)^{\oplus 3}
 \oplus \Z/3\Z \oplus \Z/9\Z, \\
 q_{63} &\cong \LF{-1/2} \oplus
 \begin{pmatrix}1&1/2\\1/2&1\end{pmatrix} \oplus
 \LF{2/3} \oplus \LF{2/9}.
\end{align*}
In the list, e.g.\ $q_5$ is isomorphic to $q_{16}$.
The column $i$ indicates the catalogue number of $\MF{G}_n$ in GAP
 (see Appendix).

\begin{equation*}
\begin{array}{c|c|c|c|c|c|c}
n & |\MF{G}_n| & i & \MF{G}_n & |q_n| & q_n & c(\MF{G}_n) \\
\hline
1 & 2 & 1 & C_2 & 256 & 2^{+8}_\II & 8 \\ 
2 & 3 & 1 & C_3 & 729 & 3^{+6} & 12 \\ 
3 & 4 & 2 & C_2^2 & 1024 & 2^{-6}_\II , 4^{-2}_\II & 12 \\ 
4 & 4 & 1 & C_4 & 1024 & 2^{+2}_2 , 4^{+4}_\II & 14 \\ 
5 & 5 & 1 & C_5 & 625 & \sharp 16 & 16 \\ 
6 & 6 & 1 & D_6 & 972 & 2^{-2}_\II , 3^{+5} & 14 \\ 
7 & 6 & 2 & C_6 & 1296 & \sharp 18 & 16 \\ 
8 & 7 & 1 & C_7 & 343 & \sharp 33 & 18 \\ 
9 & 8 & 5 & C_2^3 & 1024 & 2^{+6}_\II , 4^{+2}_2 & 14 \\ 
10 & 8 & 3 & D_8 & 1024 & 4^{+5}_1 & 15 \\ 
11 & 8 & 2 & C_2 \times C_4 & 1024 & \sharp 22 & 16 \\ 
12 & 8 & 4 & Q_8 & 512 & 2^{-3}_7 , 8^{-2}_\II & 17 \\ 
13 & 8 & 4 & Q_8 & 1024 & \sharp 40 & 17 \\ 
14 & 8 & 1 & C_8 & 512 & \sharp 26 & 18 \\ 
15 & 9 & 2 & C_3^2 & 729 & \sharp 30 & 16 \\ 
16 & 10 & 1 & D_{10} & 625 & 5^{+4} & 16 \\ 
17 & 12 & 3 & \MF{A}_4 & 576 & 2^{-2}_\II , 4^{-2}_\II , 3^{+2} & 16 \\ 
18 & 12 & 4 & D_{12} & 1296 & 2^{+4}_\II , 3^{+4} & 16 \\ 
19 & 12 & 5 & C_2 \times C_6 & 432 & \sharp 61 & 18 \\ 
20 & 12 & 1 & Q_{12} & 432 & \sharp 61 & 18 \\ 
21 & 16 & 14 & C_2^4 & 512 & 2^{+6}_\II , 8^{+1}_1 & 15 \\ 
22 & 16 & 11 & C_2 \times D_8 & 1024 & 2^{+2}_\II , 4^{+4}_0 & 16 \\ 
23 & 16 & 3 & \Gamma_2 c_1 & 512 & \sharp 39 & 17 \\ 
24 & 16 & 13 & Q_8*C_4 & 1024 & \sharp 40 & 17 \\ 
25 & 16 & 2 & C_4^2 & 256 & \sharp 75 & 18 \\ 
26 & 16 & 8 & SD_{16} & 512 & 2^{+1}_7 , 4^{+1}_7 , 8^{+2}_\II & 18 \\ 
27 & 16 & 12 & C_2 \times Q_8 & 256 & \sharp 75 & 18 \\ 
28 & 16 & 6 & \Gamma_2 d & 256 & \sharp 80 & 19 \\ 
29 & 16 & 9 & Q_{16} & 256 & \sharp 80 & 19 \\ 
30 & 18 & 4 & \MF{A}_{3,3} & 729 & 3^{+4} , 9^{-1} & 16 \\ 
31 & 18 & 3 & C_3 \times D_6 & 972 & \sharp 48 & 18 \\ 
32 & 20 & 3 & \OP{Hol}(C_5) & 500 & 2^{-2}_6 , 5^{+3} & 18 \\ 
33 & 21 & 1 & C_7 \rtimes C_3 & 343 & 7^{+3} & 18 \\ 
34 & 24 & 12 & \MF{S}_4 & 576 & 4^{+3}_3 , 3^{+2} & 17 \\ 
35 & 24 & 13 & C_2 \times \MF{A}_4 & 576 & \sharp 51 & 18 \\ 
36 & 24 & 8 & C_3 \rtimes D_8 & 432 & \sharp 61 & 18 \\ 
37 & 24 & 3 & T_{24} & 192 & \sharp 77 & 19 \\ 
38 & 24 & 3 & T_{24} & 384 & \sharp 54 & 19 \\ 
39 & 32 & 27 & 2^4 C_2 & 512 & 2^{+2}_\II , 4^{+2}_0 , 8^{+1}_7 & 17 \\ 
40 & 32 & 49 & Q_8*Q_8 & 1024 & 4^{+5}_7 & 17 \\ 
\end{array}
\end{equation*} 
\begin{equation*}
\begin{array}{c|c|c|c|c|c|c}
n & |\MF{G}_n| & i & \MF{G}_n & |q_n| & q_n & c(\MF{G}_n) \\
\hline
41 & 32 & 6 & \Gamma_7 a_1 & 512 & \sharp 56 & 18 \\ 
42 & 32 & 31 & \Gamma_4 c_2 & 256 & \sharp 75 & 18 \\ 
43 & 32 & 7 & \Gamma_7 a_2 & 256 & \sharp 80 & 19 \\ 
44 & 32 & 11 & \Gamma_3 e & 256 & \sharp 80 & 19 \\ 
45 & 32 & 44 & \Gamma_6 a_2 & 256 & \sharp 80 & 19 \\ 
46 & 36 & 9 & 3^2 C_4 & 324 & 2^{-2}_6 , 3^{+2} , 9^{-1} & 18 \\ 
47 & 36 & 11 & C_3 \times \MF{A}_4 & 432 & \sharp 61 & 18 \\ 
48 & 36 & 10 & \MF{S}_{3,3} & 972 & 2^{-2}_\II , 3^{+3} , 9^{-1} & 18 \\ 
49 & 48 & 50 & 2^4 C_3 & 384 & 2^{-4}_\II , 8^{+1}_1 , 3^{-1} & 17 \\ 
50 & 48 & 3 & 4^2 C_3 & 256 & \sharp 75 & 18 \\ 
51 & 48 & 48 & C_2 \times \MF{S}_4 & 576 & 2^{+2}_\II , 4^{+2}_2 , 3^{+2} & 18 \\ 
52 & 48 & 49 & 2^2(C_2 \times C_6) & 288 & \sharp 78 & 19 \\ 
53 & 48 & 30 & 2^2 Q_{12} & 288 & \sharp 78 & 19 \\ 
54 & 48 & 29 & T_{48} & 384 & 2^{+1}_7 , 8^{-2}_\II , 3^{-1} & 19 \\ 
55 & 60 & 5 & \MF{A}_5 & 300 & 2^{-2}_\II , 3^{+1} , 5^{-2} & 18 \\ 
56 & 64 & 138 & \Gamma_{25} a_1 & 512 & 4^{+3}_5 , 8^{+1}_1 & 18 \\ 
57 & 64 & 242 & \Gamma_{13} a_1 & 256 & \sharp 75 & 18 \\ 
58 & 64 & 32 & \Gamma_{22} a_1 & 256 & \sharp 80 & 19 \\ 
59 & 64 & 35 & \Gamma_{23} a_2 & 256 & \sharp 80 & 19 \\ 
60 & 64 & 136 & \Gamma_{26} a_2 & 256 & \sharp 80 & 19 \\ 
61 & 72 & 43 & \MF{A}_{4,3} & 432 & 4^{-2}_\II , 3^{-3} & 18 \\ 
62 & 72 & 40 & N_{72} & 324 & 4^{+1}_1 , 3^{+2} , 9^{-1} & 19 \\ 
63 & 72 & 41 & M_9 & 216 & 2^{-3}_7 , 3^{-1} , 9^{-1} & 19 \\ 
64 & 80 & 49 & 2^4 C_5 & 160 & \sharp 81 & 19 \\ 
65 & 96 & 227 & 2^4 D_6 & 384 & 2^{-2}_\II , 4^{+1}_7 , 8^{+1}_1 , 3^{-1} & 18 \\ 
66 & 96 & 70 & 2^4 C_6 & 384 & \sharp 76 & 19 \\ 
67 & 96 & 64 & 4^2 D_6 & 256 & \sharp 80 & 19 \\ 
68 & 96 & 195 & 2^3 D_{12} & 288 & \sharp 78 & 19 \\ 
69 & 96 & 204 & (Q_8*Q_8) \rtimes C_3 & 192 & \sharp 77 & 19 \\ 
70 & 120 & 34 & \MF{S}_5 & 300 & 4^{-1}_3 , 3^{+1} , 5^{-2} & 19 \\ 
71 & 128 & 931 & F_{128} & 256 & \sharp 80 & 19 \\ 
72 & 144 & 184 & \MF{A}_4^2 & 288 & \sharp 78 & 19 \\ 
73 & 160 & 234 & 2^4 D_{10} & 160 & \sharp 81 & 19 \\ 
74 & 168 & 42 & L_2(7) & 196 & 4^{+1}_1 , 7^{+2} & 19 \\ 
75 & 192 & 1023 & 4^2 \MF{A}_4 & 256 & 2^{-2}_\II , 8^{-2}_6 & 18 \\ 
76 & 192 & 955 & H_{192} & 384 & 4^{-2}_4 , 8^{+1}_7 , 3^{-1} & 19 \\ 
77 & 192 & 1493 & T_{192} & 192 & 4^{-3}_7 , 3^{+1} & 19 \\ 
78 & 288 & 1026 & \MF{A}_{4,4} & 288 & 2^{+2}_\II , 8^{+1}_1 , 3^{+2} & 19 \\ 
79 & 360 & 118 & \MF{A}_6 & 180 & 4^{-1}_5 , 3^{+2} , 5^{+1} & 19 \\ 
80 & 384 & 18135 & F_{384} & 256 & 4^{+1}_7 , 8^{+2}_6 & 19 \\ 
81 & 960 & 11357 & M_{20} & 160 & 2^{-2}_\II , 8^{+1}_1 , 5^{-1} & 19 \\ 
\end{array}
\end{equation*}

\subsection{Invariant lattices $\Lambda^G$}
\label{subsect:invariant}

For $G\in\MC{L}$, there exsits a number $n$
 such that $([G],q(\Lambda_G))\sim (\MF{G}_n,q_n)$
 (see Table \ref{subsect:qands}).
Here we give the invariant lattices $\Lambda^G$
 for each $n$.
We set
\begin{equation*}
 r=\rank \Lambda^G=22-c(G),~ d=\disc \Lambda^G,~
 q=-q_n\cong q(\Lambda^G).
\end{equation*}
In the table, we set
\begin{equation*}
 U=\begin{pmatrix}0&1\\1&0\end{pmatrix},~
 A_2=\begin{pmatrix}2&-1\\-1&2\end{pmatrix},~
 D_4=\begin{pmatrix}
 2&0&0&-1\\
 0&2&0&-1\\
 0&0&2&-1\\
 -1&-1&-1&2
 \end{pmatrix}
\end{equation*}
 and $E_8$ denotes the root latice of type $E_8$,
 as usual.
For abelian $G\in\MC{L}$,
 the Gramian matrices of $\Lambda^G$ were determined in
 \cite{garbagnatisarti09}.

\begin{equation*}
\begin{array}{c|c|c|c|c}
n & r & d
 & q
 & \text{Gramian matrix} \\
\hline
1 & 14 & - 256 & 2^{+8}_\II & U^{\oplus 3} \oplus E_8(-2) \\ 
2 & 10 & - 729 & 3^{+6}
    & U \oplus U(3)^{\oplus2} \oplus A_2(-1)^{\oplus 2} \\ 
3 & 10 & - 1024 & 2^{-6}_\II,4^{-2}_\II
    & U \oplus U(2)^{\oplus 2} \oplus D_4(-2) \\ 
4 & 8 & - 1024 & 2^{+2}_6,4^{+4}_\II
    & U \oplus U(4)^{\oplus 2} \oplus \LFP{-2}{2} \\ 
6 & 8 & - 972 & 2^{-2}_\II,3^{-5}
    & U(3) \oplus A_2(2) \oplus A_2(-1)^{\oplus 2} \\ 
9 & 8 & - 1024 & 2^{+6}_\II,4^{+2}_6 & U(2)^{\oplus 3} \oplus \LFP{-4}{2} \\ 
10 & 7 & 1024 & 4^{+5}_7 & U \oplus \LFP{4}{2} \oplus \LFP{-4}{3} \\ 
12 & 5 & 512 & 2^{-3}_1,8^{-2}_\II
    & \TQ{6}{6}{6}{-2}{2}{2} \oplus \LFP{-2}{2} \\ 
16 & 6 & - 625 & 5^{+4} & U \oplus U(5)^{\oplus 2} \\ 
17 & 6 & - 576 & 2^{-2}_\II,4^{-2}_\II,3^{+2}
    & U \oplus A_2(2) \oplus A_2(-4) \\ 
18 & 6 & - 1296 & 2^{+4}_\II,3^{+4} & U \oplus U(6)^{\oplus 2} \\ 
21 & 7 & 512 & 2^{+6}_\II,8^{+1}_7 & U(2)^{\oplus 3} \oplus \LF{-8} \\ 
22 & 6 & - 1024 & 2^{+2}_\II,4^{+4}_0 & U(2) \oplus \LFP{4}{2} \oplus \LFP{-4}{2} \\ 
26 & 4 & - 512 & 2^{+1}_1,4^{+1}_1,8^{+2}_\II
    & U(8) \oplus \LF{2} \oplus \LF{4} \\ 
30 & 6 & - 729 & 3^{+4},9^{+1} & U(3)^{\oplus 2} \oplus \BQ{2}{0}{3} \\ 
32 & 4 & - 500 & 2^{-2}_2,5^{+3} & U(5) \oplus \BQ{4}{6}{2} \\ 
33 & 4 & - 343 & 7^{-3} & U(7) \oplus \BQ{2}{4}{1} \\ 
34 & 5 & 576 & 4^{+3}_5,3^{+2} & U \oplus A_2(2) \oplus \LF{-12} \\ 
39 & 5 & 512 & 2^{+2}_\II,4^{+2}_0,8^{+1}_1
    & U(2) \oplus \LF{4} \oplus \LF{-4} \oplus \LF{8} \\ 
40 & 5 & 1024 & 4^{+5}_1 & \LFP{4}{3} \oplus \LFP{-4}{2} \\ 
\end{array}
\end{equation*}
\begin{equation*}
\begin{array}{c|c|c|c|c}
46 & 4 & - 324 & 2^{-2}_2,3^{+2},9^{+1} & A_2 \oplus \LF{6} \oplus \LF{-18} \\ 
48 & 4 & - 972 & 2^{-2}_\II,3^{-3},9^{+1} & U(3) \oplus A_2(6) \\ 
49 & 5 & 384 & 2^{-4}_\II,8^{+1}_7,3^{+1} & U(2) \oplus A_2(2) \oplus \LF{-8} \\ 
51 & 4 & - 576 & 2^{+2}_\II,4^{+2}_6,3^{+2} & U(2) \oplus \LFP{12}{2} \\ 
54 & 3 & 384 & 2^{+1}_1,8^{-2}_\II,3^{+1} & \TQ{2}{16}{16}{8}{0}{0} \\ 
55 & 4 & - 300 & 2^{-2}_\II,3^{-1},5^{-2} & U \oplus A_2(10) \\ 
56 & 4 & - 512 & 4^{+3}_3,8^{+1}_7 & \LFP{4}{3} \oplus \LF{-8} \\ 
61 & 4 & - 432 & 4^{-2}_\II,3^{+3} & U(3) \oplus A_2(4) \\ 
62 & 3 & 324 & 4^{+1}_7,3^{+2},9^{+1} & \TQ{6}{6}{12}{3}{3}{0} \\ 
63 & 3 & 216 & 2^{-3}_1,3^{+1},9^{+1} & \TQ{2}{12}{12}{6}{0}{0} \\ 
65 & 4 & - 384 & 2^{-2}_\II,4^{+1}_1,8^{+1}_7,3^{+1}
     & A_2(2) \oplus \LF{4} \oplus \LF{-8} \\ 
70 & 3 & 300 & 4^{-1}_5,3^{-1},5^{-2}
    & \TQ{4}{4}{20}{0}{0}{1} , \TQ{4}{6}{16}{1}{2}{2} \\ 
74 & 3 & 196 & 4^{+1}_7,7^{+2}
    & \TQ{2}{4}{28}{0}{0}{1} , \TQ{4}{8}{8}{1}{2}{2} \\ 
75 & 4 & - 256 & 2^{-2}_\II,8^{-2}_2
 & \left(\begin{smallmatrix}4&0&2&0\\0&4&2&0\\2&2&4&4\\0&0&4&0\end{smallmatrix}\right) \\ 
76 & 3 & 384 & 4^{-2}_4,8^{+1}_1,3^{+1} & \TQ{4}{8}{12}{0}{0}{0} \\ 
77 & 3 & 192 & 4^{-3}_1,3^{-1} & \TQ{4}{8}{8}{4}{0}{0} \\ 
78 & 3 & 288 & 2^{+2}_\II,8^{+1}_7,3^{+2} & \TQ{8}{8}{8}{2}{4}{4} \\ 
79 & 3 & 180 & 4^{-1}_3,3^{+2},5^{+1}
    & \TQ{2}{8}{12}{0}{0}{1} , \TQ{6}{6}{8}{3}{3}{0} \\ 
80 & 3 & 256 & 4^{+1}_1,8^{+2}_2 & \TQ{4}{8}{8}{0}{0}{0} \\ 
81 & 3 & 160 & 2^{-2}_\II,8^{+1}_7,5^{-1} & \TQ{4}{4}{12}{2}{2}{0} \\
\end{array}
\end{equation*}

\subsection{Trees of groups with common invariant lattices}
\label{subsect:trees}

We give the trees of
\begin{equation*}
 T_S=
 \{ \MF{G}_n \bigm| S(\MF{G}_n,q_n)\cong S \}
 =\{ \MF{G}_n \bigm| q_n\cong q(S) \}
\end{equation*}
 for $T_S$ with $\sharp T_S\geq 2$.
In the table, $\sharp n$ denotes $\MF{G}_n$.
The maximal element in each $T_S$ corresponds to
 an element in $\MC{Q}_\clos$ defined
 by (\ref{qclos}).
The extensions $\sharp 5-\sharp 16$,
 $\sharp 7-\sharp 18$, $\sharp 11-\sharp 22$
 and $\sharp 15-\sharp 30$ are studied in detail by
 Garbagnati \cite{garbagnati08,garbagnati09}.

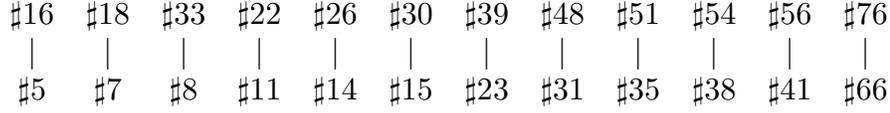
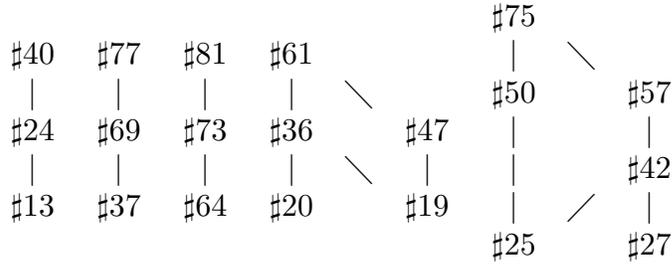
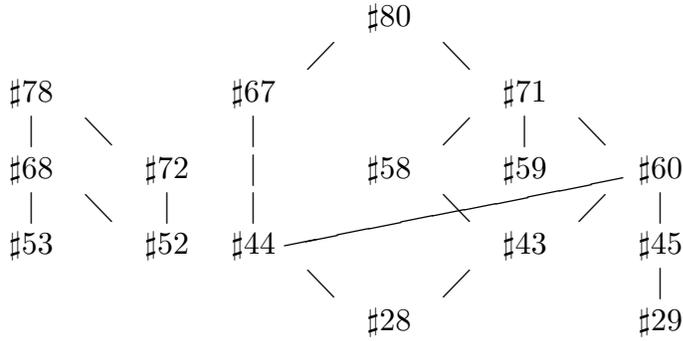
\newcommand{\tree}[2]
{\begin{array}{c}
 \sharp #1 \\
 | \\
 \sharp #2
\end{array}}

\begin{gather*}
\tree{16}{5}
\tree{18}{7}
\tree{33}{8}
\tree{22}{11}
\tree{26}{14}
\tree{30}{15}
\tree{39}{23}
\tree{48}{31}
\tree{51}{35}
\tree{54}{38}
\tree{56}{41}
\tree{76}{66} \\ \\
\begin{array}{c}
\sharp 40 \\
| \\
\sharp 24 \\
| \\
\sharp 13
\end{array} ~~
\begin{array}{c}
\sharp 77 \\
| \\
\sharp 69 \\
| \\
\sharp 37
\end{array} ~~
\begin{array}{c}
\sharp 81 \\
| \\
\sharp 73 \\
| \\
\sharp 64
\end{array} ~~
\begin{array}{ccc}
\sharp 61 &  & \\
|         & \diagdown & \\
\sharp 36 & & \sharp 47 \\
|         & \diagdown & | \\
\sharp 20 & & \sharp 19
\end{array} ~~
\begin{array}{ccc}
\sharp 75 & & \\
| & \diagdown & \\
\sharp 50 & & \sharp 57 \\
| & & | \\
| & & \sharp 42 \\
| & \diagup & | \\
\sharp 25 & & \sharp 27
\end{array} \\ \\
\begin{array}{ccc}
\sharp 78 &  & \\
|         & \diagdown & \\
\sharp 68 & & \sharp 72 \\
|         & \diagdown & | \\
\sharp 53 & & \sharp 52
\end{array} ~~
\begin{array}{ccccccc}
  &  &\sharp 80&  &  &  &  \\
  & \diagup&  &\diagdown &  &  &  \\
\sharp 67&  &  &  &\sharp 71&  &  \\
 |&  &  & \diagup&| &\diagdown &  \\
 |&  &\sharp 58&  &\sharp 59&  &\sharp 60\\
 |&  &  & \diagdown&  &\diagup &| \\
\sharp 44&  &  &  &\sharp 43&  &\sharp 45\\
  & \diagdown&  & \diagup&  &  &| \\
  &  &\sharp 28&  &  &  &\sharp 29
\end{array}
\setlength\unitlength{1truecm}
\begin{picture}(0,0)(0,0)
%\put(-1,0){$\circ$}
%\put(-5.5,-1){$\circ$}
\put(-1,0){\line(-5,-1){4.5}}
\end{picture}
% \text{\&}
%\tree{60}{44}
\end{gather*}

\subsection{Extensions of $G\in\MC{L}$}
\label{subsect:extension}

We give the list of possible extensions
 of $G\in\MC{L}_\clos$.
For example, let $G\in\MC{L}$
 of type $(\MF{G}_{55},q_{55})$, i.e.,
 $([G],q(\Lambda_G))\sim (\MF{G}_{55},q_{55})$.
Then, for $i=1,2$,
 there exsits an element $G'\in\MC{L}_\clos$
 of type $(\MF{G}_{79},q_{79})$
 such that $G\subset G'$ and
 $G'$ is conjugate to $G_i$ in Theorem \ref{thm:main}(2).
We omit the eleven maximal cases:
 $n=54,62,63,70,74,76,77,78,79,80,81$,
 for there is no proper extension.

\begin{equation*}
\begin{array}{c|l}
n & \text{extensions} \\
\hline
1 &
\hspace{-1ex}\begin{array}{l}
 3,4,6,9,10,12,16,17,18,21,22,26,30,32,34,39,40,46,48,49,51,\\
 54,55,56,61,62,63,65,70,74,75,76,77,78,79,80,81
\end{array} \\
2 &
\hspace{-1ex}\begin{array}{l}
 6,17,18,30,33,34,46,48,49,51,54,55,61,62,63,65,70,74,75,76,\\
 77,78,79,80,81
\end{array}\\
3 &
\hspace{-1ex}\begin{array}{l}
 9,10,17,18,21,22,26,34,39,40,48,49,51,54,55,56,61,62,65,70,\\
 74,75,76,77,78,79,80,81
\end{array}\\
4 &
\hspace{-1ex}\begin{array}{l}
 10,12,22,26,32,34,39,40,46,51,54,56,61,62,63,65,70,74,75,76,\\
 77,78,79,80,81
\end{array}\\
6 & 18,30,34,46,48,51,54,55,61,62,63,65,70,74,76,77,78,79,80,81\\
9 & 21,22,39,40,49,51,56,65,75,76,77,78,80,81\\
10 & 22,26,34,39,40,51,54,56,61,62,65,70,74,75,76,77,78,79,80,81\\
12 & 26,54,63,75,80,81\\
16 & 32,55,70,79,81\\
17 & 34,49,51,55,61,65,70,74,75,76,77,78,79,80,81\\
18 & 48,51,54,61,62,70,76,77,78\\
21 & 39,49,56,65,75,76,77,78,80,81\\
22 & 39,40,51,56,65,75,76,77,78,80,81\\
26 & 54,80\\
30 & 46,48,61,62,63,78,79\\
32 & 70\\
33 & 74\\
34 & 51,61,65,70,74,76,77,78,79,80,81\\
39 & 56,65,75,76,77,78,80,81\\
40 & 56,76,77,80\\
46 & 62,63,79\\
48 & 62\\
49 & 65,75,76,78,80,81\\
51 & 76,77,78\\
55 & 70,79,81\\
56 & 76,77,80\\
61 & 78\\
65 & 76,78,80,81\\
75 & 80,81
\end{array}
\end{equation*}

\subsection{Root types of $N^G$}
\label{subsect:orbits}

We give the type of the root sublattice
 of $N^G$,
 which is generated by vectors $v\in N^G$
 with $\BF{v}{v}=-2$,
 for each $(G,N)\in\MC{N}$
 such that $[G]=\MF{G}_n$ and $q(N_G)\cong q_n$
 (see Table \ref{subsect:qands}).
In the list,
 elements in $\MC{N}'$
 are enclosed by boxes
 (see Proposition \ref{prop:listuniq})
 and the number of vectors $v\in N^G$
 with $\BF{v}{v}=-4$ or $-6$ are given
 for the cases $n=32,41,56,63$.
As for Niemeier lattices $N=N_i$,
 see Table \ref{subsect:niemeier}.
\\

\noindent
$n=1$
\begin{equation*} 
\begin{array}{c|ccccc} 
\hline
i & 3 & 6 & 7 & 8 & 9 \\ 
\hline 
\text{type} & E_8 & A_1^{\oplus 9}\oplus E_7
 & D_9 & A_1^{\oplus 8}\oplus D_8 & D_8
 \\
\hline
\hline
i & 11 & 12 & 12 & 13 & 14 \\
\hline
\text{type}
 & A_1^{\oplus 6} \oplus D_4 \oplus D_6
 & D_4^{\oplus 4}
 & D_4 \oplus E_6 
 & A_1^{\oplus 10} \oplus D_6
 & D_5^{\oplus 2} \\
\hline
\hline
i & 15 & 16 & 16 & 16 & 18 \\ 
\hline 
\text{type} & A_8
 & A_1^{\oplus 8} \oplus D_4^{\oplus 2}
 & A_1^{\oplus 4} \oplus A_7
 & D_4\oplus D_5
 & A_1^{\oplus 12} \oplus D_4
 \\
\hline
\hline
i & 18 & 19 & 19 & 20 & 21 \\ 
\hline
\text{type}
 & A_1^{\oplus 3} \oplus A_3 \oplus A_5
 & A_3^{\oplus 4} & D_4^{\oplus 2}
 & A_4^{\oplus 2}
 & A_1^{\oplus 16} \\
\hline
\hline
i & 21 & 22 & \fbox{23} \\
\hline
\text{type}
 & A_1^{\oplus 4} \oplus A_3^{\oplus 2}
 & A_2^{\oplus 4} & A_1^{\oplus 8} \\
\hline
\end{array}
\end{equation*}
$n= 2$
\begin{equation*} 
\begin{array}{c|ccccccccc} 
\hline 
i & 12 & 14 & 17 & 18 & 19 & 19 & 21 & 22 & \fbox{23} \\ 
\hline 
\text{type} &  E_6
 &  D_6
 &  A_6
 &  A_2\oplus A_5
 & A_2^{\oplus 6}
 &  D_4\oplus A_2^{\oplus 2}
 &  A_3^{\oplus 2}
 &  A_2^{\oplus 3}
 &  A_1^{\oplus 6} \\
\hline 
\end{array} 
\end{equation*} 
$n= 3$
\begin{equation*} 
\begin{array}{c|ccccccc} 
\hline
i & 12 & 16 & 16 & 18 & 19 & 19 & 21 \\ 
\hline 
\text{type} &  D_4^{\oplus 2}
 &  A_1^{\oplus 8}
 &  D_4^{\oplus 2}
 &   A_1^{\oplus 6} \oplus A_3
 &  A_3^{\oplus 2}
 &  D_4^{\oplus 2}
 &  A_1^{\oplus 4}
 \\
\hline
\hline
i & 21 & 21 & 21 & \fbox{22} & 23 & 23 \\ 
\hline 
\text{type}
 & A_1^{\oplus 8}
 &  A_3\oplus A_1^{\oplus 6}
 &  A_3^{\oplus 2}
 &  A_2^{\oplus 2}
 &  A_1^{\oplus 4}
 &  A_1^{\oplus 8} \\
\hline
\end{array} 
\end{equation*} 
$n= 4$ 
\begin{equation*} 
\begin{array}{c|ccccccc} 
\hline
i & 13 & 18 & 19 & 20 & 21 & 22 & \fbox{23} \\ 
\hline 
\text{type} &  D_5
 &  D_4
 &  A_3^{\oplus 2}
 &  A_1^{\oplus 2} \oplus A_4 
 &  A_1^{\oplus 2}\oplus A_3
 &  A_1^{\oplus 2}\oplus A_2^{\oplus 2}
 & A_1^{\oplus 4} \\
\hline
\end{array} 
\end{equation*} 
$n= 5,16$ 
\begin{equation*} 
\begin{array}{c|cccc} 
\hline
i & 19 & 20 & 22 & \fbox{23} \\ 
\hline 
\text{type} &  D_4
 &  A_4
 &  A_2^{\oplus 2}
 & A_1^{\oplus 4} \\
\hline
\end{array} 
\end{equation*} 
$n= 6$ 
\begin{equation*} 
\begin{array}{c|cccccc}
\hline 
i & 12 & 12 & 14 & 18 & 18 & 19 \\ 
\hline 
\text{type}
 &  D_4
 &  E_6
 &  D_5
 &  A_1^{\oplus 3}\oplus A_2
 &  A_2\oplus A_5
 &  A_2^{\oplus 4}\\
\hline
\hline
i & 19 & 19 & 21 & 22 & 22 & \fbox{23} \\ 
\hline 
\text{type}
 &  A_2^{\oplus 2} \oplus A_3
 &  D_4 
 &  A_1^{\oplus 2}\oplus A_3
 &  A_2
 &  A_2^{\oplus 3}
 & A_1^{\oplus 4} \\
\hline
\end{array} 
\end{equation*} 
$n= 7,18$
\begin{equation*} 
\begin{array}{c|ccccccc} 
\hline
i & 12 & 18 & 19 & 19 & 21 & 22 & \fbox{23} \\ 
\hline 
\text{type} &  D_4
 &  A_1^{\oplus 3}\oplus A_2
 &  A_2^{\oplus 2}
 &  A_3
 & A_1^{\oplus 4}
 &  A_2
 & A_1^{\oplus 2} \\
\hline
\end{array} 
\end{equation*} 
$n= 8,33$
\begin{equation*} 
\begin{array}{c|cc} 
\hline
i & 21 & \fbox{23} \\ 
\hline 
\text{type} &  A_3
 & A_1^{\oplus 3} \\
\hline
\end{array} 
\end{equation*} 
$n= 9$
\begin{equation*} 
\begin{array}{c|ccccc} 
\hline
i & 21 & 21 & \fbox{23} & 23 & 23 \\ 
\hline 
\text{type}
 &  A_1^{\oplus 4}
 & A_1^{\oplus 8}
 & A_1^{\oplus 2}
 &  A_1^{\oplus 4}
 & A_1^{\oplus 8} \\
\hline
\end{array} 
\end{equation*} 
$n= 10 $
\begin{equation*} 
\begin{array}{c|ccccccc} 
\hline
i & 18 & 19 & 21 & 21 & 22 & \fbox{23} & 23 \\ 
\hline 
\text{type} &  A_3
 &  A_3^{\oplus 2}
 &  A_1^{\oplus 4}
 &  A_1^{\oplus 2}\oplus A_3
 &  A_2^{\oplus 2}
 &  A_1^{\oplus 2}
 &  A_1^{\oplus 4} \\
\hline
\end{array} 
\end{equation*} 
$n= 11,22$
\begin{equation*} 
\begin{array}{c|cccc} 
\hline
i & 21 & \fbox{23} & 23 \\ 
\hline 
\text{type} & A_1^{\oplus 4}
 & A_1^{\oplus 2}
 & A_1^{\oplus 4} \\
\hline
\end{array} 
\end{equation*} 
$n= 12$ 
\begin{equation*} 
\begin{array}{c|ccc} 
\hline
i & 18 & 22 & \fbox{23} \\ 
\hline 
\text{type} &  D_4
 &  A_1^{\oplus 3}\oplus A_2
 & A_1^{\oplus 4} \\
\hline
\end{array} 
\end{equation*} 
$n= 13,24,28,29,37,40,43,44,45,59,60,67,69,71,77,80$
\begin{equation*} 
\begin{array}{c|c} 
\hline
i & \fbox{23} \\ 
\hline 
\text{type} &  A_1^{\oplus 2} \\
\hline
\end{array} 
\end{equation*} 
$n= 14,26$
\begin{equation*} 
\begin{array}{c|ccc} 
\hline
i & 18 & 22 & \fbox{23} \\ 
\hline 
\text{type} &  A_3
 &  A_1\oplus A_2
 & A_1^{\oplus 2} \\
\hline
\end{array} 
\end{equation*} 
$n= 15,30$
\begin{equation*} 
\begin{array}{c|ccc} 
\hline
i & 19 & 22 & \fbox{23} \\ 
\hline 
\text{type} &  A_2^{\oplus 3}
 &  A_2^{\oplus 3}
 & A_1^{\oplus 3} \\
\hline
\end{array} 
\end{equation*} 
$n= 17$ 
\begin{equation*} 
\begin{array}{c|cccccccc} 
\hline
i & 19 & 19 & 21 & 21 & 22 & 23 & 23 & \fbox{23} \\ 
\hline 
\text{type}
 &  A_2^{\oplus 2}
 &  A_2\oplus D_4
 &  A_3
 &  A_3^{\oplus 2}
 &  A_2^{\oplus 2}
 & A_1^{\oplus 2}
 & A_1^{\oplus 4}
 &  A_1^{\oplus 5} \\
\hline
\end{array} 
\end{equation*} 
$n= 19,20,36,47,61$
\begin{equation*} 
\begin{array}{c|cc} 
\hline
i & 19 & \fbox{23} \\ 
\hline 
\text{type} &  A_2^{\oplus 2}
 &  A_1^{\oplus 2} \\
\hline
\end{array} 
\end{equation*} 
$n= 21$
\begin{equation*} 
\begin{array}{c|ccc} 
\hline
i & 23 & \fbox{23} \\ 
\hline 
\text{type} & A_1^{\oplus 4}
 &  A_1^{\oplus 8} \\
\hline
\end{array} 
\end{equation*} 
$n= 23,39$
\begin{equation*} 
\begin{array}{c|ccc} 
\hline
i & \fbox{23} & 23 & 23 \\ 
\hline 
\text{type}
 & A_1^{\oplus 2}
 & A_1^{\oplus 4}
 & A_1^{\oplus 4} \\
\hline
\end{array} 
\end{equation*} 
$n= 25,27,42,50,57,75$
\begin{equation*} 
\begin{array}{c|c} 
\hline
i & \fbox{23} \\ 
\hline 
\text{type} & A_1^{\oplus 4} \\
\hline
\end{array} 
\end{equation*} 
$n= 31$
\begin{equation*} 
\begin{array}{c|cccc} 
\hline
i & 19 & 19 & 22 & \fbox{23} \\ 
\hline 
\text{type} & { A_2} & { A_2}
 & { A_2}
 & { A_1} \\
\hline
\end{array} 
\end{equation*} 
$n= 32$
\begin{equation*} 
\begin{array}{c|ccccc} 
\hline
i & 19 & 20 & 20 & 22 & \fbox{23} \\ 
\hline 
\text{type} &  A_3
 & A_1^{\oplus 2}
 &  A_4
 &  A_1\oplus A_2
 & A_1^{\oplus 2} \\
\hline
\sharp\{v\in N^G\bigm| \BF{v}{v}=-4 \}
 & & 14 &   &  & 22 \\
\hline
\end{array} 
\end{equation*} 
$n= 34$
\begin{equation*} 
\begin{array}{c|ccccc} 
\hline
i & 19 & 19 & 21 & 21 & 21 \\ 
\hline 
\text{type}
 &  A_2^{\oplus 2}
 &  A_2\oplus A_3
 & A_1^{\oplus 2}
 &  A_1^{\oplus 2}\oplus A_3
 &  A_3 \\
\hline
\hline
i & 22 & 23 & 23 & 23 & \fbox{23} \\ 
\hline 
\text{type}
 &  A_2^{\oplus 2}
 & A_1^{\oplus 2}
 & A_1^{\oplus 2} 
 & A_1^{\oplus 3}
 & A_1^{\oplus 4} \\
\hline
\end{array} 
\end{equation*} 
$n= 35,51$
\begin{equation*} 
\begin{array}{c|cccccc} 
\hline
i & 21 & 21 & \fbox{23} & 23 & 23 \\ 
\hline 
\text{type}
 & A_1^{\oplus 2}
 & A_1^{\oplus 4}
 & A_1
 & A_1^{\oplus 2}
 & A_1^{\oplus 2} \\
\hline
\end{array} 
\end{equation*} 
$n= 38,54$
\begin{equation*} 
\begin{array}{c|ccc} 
\hline
i & 18 & 22 & \fbox{23} \\ 
\hline 
\text{type} & { A_2}
 & { A_2}
 & { A_1} \\
\hline
\end{array} 
\end{equation*} 
$n= 41$
\begin{equation*} 
\begin{array}{c|ccc} 
\hline
i & 23 & 23 & \fbox{23} \\ 
\hline 
\text{type} & A_1^{\oplus 2} & A_1^{\oplus 2}
 & A_1^{\oplus 2} \\
\hline
\sharp\{v\in N^G\bigm| \BF{v}{v}=-4 \}
 & 26 & 26 & 42 \\
\hline
\end{array} 
\end{equation*} 
$n= 46$
\begin{equation*} 
\begin{array}{c|ccc} 
\hline
i & 22 & 22 & \fbox{23} \\ 
\hline 
\text{type}
 & A_1^{\oplus 2} \oplus A_2
 & A_1 \oplus A_2^{\oplus 2}
 & A_1^{\oplus 3} \\
\hline
\end{array} 
\end{equation*} 
$n= 48$
\begin{equation*} 
\begin{array}{c|ccc} 
\hline
i & 19 & 22 & \fbox{23} \\ 
\hline 
\text{type} & { A_2}
 & { A_2}
 & { A_1} \\
\hline
\end{array} 
\end{equation*} 
$n= 49$
\begin{equation*} 
\begin{array}{c|ccc} 
\hline
i & 23 & 23 & \fbox{23} \\ 
\hline 
\text{type}
 & { A_1}
 & A_1^{\oplus 4}
 & A_1^{\oplus 5} \\
\hline
\end{array} 
\end{equation*} 
$n= 52,53,68,72,78$
\begin{equation*} 
\begin{array}{c|cc} 
\hline
i & 23 & \fbox{23} \\ 
\hline 
\text{type}
 & { A_1}
 &  A_1^{\oplus 2} \\
\hline
\end{array} 
\end{equation*} 
$n= 55$
\begin{equation*} 
\begin{array}{c|ccccc} 
\hline
i & 19 & 22 & 22 & 23 & \fbox{23} \\ 
\hline 
\text{type} & { D_4}
 & { A_2}
 &  A_2^{\oplus 2}
 &  A_1^{\oplus 3}
 &  A_1^{\oplus 4} \\
\hline
\end{array} 
\end{equation*} 
$n= 56$
\begin{equation*} 
\begin{array}{c|cc} 
\hline
i & 23 & \fbox{23} \\ 
\hline 
\text{type} & A_1^{\oplus 2}
 & A_1^{\oplus 2} \\
\hline
\sharp\{v\in N^G\bigm| \BF{v}{v}=-4 \}
 & 26 & 42 \\
\hline
\end{array} 
\end{equation*} 
$n=58$
\begin{equation*} 
\begin{array}{c|cc} 
\hline
i & 23 & 23 \\ 
\hline 
\text{type}
 &  A_1^{\oplus 2}
 &  A_1^{\oplus 2} \\
\hline
\end{array} 
\end{equation*} 
$n= 62$
\begin{equation*} 
\begin{array}{c|cc} 
\hline
i & 22 & \fbox{23} \\ 
\hline 
\text{type} & { A_2}
 & { A_1} \\
\hline
\end{array} 
\end{equation*} 
$n= 63$
\begin{equation*} 
\begin{array}{c|ccc} 
\hline
i & 22 & 22 & \fbox{23} \\ 
\hline 
\text{type} &  A_1^{\oplus 3}
 & A_1^{\oplus 2}\oplus  A_2
 & A_1^{\oplus 3} \\
\hline
\sharp\{v\in N^G\bigm| \BF{v}{v}=-6 \}
 & 14 & & 26 \\
\hline
\end{array} 
\end{equation*} 
$n= 64,73,81$
\begin{equation*} 
\begin{array}{c|cc} 
\hline
i & 23 & \fbox{23} \\ 
\hline 
\text{type} &  A_1^{\oplus 3}
 &  A_1^{\oplus 4} \\
\hline
\end{array} 
\end{equation*} 
$n= 65$
\begin{equation*} 
\begin{array}{c|cccc} 
\hline
i & 23 & 23 & 23 & \fbox{23} \\ 
\hline 
\text{type}
 & { A_1}
 &  A_1^{\oplus 2}
 &  A_1^{\oplus 3}
 &  A_1^{\oplus 4} \\
\hline
\end{array} 
\end{equation*} 
$n= 66,76$
\begin{equation*} 
\begin{array}{c|ccc} 
\hline
i & 23  & 23& \fbox{23} \\ 
\hline 
\text{type} & { A_1}
 & { A_1}
 &  A_1^{\oplus 2} \\
\hline
\end{array} 
\end{equation*} 
$n= 70$
\begin{equation*} 
\begin{array}{c|cccc} 
\hline
i & 19 & 22 & 23 & \fbox{23} \\ 
\hline 
\text{type} & { A_3}
 & { A_2}
 & { A_1}
 &  A_1^{\oplus 2} \\
\hline
\end{array} 
\end{equation*} 
$n= 74$
\begin{equation*} 
\begin{array}{c|ccc} 
\hline
i & 21 & 23 & \fbox{23} \\ 
\hline 
\text{type} & { A_3}
 &  A_1^{\oplus 2}
 &  A_1^{\oplus 3} \\
\hline
\end{array} 
\end{equation*} 
$n= 79$
\begin{equation*} 
\begin{array}{c|ccc} 
\hline
i & 22 & 23 & \fbox{23} \\ 
\hline 
\text{type} &  A_2^{\oplus 2}
 &  A_1^{\oplus 2}
 &  A_1^{\oplus 3} \\
\hline
\end{array} 
\end{equation*} 

\section*{Appendix: computations using GAP}

In this appendix, we briefly explain how to check
 Lemmas \ref{lem:maxconj}--\ref{lem:f384}
 using GAP \cite{GAP}.

For Lemma \ref{lem:maxconj},
 consider the case $\MF{H}=T_{192}$($=\MF{G}_{77}$).
By Table \ref{subsect:trees},
 it is sufficient to check that
 there exists only one conjugacy class of
 subgroups of $T_{192}$ which are isomorphic to
 $T_{24}=\MF{G}_{37}$
 (resp.\ $(Q_8*Q_8) \rtimes C_3=\MF{G}_{69}$).
GAP has the catalogue of all groups of small orders and
 the command \verb|SmallGroup|$(k,i)$ returns
 the $i$-th group of order $k$ in the catalogue (cf. \cite{beo}).
For example,
 \verb|SmallGroup|$(192,1493)$ returns $T_{192}$
 by Table \ref{subsect:qands}.
The command \verb|IsomorphicSubgroups|$(G,H)$ enumerates
 all conjugacy classes of subgroups of $G$
 which are isomorphic to $H$.
Hence we can check the assertion as follows.\footnote{
 A command terminated by two semicolons does not show the result.}

\begin{quote}
\begin{verbatim}
gap> h:=SmallGroup(192,1493);;
gap> g1:=SmallGroup(24,3);;
gap> g2:=SmallGroup(96,204);;
gap> Size( IsomorphicSubgroups( h , g1 ) );
1
gap> Size( IsomorphicSubgroups( h , g2 ) );
1
\end{verbatim}
\end{quote}

Here the command \verb|Size|$(a)$ returns
 the size of the object $a$.
The cases $\MF{H}=T_{48},H_{192},M_{20}$ are similar.

For Lemma \ref{lem:a44},
 we realize $G,\widetilde{G}$
 as quotient groups of
 subgroups of $\MF{S}_{36}$.
For example, the linear transformations
\begin{align*}
 (x,y,z,u,v,w)&\mapsto (e^{2\pi i/6}x,y,z,u,v,w), \\
 (x,y,z,u,v,w)&\mapsto (y,x,z,u,v,w)
\end{align*}
 correspond to
\begin{gather*}
 (1~ 2~ 3~ 4~ 5~ 6), \\
 (1~ 7)(2~ 8)(3~ 9)(4~ 10)(5~ 11)(6~ 12),
\end{gather*}
 respectively.
We can check Lemma \ref{lem:a44} as follows.

\begin{quote}
\begin{verbatim}
a1:=(1,2,3,4,5,6);
a2:=(7,8,9,10,11,12);
a3:=(13,14,15,16,17,18);
a4:=(19,20,21,22,23,24);
a5:=(25,26,27,28,29,30);
a6:=(31,32,33,34,35,36);
a123456:=a1*a2*a3*a4*a5*a6;

b123:=(1,7,13)(2,8,14)(3,9,15)(4,10,16)(5,11,17)*
(6,12,18);
b456:=(19,25,31)(20,26,32)(21,27,33)(22,28,34)*
(23,29,35)(24,30,36);
b23:=(7,13)(8,14)(9,15)(10,16)(11,17)(12,18);
b56:=(25,31)(26,32)(27,33)(28,34)(29,35)(30,36);
b14:=(1,19)(2,20)(3,21)(4,22)(5,23)(6,24);
b2536:=(7,25,13,31)(8,26,14,32)(9,27,15,33)*
(10,28,16,34)(11,29,17,35)(12,30,18,36);

g0:=Group(
 a1^3*a2^3,
 a4^3*a5^3,
 a5^2*a6^4*b123,
 a2^4*a3^2*b456,
 a1^3*a2^3*a3^3*b23*b56,
 a123456
);
gg0:=ClosureGroup(gg0, b14*b2536 );
n:=Group(a123456);
f:=NaturalHomomorphismByNormalSubgroup(gg0,n);
g:=Image(f,g0);
gg:=Image(f);

list:=[[48,49],[48,30],[96,195],[144,184]];
for nn in list do
 subgrps:=IsomorphicSubgroups(gg,SmallGroup(nn));
 subgrps:=Filtered(subgrps,x->IsSubgroup(g,Image(x)));
 Display(Size(subgrps));
od;
\end{verbatim}
\end{quote}

Here e.g.\ \verb|a1^3*a2^3| corresponds to
 the transformation (\ref{eq:minusxy}).
The quotient groups \verb|g|,\,\verb|gg|
 by the group \verb|n|,
 which corresponds to
 the subgroup of homothetic transformations,
 are $G,\widetilde{G}$,
 respectively.
By Table \ref{subsect:trees},
 it is sufficient to check that
 there exists only one conjugacy class of
 subgroups of $\widetilde{G}$ which are isomorphic to
 $\MF{G}_n$ and contained in $G$ for $n=52,53,68,72$.
This is done in the last paragraph of the above program.
The result is the following.

\begin{quote}
\begin{verbatim}
1
1
1
1
\end{verbatim}
\end{quote}

Thus Lemma \ref{lem:a44} has been checked.
Lemma \ref{lem:f384} is similarly checked.

\end{document}